\documentclass[12pt]{article}

\usepackage[left=2cm,right=2cm,top=2cm,bottom=2cm]{geometry}
\usepackage[latin1]{inputenc}
\usepackage{cite}

\usepackage{amsmath,amssymb,latexsym,amsthm}

\usepackage[pdftex,pdfpagelabels,bookmarks,hyperindex,hyperfigures]{hyperref}

\newtheorem{theorem}{Theorem}
\newtheorem{thm}[theorem]{Theorem}
\newtheorem{lem}[theorem]{Lemma}
\newtheorem{cor}[theorem]{Corollary}

\theoremstyle{definition}
\newtheorem{definition}{Definition}

\makeatletter
\newcommand{\subjclass}[2][2010]{%
	\let\@oldtitle\@title%
	\gdef\@title{\@oldtitle\footnotetext{#1 \emph{MSC2010: } #2}}%
}
\newcommand{\keywords}[1]{%
	\let\@@oldtitle\@title%
	\gdef\@title{\@@oldtitle\footnotetext{\emph{Keywords: } #1.}}%
}
\makeatother

\title{An alternative foundation and the generalized continuum hypothesis}


\author{Eddy El Khalil \\ \href{mailto:el\_khalil.eddy@courrier.uqam.ca}{el\_khalil.eddy@courrier.uqam.ca} \\ Universit\'{e} du Qu\'{e}bec \`{a} Montr\'{e}al} 
\date{}

\begin{document}

\maketitle

\begin{abstract}
In this paper I introduce a new and intuitive first-order foundational theory (where the concept of set is not primitive) and use it to show that the power set of an infinite set does not exist. In particular, proofs of uncountability of a set are essentially proofs of the non-existence of that specified set. In a certain sense, uncountability is shown to be a form of incompleteness. Also, the axiom of choice is shown to be a straightforward theorem. In view of the non-existence of a set of all real numbers or more generally the non-guaranteed existence of the completion of a metric space, topological concepts are re-introduced in the context of ``extensions''. Measure theory is also reformulated accordingly.
\end{abstract}

\section{Introduction}
The generalized continuum hypothesis roughly says that there exists no set of cardinality strictly between the cardinality of a given infinite set and the cardinality of its power set. In light of the independence of the generalized continuum hypothesis from ZFC (Cohen\cite{cohen1}\cite{cohen2} and G\"odel\cite{godel}), it is safe to say that redefining the concept of set would be a sensible approach to resolve the generalized continuum hypothesis. The main approach that is used to settle the generalized continuum hypothesis consists in adding axioms directly to ZFC (see Koellner\cite{sep-continuum-hypothesis}); the approach used in this paper consists of ``decomposing'' the concept of set into somewhat more primitive notions. I contend that these ``more primitive notions'' should reflect the fact that a formalized language will be used to express the foundations of mathematics; in a certain sense, the study of mathematical objects will be bootstrapped upon this self-reflection. In ZFC, ``representation of mathematical objects'' and ``meaning of mathematical objects'' are both expressed by the same axiomatization via the concept of set. I claim that foundational issues concerning cardinality can be solved if \textit{representation} and \textit{meaning} are distinguished through two distinct primitive domains (from the point of view of first-order logic) and if the concept of set is expressed using the notions of \textit{representation} and \textit{meaning}. The resulting axiomatic system will resemble the one used in second-order arithmetic (as introduced in Hilbert/Bernays\cite{hilbert2013grundlagen1}\cite{hilbert2013grundlagen2} and as used in Simpson\cite{simpson2009subsystems} for reverse mathematics) having concatenation (characterized closely to what can be found in Tarski\cite{Tarski1936}, Quine\cite{quine1946concatenation}, Corcoron/Frank/Maloney\cite{corcoran1974string} and Grzegorczyk\cite{Grzegorczyk_2005}) as a replacement for arithmetic.

In context of ZFC, the foundation I introduce can be seen as an argument against the adoption of the power set axiom. I take the position that it is more natural to assume that ``every idea can be finitely represented'' than to assume that the power set of any set must exist. I assume that ideas can be expressed using sets and that sets and their elements can be represented by finite binary words (which will serve as medium of \textit{representation} and stand as primitive individuals in the foundation introduced). So using the countability of finite binary words and Cantor's diagonal argument I conclude that the power set of an infinite set cannot exist. Essentially, Cantor's diagonal argument will be used to construct a new element by giving ``new meaning'' to a chosen word; to carry this reasoning out, \textit{meaning} will be tracked and developed using ``systems'' (which will stand as primitive individuals). Moreover, for any set of subsets of some infinite set there exists a set of subsets (of the same infinite set) which can be obtained by ``adding'' a new subset; similar to the fact that given any set of ordinals it is possible to ``add'' a new ordinal to that set. In a certain sense, this perspective provides a natural resolution of Skolem's paradox. Skolem's paradox roughly asks how can a consistent theory that proves the existence of an uncountable set have a countable model (see Skolem\cite{skolem}). If we step back to a foundational point of view, instead of explaining away Skolem's paradox as a ``mathematics vs. metamathematics'' issue, I contend that Skolem's paradox should be seen as motivation to separate \textit{representation} from \textit{meaning} at the foundational level in order to more accurately capture the nature of metamathematics. Essentially, this allows us to highlight the finite nature of \textit{representation}, which is usually assumed at the metamathematical layer (i.e., first-order formulas are finite expressions), while still allowing the potentially infinite nature of \textit{meaning} to be expressed.

The paper will be structured as follows. In section 2, I introduce a new first-order foundation based on ``words'' and ``systems'', and state some basic definitions and consequences. In section 3, the concept of set and other commonly used structures are defined. In section 4, we see that every set is countable and that the axiom of choice becomes a simple theorem. In section 5 it is shown that the power set of an infinite set does not exist and that in particular a set containing all real numbers cannot exist. In a certain sense, uncountability is shown to be a form of incompleteness. In section 6, topological notions and classical results are reformulated through the lens of ``extensions''. In section 7, measure-theoretic notions and integrability are reformulated.

\section{Words and systems}

The main idea is to assume that every human thought involves exclusively \textit{representation} and \textit{meaning}. \textit{Representation} cannot be understood without \textit{meaning}, and \textit{meaning} cannot be understood without \textit{representation}. It is then assumed that it is necessary and sufficient for a foundational theory to standardize the behavior of \textit{representation} and \textit{meaning}. I claim that the following theory adequately standardizes the interaction of \textit{representation} and \textit{meaning}.

Lets consider a two-sorted first-order theory (with equality) having \texttt{\textsl{words}} and \texttt{\textsl{systems}} as primitive individuals. \texttt{\textsl{Words}} will be denoted by indexed or non-indexed lower-case letters ($a,b,\dots , $ $z, a_0,\dots ,z_0, \dots ,a_n, \dots ,z_n, \dots $). \texttt{\textsl{Systems}} will be denoted by indexed or non-indexed upper-case letters ($A,B,\dots , Z$,$A_0,\dots ,Z_0, \dots ,A_n, \dots ,Z_n, \dots $). There will be two \texttt{\textsl{word}} constants denoted by 0 and 1, one binary relation (seen as a membership relation between \texttt{\textsl{words}} and \texttt{\textsl{systems}}) expressed using $\in$ (first argument is a \texttt{\textsl{word}} and second argument is a \texttt{\textsl{system}}) and one binary function (seen as a concatenation operator on \texttt{\textsl{words}}) expressed by juxtaposing \texttt{\textsl{words}}. The following axioms and axiom schema will be satisfied:
\begin{enumerate}
	\item (symbols) $0\neq 1 \wedge \forall x \forall y ( xy\neq 0 \wedge xy\neq 1 )$
	\item (associativity) $\forall x \forall y \forall z ( (xy)z=x(yz) )$
	\item (reading) $\forall x \forall y ( x0 \neq y1 )$
	\item (simplification) $\forall x_1 \forall x_2 \forall y_1 \forall y_2 ((x_1y_1 = x_2y_2 \wedge y_1 = y_2)\Rightarrow x_1=x_2)$
	\item (extensionality) $\forall S_1 \forall S_2 ( \forall x( x\in S_1 \Leftrightarrow x\in S_2 ) \Rightarrow S_1 = S_2) $
	\item (word induction) $\forall S ( (0\in S \wedge 1\in S \wedge \forall x( x\in S \Rightarrow ( x0\in S \wedge x1\in S ) ) ) \Rightarrow \forall y( y\in S ) )$
	\item (comprehension) $\forall S_1,\dots , S_m \forall x_1,\dots , x_n \exists S \forall x ( x\in S \Leftrightarrow \phi ( x, x_1,\dots ,x_n, S_1,\dots ,S_m ) )$ where $S$ is not a free variable in $\phi$ . 
\end{enumerate}

Essentially, words will be used as medium of representation and systems will be used to express meaning. In a certain sense, systems can be viewed as formal languages, and therefore we could say that meaning is conveyed using formal languages.

Informally and considered in context of each other, the axioms and axiom schema can be roughly described as follows. 
The (symbols) axiom says that the word constants $0$ and $1$ are distinct and ``indivisible'' words. In a certain sense $0$ and $1$ can be seen as letters.
The (associativity) axiom states that concatenation is associative.
The (reading) axiom says that if two words are the same they must end by the same letter.
The (simplification) axiom says that  if two words are the same 
and end by some same word then they begin by some same word.
The (extensionality) axiom says that if two systems contain the same words then these systems must be the same.
Using (comprehension), the (word induction) axiom shows that every word is a finite string of ``zeros'' and ``ones''.
The (comprehension) axiom schema says that any ``adequate'' property determines a system.

The (associativity) axiom will often be used implicitly. 
We now explore some basic definitions and consequences.

%   System relation and operations
\begin{definition}
	Let $S_1$, $S_2$ and $S_3$ be systems and $x$ be a word.
	\begin{itemize}
		\item If $x\in S_1$ then we say that $x$ is an {\bf {\itshape element of $S_1$}} or that {\bf {\itshape $S_1$ contains $x$}}.
		\item If for any word $x$ we have $x\in S_1 \Rightarrow x\in S_2$ then we say that $S_1$ is a {\bf {\itshape subsystem of}} $S_2$ or that $S_1$ is {\bf {\itshape smaller than }} $S_2$ or write $S_1 \subseteq S_2$.
		\item If $\forall x( x\in S_3 \Leftrightarrow (x\in S_1 \vee x\in S_2))$ then we say that $S_3$ is the {\bf {\itshape union of}} $S_1$ {\bf {\itshape and }} $S_2$ or write $S_3=S_1 \cup S_2$.
		\item If $\forall x(x\in S_3 \Leftrightarrow (x\in S_1 \wedge x\in S_2))$ then we say that $S_3$ is the {\bf {\itshape intersection of}} $S_1$ {\bf {\itshape and}} $S_2$ or write $S_3=S_1 \cap S_2$.
		\item If $\forall x(x\in S_3 \Leftrightarrow (x\in S_1 \wedge x\notin S_2))$ then we say that $S_3$ is the {\bf {\itshape system difference of}} $S_1$ {\bf {\itshape by}} $S_2$ or write $S_3=S_1 \setminus S_2$.
		\item If $\forall x(x\notin S_1)$ then we say that $S_1$ is the {\bf {\itshape empty system}} or write $S_1=\emptyset$.
	\end{itemize}
\end{definition}

% Smallest system

% Universal system
%\begin{definition}
%Given a system $S$, if every word is in $S$ then we say that $S$ is the {\bf {\itshape universal system}}.
%\end{definition}

\begin{definition}
Let $x$ and $y$ be words. If $y=x$ or there exists $z$ such that $y=xz$ then we say that $x$ is a {\bf {\itshape prefix }} of $y$. If $y=x$ or there exists $z$ such that $y=zx$ then we say that $x$ is a {\bf {\itshape suffix}} of $y$.
\end{definition}

% Prefix + suffix transitivity
\begin{lem}
Let $x$, $y$ and $z$ be words.
\begin{enumerate} 
	\item If $x$ is a prefix of $y$ and $y$ is a prefix of $z$ then $x$ is a prefix of $z$.
	\item If $x$ is a suffix of $y$ and $y$ is a suffix of $z$ then $x$ is a suffix of $z$.
\end{enumerate}
\end{lem}

% When using associativity we drop parentheses

\begin{proof}
1. If $x=y$ or $y=z$ then trivially $x$ is a prefix of $z$. If $x\neq y$ and $y\neq z$ then there exist $x_1$ and $y_1$ such that $y=xx_1$ and $z=yy_1$ so by (associativity) $z=x(x_1y_1)$ showing that $x$ is a prefix of $z$.

2. If $x=y$ or $y=z$ then trivially $x$ is a suffix of $z$. If $x\neq y$ and $y\neq z$ then there exist $x_1$ and $y_1$ such that $y=x_1x$ and $z=y_1y$ so by (associativity) $z=(y_1x_1)x$ showing that $x$ is a suffix of $z$.

\end{proof}

% 1x 0x lemma
\begin{lem} \label{prepost} Let $x$ be a word such that $x\neq 0$ and $x\neq 1$.
\begin{enumerate}
	\item There exists $x_0$ such that $x=x_00$ or $x=x_01$.
	\item There exists $x_0$ such that $x=0x_0$ or $x=1x_0$.
\end{enumerate}
\end{lem}

\begin{proof}
1. Let $S_1$ be the smallest system such that 
$ 0\in S_1 \wedge 1\in S_1 \wedge \forall x(x\in S_1 \Rightarrow (x0\in S_1 \wedge x1\in S_1 ))$. 
$S_1$ exists by (comprehension) and we see by (word induction) that $S_1$ contains every word. Suppose there exists $x_1\in S_1$ with $x_1\neq 0$ and $x_1\neq 1$ such that for all $y$ we have $x_1\neq y0$ and $x_1\neq y1$. Using (comprehension) we could construct a system $S_2$ such that $\forall x(x\in S_2 \Leftrightarrow (x\in S_1 \wedge x\neq x_1))$ and $S_2$ would still contain every word by (word induction) which gives us a contradiction.

2. Let $S$ be the smallest system such that 
$ \forall x( x0\in S \wedge x1\in S) $.
By 1. we see that if $y\neq 0$ and $y\neq 1$ then $y\in S$. If $y=00$ or $y=01$ or $y=10$ or $y=11$ we easily see that there exists $x_0$ such that $y=0x_0$ or $y=1x_0$.

Suppose that for some $k$ with $k\neq 0$ and $k\neq 1$ there exists $x_0$ such that $k=0x_0$ or $k=1x_0$. By (associativity) we see that 
$$ k0 = (0x_0)0 = 0(x_00) \text{ or } k0 = (1x_0)0 = 1(x_00) $$
and
$$ k1 = (0x_0)1 = 0(x_01) \text{ or } k1 = (1x_0)1 = 1(x_01) $$
By (word induction) we conclude that $\forall y(y\in S \Rightarrow \exists x_0(y=0x_0 \vee y=1x_0) )$.
\end{proof}

The previous lemma will often be used implicitly.

% Word theorems
The following theorem confirms the ``left'' version of (reading).
% Left reading
\begin{thm} \label{leftreading}
$\forall x\forall y (1x\neq 0y)$
\end{thm}
\begin{proof}
$$ 10 \neq 00 \text{ by (simplification) and (symbols) }  $$
$$ 10 \neq 01 \text{ by (reading) } $$
$$ 11 \neq 01 \text{ by (simplification) and (symbols) } $$
$$ 11 \neq 00 \text{ by (reading) } $$

If $y=y_00$ or $y=y_01$ then we obtain $10\neq 0y$ and $ 11\neq 0y$ by (simplification), (reading) and (symbols). So we must have $$ \forall y (10 \neq 0y \wedge 11 \neq 0y) $$

Suppose for a given $x_0$ we have $\forall y (1x_0 \neq 0y)$. If $y=0$ then $1x_00 \neq 0y$ by (simplification) and (symbols) and $1x_01 \neq 0y$ by (reading). If $y=1$ then $1x_00 \neq 0y$ by (reading) and $1x_01 \neq 0y$ by (simplification) and (symbols). If $y=y_00$ for some $y_0$ then $1x_01 \neq 0y$ by (reading) and $1x_00 \neq 0y$ by (simplification) and hypothesis.
If $y=y_01$ for some $y_0$ then $1x_01 \neq 0y$ by (simplification) and hypothesis and $1 x_0 0 \neq 0y$ by (reading). So by lemma \ref{prepost}, (comprehension) and (word induction) we get $\forall x \forall y (1x \neq 0y)$. 

\end{proof}

The next theorem confirms the ``left'' version of (simplification).
% left simplification
\begin{thm} \label{leftsimpli}
$\forall x_1 \forall x_2 \forall y_1 \forall y_2 ((x_1y_1 = x_2y_2 \wedge x_1 = x_2)\Rightarrow y_1=y_2)$
\end{thm}

\begin{proof}
$$ \forall y_2 (00 = 0y_2 \Rightarrow y_2=0) \text{ by (reading), (simplification), (symbols) and lemma \ref{prepost} } $$
$$ \forall y_2 (01 = 0y_2 \Rightarrow y_2=1) \text{ by (reading), (simplification), (symbols) and lemma \ref{prepost} } $$ 
$$ \forall y_2 (10 = 1y_2 \Rightarrow y_2=0) \text{ by (reading), (simplification), (symbols) and lemma \ref{prepost} } $$
$$ \forall y_2 (11 = 1y_2 \Rightarrow y_2=1) \text{ by (reading), (simplification), (symbols) and lemma \ref{prepost} } $$

Suppose for a given $y_1$ we have $\forall y_2( (0y_1 = 0y_2 \Rightarrow y_1=y_2 ) \wedge (1y_1 = 1y_2 \Rightarrow y_1=y_2)) $. For some $y_2$, if $0y_10 = 0y_2$ then $\exists y_0 (y_2=y_00)$ ( by (reading), (simplification), (symbols) and  lemma \ref{prepost} ). So by (simplification) and hypothesis we get $$(0y_10 = 0y_2) \Rightarrow (0y_1 = 0y_0) \Rightarrow (y_1=y_0) \Rightarrow (y_10 = y_2)$$ 
For some $y_2$, if $0y_11 = 0y_2$ then $\exists y_0 (y_2=y_01)$ ( by (reading), (simplification), (symbols) and lemma \ref{prepost} ). So $$(0y_11 = 0y_2) \Rightarrow (0y_1 = 0y_0) \Rightarrow (y_1=y_0) \Rightarrow (y_11 = y_2)$$
For some $y_2$, if $1y_10 = 1y_2$ then $\exists y_0 (y_2=y_00)$ ( by (reading), (simplification), (symbols) and lemma \ref{prepost} ). So $$(1y_10 = 1y_2) \Rightarrow (1y_1 = 1y_0) \Rightarrow (y_1=y_0) \Rightarrow (y_10 = y_2)$$
For some $y_2$, if $1y_11 = 1y_2$ then $\exists y_0 (y_2=y_01)$ ( by (reading), (simplification), (symbols) and lemma \ref{prepost} ). So $$(1y_11 = 1y_2) \Rightarrow (1y_1 = 1y_0) \Rightarrow (y_1=y_0) \Rightarrow (y_11 = y_2)$$
By (comprehension) and (word induction) we obtain $\forall y_1 \forall y_2((0y_1 = 0y_2 \Rightarrow y_1=y_2) \wedge (1y_1 = 1y_2 \Rightarrow y_1 = y_2))$.

Now suppose that for some $k$ we have $\forall x_1 \forall y_1 \forall y_2 ( (x_1y_1=ky_2 \wedge x_1=k) \Rightarrow y_1=y_2 )$. If $x_1y_1=k0y_2$ and $x_1=k0$ then by (associativity) and hypothesis we obtain $0y_1=0y_2$ so by what was shown above we get $y_1=y_2$. Similarly, if $x_1y_1=k1y_2$ and $x_1=k1$ then we get $y_1=y_2$. Using (comprehension) and (word induction) we obtain the desired result. 

\end{proof}

% integrity 1
% A1 + A3
\begin{thm} \label{integri}
$\forall x \forall y (x\neq xy \wedge x\neq yx)$
\end{thm}

\begin{proof}
We can easily see that $\forall y(0\neq 0y \wedge 0\neq y0 \wedge 1 \neq 1y \wedge 1\neq y1)$ by (symbols), so if $x=0$ or $x=1$ then we get $\forall y (x\neq xy \wedge x\neq yx)$. Let $x$ be such that $x\neq 0$ and $x\neq 1$, by lemma \ref{prepost} we get $\exists x_0(x=x_00 \vee x=x_01)$ and $\exists x_1(x=0x_1 \vee x=1x_1)$. Suppose for some $y$ we have $x=xy$. We get $x_00 = x_00y$ or $x_01 = x_01y$ for some $x_0$ and by left simplification (theorem \ref{leftsimpli}) we obtain $0=0y$ or $1=1y$ which contradicts (symbols). Suppose for some $y$ we have $x=yx$, then we get $0x_1 = y0x_1$ or $1x_1 = y1x_1$ for some $x_1$ and by (simplification) we obtain $0=y0$ or $1=y1$ which contradicts (symbols). So we must conclude that $\forall x \forall y (x\neq xy \wedge x\neq yx)$.
\end{proof}

% integrity 2
% A2
\begin{thm} \label{integri2}
$\forall x \forall y \forall z(x\neq yxz)$.
\end{thm}

\begin{proof}
By (symbols) we see that $0\neq y0z$ and $1\neq y1z$. Suppose for some word $k$ we have that $k\neq ykz$ for all words $y$ and $z$. Lets proceed by contradiction and assume that $k0=y(k0)z$ for some word $y$ and some word $z$. By (reading) and lemma \ref{prepost} we get $z=z_10$ for some $z_1$ or $z=0$. If $z=z_10$ then by (associativity) and (simplification) we get $k=yk0z_1$ which produces a contradiction. If $z=0$ then by (simplification) we get $k=yk0$ giving us also a contradiction. So we must have $k0\neq y(k0)z$. By similar logic we must have $k1\neq y(k1)z$, so using (comprehension) and (word induction) we get the desired result.
\end{proof}

% Definition of right-generated and left-generated by system
% Definition of right/left generated
\begin{definition}
Let $S_1$ and $S_2$ be systems and $w$ a word.
	\begin{itemize}
		\item If $S_2$ is the smallest system such that
				$$ w\in S_2 \wedge \forall x( x\in S_2 \Rightarrow xw\in S_2) $$ then we say that $S_2$ is {\bf {\itshape right-generated by}} $w$.
		\item If $S_2$ is the smallest system such that 
				$$ \forall x\in S_1(x\in S_2) \wedge \forall x_1\in S_1 \forall x_2\in S_2( x_2x_1\in S_2) $$
				then we say that $S_2$ is {\bf {\itshape right-generated by}} $S_1$.
		\item If $S_2$ is the smallest system such that
				$$ w\in S_2 \wedge \forall x( x\in S_2 \Rightarrow wx\in S_2) $$ then we say that $S_2$ is {\bf {\itshape left-generated by}} $w$.
		\item If $S_2$ is the smallest system such that 
				$$ \forall x\in S_1(x\in S_2) \wedge \forall x_1\in S_1 \forall x_2\in S_2( x_1x_2\in S_2) $$
				then we say that $S_2$ is {\bf {\itshape left-generated by}} $S_1$.
	\end{itemize}
\end{definition}
% Definition of right-generated and left-generated by word

% Definition of prefix-generated
\begin{definition}
Given systems $S$, $S_1$ and $S_2$, if $S_1$ is right-generated by $S$ and $S_2$ is the smallest system containing every prefix of every word in $S_1$ then we say that $S_2$ is {\bf {\itshape prefix-generated by }}$S$.
\end{definition}

% prefix-generated and right-generated (prefix induction)
\begin{lem} \label{prefge}
Given systems $S_1$, $S_2$ and $S_3$, if $S_2$ is right-generated by $S_1$ and $S_3$ is prefix-generated by $S_1$ then $S_2$ is a subsystem of $S_3$.
\end{lem}

\begin{proof}
Since every word is a prefix of itself it follows immediately from definition of $S_3$ that $S_2\subseteq S_3$.
\end{proof}

% (C4) "Prefix of word increment"(C4=prefix of word increment)
\begin{lem} \label{overlap}
Given words $x$, $y_1$ and $y_2$, if $y_1y_2=x0$ or $y_1y_2=x1$ then $y_1$ is a prefix of $x$.
\end{lem}

\begin{proof}
If $y_2=0$ then by (reading) and (simplification) we must have $y_1=x$ so $y_1$ would be a prefix of $x$. If $y_2=1$ then similarly $y_1$ would be a prefix of $x$.

If $y_2\neq 0$ and $y_2\neq 1$ then by lemma \ref{prepost} there exists $y_3$ such that $y_2=y_30$ or $y_2=y_31$. So by (associativity), (reading) and (simplification) we must have $y_1y_3=x$ which implies that $y_1$ is a prefix of $x$.
\end{proof}

More generally we have the following theorem.

% lemm 0.6.1 and lem 0.6.2 (A13 - A14 part 2)(A13-A14(lem0.6.1-lem0.6.2) prefix)
\begin{thm} \label{overlap2}
Given words $x_1$, $x_2$, $y_1$ and $y_2$, if $x_1x_2=y_1y_2$ then $x_1$ is a prefix of $y_1$ or $y_1$ is a prefix of $x_1$.
\end{thm}

\begin{proof}
If $y_2=0$ or $y_2=1$ and $x_1x_2=y_1y_2$ then by lemma \ref{overlap} it follows that $x_1$ is a prefix of $y_1$. Suppose that for some $k$ we have $$\forall x_1 \forall x_2 \forall y_1 ( x_1x_2=y_1k \Rightarrow ( \text{ $x_1$ is a prefix of $y_1$ or $y_1$ is a prefix of $x_1$ } ) )$$
If $x_1x_2=y_1k0$ then by (reading), (associativity) and (simplification) we either get $x_1=y_1k$ (which implies that $y_1$ is a prefix of $x_1$) or $x_1x_3=y_1k$ with $x_2=x_30$. So by hypothesis $x_1$ is a prefix of $y_1$ or $y_1$ is a prefix of $x_1$. Similarly if $x_1x_2=y_1k_11$ we see that $x_1$ is a prefix of $y_1$ or $y_1$ is a prefix of $x_1$. By (comprehension) and (word induction) we obtain the desired result.
\end{proof}

% lemm 0.6.1 and lem 0.6.2 (A13 - A14 part 1)(A13-A14(lem0.6.1-lem0.6.2) suffix)
\begin{thm} \label{overlas}
Given words $x_1$, $x_2$, $y_1$ and $y_2$, if $x_1x_2=y_1y_2$ then $x_2$ is a suffix of $y_2$ or $y_2$ is a suffix of $x_2$
\end{thm}

\begin{proof}
If $y_2=0$ or $y_2=1$ and $x_1x_2=y_1y_2$ then by (reading) and lemma \ref{prepost} it follows that $y_2$ is a suffix of $x_2$.
Suppose that for some $k$ we have $$\forall x_1 \forall x_2 \forall y_1 ( x_1x_2=y_1k \Rightarrow ( \text{$x_2$ is a suffix of $k$ or $k$ is a suffix of $x_2$ } ) )$$ If $x_1x_2=y_1(k0)$ then by (reading), (associativity) and (simplification) we either get $x_1=y_1k$ (which implies that $x_2$ is a suffix of $k0$) or $x_1x_3=y_1k$ with $x_2=x_30$ for some $x_3$. By hypothesis $x_3$ is suffix of $k$ or $k$ is a suffix of $x_3$ which implies that $x_2$ is a suffix of $k0$ or $k0$ is a suffix of $x_2$. Similarly if $x_1x_2=y_1k1$ we see that $x_2$ is a suffix of $k1$ or $k1$ is a suffix of $x_2$. By (comprehension) and (word induction) we obtain the desired result.
\end{proof}

% Definition of minimal prefix
\begin{definition}
Let $S$ be a system and $x$ a word in $S$. If there exists $x_1\in S$ such that $$\forall y\in S ( \text{ $y$ is a prefix of $x$ } \Rightarrow \text{ $x_1$ is a prefix of $y$ } )$$
then we say that $x_1$ is a {\bf {\itshape minimal prefix of $x$ in $S$}}.
\end{definition}

% Existence of minimal prefix
% Dependencies : (C4) (Prefix transitivity) (Prefix induction) (symbols)
\begin{thm} \label{existsminpref}
Given word $x$ and system $S$, if $x\in S$ then there exists $x_1$ such that $x_1$ is a minimal prefix of $x$ in $S$.
\end{thm}

\begin{proof}
Let $S_1$ be the system containing all prefixes of $x$ which are in $S$. Suppose $S_1$ contains no minimal prefix of $x$. $0$ and $1$ would not be in $S_1$ otherwise one of them would be a minimal prefix by (associativity), lemma \ref{prepost} and theorem \ref{leftreading}. Now assume that $S_1$ contains none of the prefixes of some word $k$. If $S_1$ contains some prefix of $k0$ or $k1$ then by definition of prefix, lemma \ref{overlap} and by hypothesis, $S_1$ would contain either $k0$ or $k1$. But by lemma \ref{overlap}, theorem \ref{overlap2}, prefix transitivity and hypothesis, if $S_1$ contains either $k0$ or $k1$ then either $k0$ or $k1$ would be a minimal prefix of $x$ in $S_1$. So both $k0$ and $k1$ are not in $S_1$. By lemma \ref{prefge} we conclude that $S_1$ is empty which produces a contradiction since $x$ is in $S_1$. So $S$ must contain a minimal prefix of $x$.
\end{proof}

% Counting numbers and Natural numbers
%Counting number definition ( c is a counting number over b )
\begin{definition}
Let $S$ be right-generated by some word $b$. We say that a word $c$ is a {\bf {\itshape counting number over }} $b$ if $c\in S$. 
\end{definition}

If $c_1$, $c_2$ and $c_3$ are counting numbers over some same word $b$ and $c_1c_2=c_3$ then we will usually write $c_1 + c_2 = c_3$.

% Definition of natural numbers and addition
\begin{definition}
We say that $n$ is a {\bf {\itshape natural number}} if $n$ is a counting number over $1$.
\end{definition}

% Definition of natural number system
\begin{definition}
Given a system $S$, if $\forall n( n\in S \Leftrightarrow n \text{ is a natural number } )$ then we say that $S$ is the {\bf {\itshape natural number system}} or write $S=\mathbb{N}$.
\end{definition}

Whenever some word $n$ is a natural number we will just write $n\in \mathbb{N}$.

% less than or equal for natural numbers

For any natural numbers $n_1$ and $n_2$, if there exists a natural number $n_3$ such that $n_1+n_3=n_2$ then we write $n_1<n_2$. If $n_1<n_2$ or $n_1=n_2$ then we write $n_1\leq n_2$.

% Parsing lemma

\begin{lem} \label{wordnat}
Given
\begin{itemize}
	\item words $x_1$ and $x_2$
	\item natural numbers $n_1$ and $n_2$
\end{itemize}
if $x_10n_1=x_20n_2$ then $x_1=x_2$ and $n_1=n_2$.
\end{lem}

\begin{proof}
If either $n_1=1$ or $n_2=1$ then by (readability) and (simplification) we must have $n_1=1$ and $n_2=1$ which in turn implies that $x_1=x_2$ by (simplification).

Lets proceed by induction on $n_1$.
Suppose for some natural number $k$ we have $\forall x_1 \forall x_2 \forall n_2(x_10k=x_20n_2 \Rightarrow (k=n_2 \wedge x_1=x_2))$. Now if $x_10k1=x_20n_2$ then by (simplification) and by (readability) we must obtain $x_10k=x_20n_3$ where $n_2=n_3 +1$. By hypothesis we get $k=n_3$ which in turn implies $k1=n_2$ and we also get $x_1=x_2$. By (comprehension) and by the inductive definition of $\mathbb{N}$ we obtain the desired result.
\end{proof}

% Counting number association to natural number
% Dependencies :	
% Definition of counting number association

\begin{definition}
Given 
\begin{itemize}
	\item counting number $c$ over $b$
	\item natural number $n$
\end{itemize}
Let $S$ be the smallest system such that 
$$ b01\in S \wedge \forall x\forall m( m\in \mathbb{N} \Rightarrow ( x0m\in S \Rightarrow xb0m1\in S )) $$ If $c0n\in S$ then we say that $n$ is the {\bf {\itshape natural number associated to $c$ over $b$}} or write $n \otimes b = c$.  
\end{definition}

Using lemma \ref{wordnat} we see that if $n_1 \otimes b = w$ and $n_2 \otimes b = w$ then $n_1=n_2$.

If $n_1$, $n_2$ and $n_3$ are natural numbers with $n_1\otimes n_2=n_3$ we will say that $n_3$ is the {\bf {\itshape natural product}} of $n_1$ and $n_2$ or write $n_1\cdot n_2=n_3$.

% Well ordering of natural numbers
% Use minimal prefix theorem
\begin{thm} \label{existsminnat}
If every element of a non-empty system $S$ is a natural number then $S$ contains a minimal natural number.
\end{thm}

\begin{proof}
Since every natural number is a counting number over $1$ then for any given natural numbers $n_1$ and $n_2$, either $n_1$ is a prefix of $n_2$ or $n_2$ is a prefix of $n_1$. Using theorem \ref{existsminpref} (existence of minimal prefix) we see that $S$ must contain a minimal natural number.
\end{proof}

% Counting numbers and Peano properties

% Theorem on counting number properties
The following theorem highlights some expected properties of counting numbers and natural numbers.
\begin{thm}
Given
\begin{itemize}
	\item word $x$
	\item natural numbers $n_1$, $n_2$ and $n_3$
\end{itemize}
we have
\begin{enumerate}
	% x-neutral
	\item $1\otimes x=x$	\text{ (x-neutral)}
	% commutativity of addition
	\item $n_1 + n_2 = n_2 + n_1$	\text{ (commutativity of addition)}
	% x-additivity
	\item $(n_1\otimes x)+(n_2\otimes x)=(n_1+n_2)\otimes x$	\text{ (x-additivity)}
	% x-associativity
	\item $n_1\otimes (n_2\otimes x)=(n_1\otimes n_2)\otimes x$	\text{ (x-associativity)}
	% x-distributivity
	\item $n_1\otimes ((n_2+n_3)\otimes x)=((n_1\otimes n_2)\otimes x)+((n_1\otimes n_3)\otimes x)$	\text{ (x-distributivity) }
	% distributivity ( use x-distr + x-neutral )
	\item $n_1\otimes (n_2+n_3) = n_1\otimes n_2 + n_1\otimes n_3$	\text{ (distributivity) }
	% commutativity of multiplication
	\item $n_1\otimes n_2=n_2\otimes n_1$	\text{ (commutativity of multiplication)}
	% x-commutativity
	\item $n_1\otimes (n_2\otimes x)=n_2\otimes (n_1\otimes x)$	\text{ (x-commutativity) }
\end{enumerate}
\end{thm}
\begin{proof}
% Proof of x-neutral
1. It follows immediately from definition.

% Proof commutativity of addition
2. Suppose for some natural number $k$ we have $1+k=k+1$ then by (associativity) we get $1+(k+1)=(1+k)+1=(k+1)+1$. The desired result is obtained using induction.

% Proof of x-additivity
3. By (x-neutral) and by definition we obtain $(n_1\otimes x)+(1\otimes x)=(n\otimes x)+x=(n+1)\otimes x$. Suppose for some natural number $k$ we have $(n_1\otimes x)+(k\otimes x)=(n_1+k)\otimes x$. By definition we have 
$$(n_1\otimes x)+((k+1)\otimes x)=(n_1\otimes x)+(k\otimes x)+x$$
$$=((n_1+k)\otimes x)+x\text{ (by hypothesis)}$$ $$=((n_1+k)+1)\otimes x=(n_1+(k+1))\otimes x\text{ (by definition and associativity)}$$ We conclude the proof using induction.

% Proof of x-associativity
4. By (x-neutral) we have $1\otimes (n_2\otimes x)=n_2\otimes x=(1\otimes n_2)\otimes x$. Now suppose for some natural number $k$ we have $$k\otimes (n_2\otimes x)=(k\otimes n_2)\otimes x$$
By definition we get
$$(k+1)\otimes (n_2\otimes x)=(k\otimes (n_2\otimes x))+n_2\otimes x$$
$$=((k\otimes n_2)\otimes x)+(n_2\otimes x)\text{ (by hypothesis)}$$
$$=((k\otimes n_2)+n_2)\otimes x\text{ (by x-additivity)}$$
$$=((k+1)\otimes n_2)\otimes x\text{ (by definition)}$$
We conclude the proof using induction.

% Proof of x-distributivity
5. By definition we have $$1\otimes ((n_2+n_3)\otimes x)=(n_2+n_3)\otimes x$$
$$=(n_2\otimes x)+(n_3\otimes x)\text{ (by x-additivity)}$$
$$=((1\otimes n_2)\otimes x)+((1\otimes n_3)\otimes x)\text{ (by x-neutral)}$$
Now suppose for some natural number $k$ we have
$$k\otimes ((n_2+n_3)\otimes x)=((k\otimes n_2)\otimes x)+((k\otimes n_3)\otimes x)$$
By definition we get
$$(k+1)\otimes ((n_2+n_3)\otimes x)=k\otimes ((n_2+n_3)\otimes x)+((n_2+n_3)\otimes x)$$
$$=((k\otimes n_2)\otimes x)+((k\otimes n_3)\otimes x)+((n_2+n_3)\otimes x)\text{ (by hypothesis)}$$
$$=((k\otimes n_2)+(k\otimes n_3)+(n_2+n_3))\otimes x\text{ (by x-additivity)}$$
$$=((k\otimes n_2)+n_2+(k\otimes n_3)+n_3)\otimes x\text{ (by commutativity)}$$
$$=(((k+1)\otimes n_2)+((k+1)\otimes n_3))\otimes x\text{ (by definition)}$$
$$=(((k+1)\otimes n_2)\otimes x)+(((k+1)\otimes n_3)\otimes x)\text{ (by x-additivity)}$$
We conclude the proof using (word induction).

% Proof of distributivity
6. Follows immediately using (x-neutral) and (x-distributivity)

% Proof of commutativity of multiplication
7. Suppose $1\otimes n=n\otimes 1$. We get 
$$ 1\otimes(n+1) = (1\otimes n)+(1\otimes 1) \text{ (by distributivity) } $$
$$ = (n+1)\otimes 1 \text{ (by hypothesis and definition)} $$
Suppose for some natural number $k$ we have $k\otimes n = n\otimes k$. We get 
$$(k+1)\otimes n = k\otimes n + n \text{ (by definition) }$$
$$ = n\otimes k + n\otimes 1 \text{ (by hypothesis and by (x-neutral))} $$
$$ = n\otimes (k+1) \text{ (by distributivity) } $$
We conclude the proof using induction.

% Proof of x-commutativity
8. By (x-associativity) we have 
$$ n_1\otimes ( n_2\otimes x) = (n_1\otimes n_2)\otimes x $$
$$ = (n_2\otimes n_1 )\otimes x \text{ (by commutativity) } $$
$$ = n_2 \otimes (n_1\otimes x) \text{ (by x-associativity) } $$ 
\end{proof}

Now we can define the concept of length.

% Length of word
\begin{definition}
Let $S$ be the smallest system such that 
$$ 001\in S \wedge 101\in S \wedge \forall x \forall n( n\in \mathbb{N} \Rightarrow (x0n\in S \Rightarrow ( x00n1\in S \wedge x10n1\in S ) ) ) $$
Given any word $w$ and any natural number $n$, if $w0n\in S$ then we say that $n$ is the {\bf {\itshape length}} of $w$. 
\end{definition}

Using lemma \ref{wordnat} we can see that length is adequately defined for every word.

The next theorem shows that the notions of ``left-generated'' and ``right-generated'' are equivalent.

% Right generated = Left generated theorem

\begin{thm} \label{leftrightequi}
Given systems $S$, $S_1$ and $S_2$, if $S_1$ is right-generated by $S$ and $S_2$ is left-generated by $S$ then $S_1=S_2$.
\end{thm}
% Recheck
\begin{proof}
By definition we have $S\subseteq S_1$ and $S\subseteq S_2$. Let $S_3$ be the smallest system such that
$$ \forall x\in S ( x01\in S_3) \wedge \forall y \forall n( n\in \mathbb{N} \Rightarrow ( y0n\in S_3 \Rightarrow \forall z\in S(yz0n1\in S_3))) $$
and let $S_4$ be the smallest system such that
$$ \forall x\in S ( x01\in S_4) \wedge \forall y \forall n( n\in \mathbb{N} \Rightarrow ( y0n\in S_3 \Rightarrow \forall z\in S(zy0n1\in S_4))) $$

% Use parsing lemma
We have that for any $x\in S_1$ there exists a natural number $n$ such that $x0n\in S_3$ and for any $x\in S_2$ there exists a natural number $n$ such that $x0n\in S_4$.
Using lemma \ref{wordnat} we see that for any word $x$ and any natural number $n$ if $x0n\in S_3$ then $x\in S_1$.  Similarly, using lemma \ref{wordnat} we see that for any word $x$ and any natural number $n$ if $x0n\in S_4$ then $x\in S_2$.  So it will suffice to prove that $S_3=S_4$ to show that $S_1=S_2$.
 
If $x\in S$ then by definition of $S_3$ and $S_4$ we see that $x01\in S_3$ and $x01\in S_4$.
Suppose for some natural number $k$ that for all natural numbers $j\leq k$ we have $\forall x(x0j\in S_3 \Leftrightarrow x0j\in S_4)$.

If $k=1$ and $x0k\in S_3$ then $x\in S$, so we see directly that for any $z\in S$ we have $zx0k1\in S_3$. If $k=1$ and $x0k\in S_4$ then $x\in S$, so we see directly that for any $z\in S$ we have $xz0k1\in S_4$.

If $k>1$ and $x0k\in S_3$ then by hypothesis $x0k\in S_4$ so $x=z_1x_1$ for some $z_1\in S$ where $x_10k_1\in S_4$ with $k = k_1 + 1$. By hypothesis we get $x_10k_1\in S_3$ so it follows again by hypothesis that for any $z\in S$ we obtain $x_1z0k\in S_4$ (since $x_1z0k\in S_3$ by definition of $S_3$) which implies $z_1x_1z0k1\in S_4$ (by definition of $S_4$) showing that $xz0k1\in S_4$.

If $k>1$ and $x0k\in S_4$ then by hypothesis $x0k\in S_3$ so $x=x_1z_1$ for some $z_1\in S$ where $x_10k_1\in S_3$ with $k = k_1 + 1$. By hypothesis we get $x_10k_1\in S_4$ so it follows again by hypothesis that for any $z\in S$ we obtain $zx_10k\in S_3$ (since $zx_10k\in S_4$ by definition of $S_4$) which implies $zx_1z_10k1\in S_3$ (by definition of $S_3$) showing that $zx0k1\in S_3$.
So by inductive definition we see that $S_3=S_4$ which lets us conclude that $S_1=S_2$. 
\end{proof}

We saw in the proof of the preceding theorem that intermediate systems were used to parse the relevant inductive structure. The concept of parsing will serve as a foundational basis on which the concept of set will be defined. To introduce the notion of set we will first need to define some concepts which will be useful to a simple  standardization of ``parsing''.

\begin{definition}
Given
\begin{itemize}
	\item systems $S$ and $S_1$
	\item word $x$
\end{itemize}
if $S_1$ is right-generated by $S$ and $x\in S_1$ then we will say that $x$ is a {\bf {\itshape word over $S$}}.
\end{definition}

% Definition of alphabet

% definition of left/right readability
\begin{definition}
Let $S$ be a system.
\begin{itemize}
	\item If $\forall x \forall y \forall z_1\in S \forall z_2\in S( xz_1=yz_2 \Rightarrow z_1=z_2 )$ then we say that $S$ is {\bf {\itshape right-readable}}.
	\item If $\forall x \forall y \forall z_1\in S \forall z_2\in S( z_1x=z_2y \Rightarrow z_1=z_2 )$ then we say that $S$ is {\bf {\itshape left-readable}}.
\end{itemize}
\end{definition}

% Definition of alphabet
\begin{definition}
We will say that $S$ is an {\bf {\itshape alphabet}} if it is right-readable and left-readable.
\end{definition}

% prefix\suffix caracterisation of left\right readability
\begin{thm} \label{readabilitychar}
Let $S$ be a system.
\begin{enumerate}
	\item $S$ is right-readable if and only if 
	$$ \forall x\in S \forall y\in S( \text{ $x$ is a suffix of $y$ }\Rightarrow x=y ) $$
	\item $S$ is left-readable if and only if 
	$$ \forall x\in S \forall y\in S( \text{ $x$ is a prefix of $y$ }\Rightarrow x=y ) $$
\end{enumerate}
\end{thm}

\begin{proof}
1. If $S$ is right-readable suppose for some $x\in S$ and some $y\in S$ we have that $x$ is a suffix of $y$ and that $y=wx$ for some word $w$. So for any word $w_1$ we get $w_1y=w_1wx$ which implies by right-readability that $y=x$ giving us a contradiction. It follows that $ \forall x\in S \forall y\in S( \text{ $x$ is a suffix of $y$ }\Rightarrow x=y )$.

Now suppose we have $ \forall x\in S \forall y\in S( \text{ $x$ is a suffix of $y$ }\Rightarrow x=y ) $. For any words $w_1$ and $w_2$ and any words $x\in S$ and $y\in S$ we get by theorem \ref{overlas} that $w_1x=w_2y$ implies that either $x$ is a suffix of $y$ or that $y$ is a suffix of $x$. By assumption it follows that $x=y$ so $S$ must be right-readable.

2. If $S$ is left-readable suppose for some $x\in S$ and some $y\in S$ we have that $x$ is a prefix of $y$ and that $y=xw$ for some word $w$. So for any word $w_1$ we get $yw_1=xww_1$ which implies by left-readability that $y=x$ giving us a contradiction. It follows that $ \forall x\in S \forall y\in S( \text{ $x$ is a prefix of $y$ }\Rightarrow x=y )$.

Now suppose we have $ \forall x\in S \forall y\in S( \text{ $x$ is a prefix of $y$ }\Rightarrow x=y ) $. For any words $w_1$ and $w_2$ and any words $x\in S$ and $y\in S$ we get by theorem \ref{overlap2} that $xw_1=yw_2$ implies that either $x$ is a prefix of $y$ or that $y$ is a prefix of $x$. By assumption it follows that $x=y$ so $S$ must be left-readable.
\end{proof}

% alphabet integrity
\begin{thm} \label{alphintegri}
Given an alphabet $A$ and a word $x$, if $x\in A$ then for any word $w_1$ over $A$ and any word $w_2$ over $A$ we must have $x\neq w_1w_2$.
\end{thm}

\begin{proof}
By theorem \ref{readabilitychar} we see that if $x\in A$, $w_1\in A$ and $w_2\in A$ then $w_1w_2\neq x$. Suppose that for some word $k$ over $A$ we have $w_1k\neq x$ for any word $x\in A$ and any word $w_1$ over $A$. Lets proceed by induction (inductive definition of words over $A$) on $k$ and by contradiction. Suppose that $w_1ka=x$ for some word $x\in A$ and some word $a\in A$. So we must have $zw_1ka=zx$ for any word $z\in A$ which implies by right readability of $A$ and by (simplification) that $zw_1k=z$ contradicting the hypothesis since $zw_1$ is a word over $A$ by theorem \ref{leftrightequi}. Using the inductive definition of words over $A$ we obtain the desired result.
\end{proof}

% Definition of independence
\begin{definition}
Given alphabets $A_1$ and $A_2$, we say that {\bf {\itshape $A_1$ and $A_2$ are independent}} if
\begin{itemize}
	\item $A_1 \cup A_2$ is an alphabet
	\item $A_1 \cap A_2 = \emptyset$
\end{itemize}
\end{definition}

Let $x$ and $y$ be words and $A$ be an alphabet. If for any system $S$ which contains only $x$ we have that $S$ and $A$ are independent then we will say that {\bf {\itshape $A$ and $x$ are independent}} (or {\bf {\itshape $x$ and $A$ are independent}}). If for any system $S_1$ containing only $x$ and any system $S_2$ containing only $y$ we have that $S_1$ and $S_2$ are independent then we will say that {\bf {\itshape $x$ and $y$ are independent}}.

% (A15) independence mixing
\begin{lem} \label{indeph}
Given
\begin{itemize}
	\item alphabets $S_1$ and $S_2$
	\item words $r_1$ and $t_1$ both over $S_1$
	\item word $r_2$ over $S_2$
\end{itemize}
if $S_1$ and $S_2$ are independent then
\begin{enumerate}
	\item $r_1r_2$ is not a word over $S_1$
	\item $r_1r_2$ is not a word over $S_2$
	\item $r_1r_2t_1$ is not a word over $S_1$
	\item $r_1r_2t_1$ is not a word over $S_2$
\end{enumerate}
\end{lem}

\begin{proof}
1. If $r_2$ is in $S_2$ then by independence and the inductive definition of words over $S_1$ we see that $r_1r_2$ is not a word over $S_1$. Suppose for some word $k$ over $S_2$ that $r_1k$ is not a word over $S_1$. For any $u$ in $S_2$, by the inductive definition of words over $S_1$, we deduce that $r_1ku$ is not a word over $S_1$. We complete the proof using the inductive definition of words over $S_2$.

2. Using the left-right system generation equivalence(theorem \ref{leftrightequi}) and similar reasoning as in 1. we deduce that $r_1r_2$ is not a word over $S_2$.

3. Using 1. and the inductive definition of words over $S_1$ we see that $r_1r_2t_1$ is not a word over $S_1$.

4. By the inductive definition of words over $S_2$ we see that $r_1r_2t_1$ is not a word over $S_2$.
\end{proof}

% Universal Seperation Theorem
% Dependencies :
\begin{thm} \label{indept}
Given
\begin{itemize}
	\item alphabets $S_1$ and $S_2$
	\item words $r_1$ and $t_1$ over $S_1$
	\item words $r_2$ and $t_2$ over $S_2$
	\item words $x$ and $y$
\end{itemize}
if $S_1$ and $S_2$ are independent then
\begin{enumerate}
	\item $r_1r_2=t_1t_2\Rightarrow (r_1=t_1 \wedge r_2=t_2)$
	\item $(xr_1r_2=yt_1t_2 \wedge r_1=t_1)\Rightarrow (r_2=t_2 \wedge x=y)$
\end{enumerate}
\end{thm}

\begin{proof}
1. Suppose $r_1r_2=t_1t_2$. If $r_2$ is in $S_2$ then by (simplification) and definition of independence we deduce that $t_2$ is in $S_2$ and $r_2=t_2$ so by (simplification) $r_1=t_1$. Suppose that for some word $k$ over $S_2$ we have that $r_1k=t_1t_2\Rightarrow (k=t_2 \wedge r_1=t_1)$. For any $u$ in $S_2$, if $r_1ku=t_1t_2$ then $t_2$ is not in $S_2$ by lemma \ref{indeph}. So $t_2=w_2u$ where $w_2$ is a word over $S_2$ and by (simplification) we get $r_1k=t_1w_2$. By hypothesis we get $k=w_2$ and $r_1=t_1$ which implies that $ku=w_2u=t_2$ so by the inductive definition of words over $S_2$ we obtain the desired result.

2. Suppose $xr_1r_2=yt_1t_2 \wedge r_1=t_1$. If $r_2$ is in $S_2$ then by independence and (simplification) $t_2$ must be in $S_2$ and $r_2=t_2$ so by (simplification) $xr_1=xt_1$. Since $r_1=t_1$ then by (simplification) we get $x=y$. Suppose for some word $k$ over $S_2$ we have that $(xr_1k=yt_1t_2 \wedge r_1=t_1)\Rightarrow (k=t_2 \wedge x=y)$. For any $u$ in $S_2$, if $xr_1ku=yt_1t_2$ and $r_1=t_1$ then by independence $t_2$ is not in $S_2$ so by (simplification) $xr_1k=yt_1w_2$ where $w_2$ is a word over $S_2$ with $t_2=w_2u$. By hypothesis we get $x=y$ and $k=w_2$ which implies that $ku=w_2u=t_2$ and by the inductive definition of words over $S_2$ we complete the proof.
\end{proof}

% Definition of expressivity
\begin{definition}
Let $A$ be a non-empty alphabet and $S$ be right-generated by $A$. We will say that $A$ is {\bf {\itshape expressive }}if there exist independent alphabets $B_1$ and $B_2$ with $B_1\subseteq S$ and $B_2\subseteq S$.
\end{definition}
% Theorem on caracterisation of expressivity
\begin{thm} \label{express}
An alphabet $A$ is expressive if and only if $A$ contains at least two distinct words.
\end{thm}

\begin{proof}
If $A$ contains only one word then for any two words $x$ and $y$ both over $A$ we have that $x$ is a suffix of $y$ or $y$ is a suffix of $x$. Using theorem \ref{readabilitychar} we deduce that we cannot form two independent alphabets from the words produced over $A$.

Suppose $A$ contains $x$ and $y$ with $x\neq y$. Let $B_1$ be a system containing only the word $x$ and $B_2$ be a system containing only the word $y$. We see that $B_1$ and $B_2$ are alphabets and since $A$ is an alphabet we can also deduce that $B_1$ and $B_2$ are independent.

% TO CUT
% END TO CUT

\end{proof}

In a certain sense, the concept of expressivity is the main motivation behind the (symbols) axiom. The next theorem shows the ``inductive power'' of expressivity.

% Theorem on self-inducing property of expressivity
\begin{thm}
Let $S$ be right-generated by an expressive alphabet $A$. There exist independent expressive alphabets $B_1$ and $B_2$ such that $B_1\subseteq S$ and $B_2\subseteq S$.
\end{thm}

\begin{proof}
By theorem \ref{express} there exist distinct words $x$ and $y$ in $A$. Let $B_1$ be a system containing only $xx$ and $xy$ and $B_2$ be a system containing only $yx$ and $yy$. By theorem \ref{readabilitychar}, left simplification and (simplification) we deduce that
\begin{itemize}
	\item $xx$ is not a prefix or suffix of $xy$ and vice-versa
	\item $yy$ is not a prefix or suffix of $xy$ and vice-versa
	\item $xx$ is not a prefix or suffix of $yx$ and vice-versa
	\item $yy$ is not a prefix or suffix of $yx$ and vice-versa
	\item $xx$ is not a prefix or suffix of $yy$ and vice-versa
	\item $xy$ is not a prefix or suffix of $yx$ and vice-versa
\end{itemize}
So we can see that $B_1$ and $B_2$ are independent and by theorem \ref{express} we conclude that both $B_1$ and $B_2$ are expressive alphabets.
\end{proof}

\section{Sets and structures}

\subsection{Sets}

The main idea behind the notion of set will be to represent systems using words from an expressive alphabets, establish meaning via a ``background system'' and extract set semantics using a parsing scheme.

The next definition will highlight the ``parsing tool'' which will be used to extract the meaning of a set. 

% Definition of set foundation
\begin{definition}
Let $S_0$ be an expressive alphabet and $e_0$ be a word. If $e_0$ and $S_0$ are independent then we say that $(e_0,S_0)$ is a {\bf {\itshape set foundation}}.
\end{definition}
 
% Definition of set
\begin{definition}
Given
\begin{itemize}
	\item set foundation $(e_0,S_0)$
	\item system $S$
	\item word $e$
\end{itemize}
if $e$ is a word over $S_0$ then we say that $(e)_{S}$ is a {\bf {\itshape set over}} $(e_0,S_0)$, $e$ is the {\bf {\itshape representation of $(e)_{S}$ over $(e_0,S_0)$}} and $S$ is the {\bf {\itshape context system of $(e)_{S}$ over $(e_0,S_0)$}}.
\end{definition}

If $(e)_{S}$ is a set over $(e_0,S_0)$ and $xe_0e\in S$ then we say that $x$ is an {\bf {\itshape element of }} $(e)_{S}$ {\bf {\itshape over}} $(e_0,S_0)$ or write $x\overset{}{\underset{ (e_0,S_0)} {\in } }(e)_{S}$. If $x$ is not an element of $(e)_{S}$ over $(e_0,S_0)$ we will write $x\overset{}{\underset{ (e_0,S_0)} {\notin } }(e)_{S}$.

For example, let $S_0$ be an alphabet which contains only $01$ and $10$ and let $S$ be some system. If $00010\in S$ then we can say that $0$ is an element of $(10)_{S}$ over $(00,S_0)$.

Using theorem \ref{indept}, we can verify that for any system $S_1$ and any set foundation $(e_0,S_0)$ there exists a word $e$ and a system $S$ such that $\forall x(x\in S_1 \Leftrightarrow x\overset{}{\underset{ (e_0,S_0)} {\in } }(e)_{S})$.

We can see sets as ``internalized systems'' which can contain representations of sets and context systems as safeguards against malformed sets. In a certain sense, the existence of a context system depends on the consistency of the specification of a set.

The set builder notation will be adapted by mentioning the set foundation underneath the equality sign. For example, instead of writing 

$$ \forall x (x\overset{}{\underset{ (e_0,S_0)} {\in } }(s)_{S} \Leftrightarrow (x=0 \vee x=00 \vee x=1))$$

we would write

$$ (s)_{S}\overset{}{\underset{ (e_0,S_0)} {= } }\{0,00,1  \}$$

When dealing with logical formulas, instead of writing

$$ \forall x (x\overset{}{\underset{ (e_0,S_0)} {\in } }(s)_{S} \Leftrightarrow \phi(x) ) $$

we might write 

$$ (s)_{S}\overset{}{\underset{ (e_0,S_0)} {= } }\{x \mid \phi(x)  \}$$

The notation ``$\overset{}{\underset{ (e_0,S_0)} {= } }$'' will be treated like an equality symbol when there is no ambiguity.

% Definition of set operations

% - Union
% - Intersection
% - Substraction

\begin{definition}
Given
\begin{itemize}
	\item set foundation $(e_0,S_0)$
	\item sets $(e_1)_{E_1}$ over $(e_0,S_0)$, $(e_2)_{E_2}$ over $(e_0,S_0)$ and $(e_3)_{E_3}$ over $(e_0,S_0)$
\end{itemize}
we say
\begin{itemize}
	\item $(e_1)_{E_1}$ is a {\bf {\itshape subset of $(e_2)_{E_2}$ over $(e_0,S_0)$}}  or write $(e_1)_{E_1}\overset{}{\underset{ (e_0,S_0)} {\subseteq } }(e_2)_{E_2}$ ( or $(e_2)_{E_2}\overset{}{\underset{ (e_0,S_0)} {\supseteq} }(e_1)_{E_1}$) if for any word $x$, $x\overset{}{\underset{ (e_0,S_0)} {\in } }(e_1)_{E_1} \Rightarrow x\overset{}{\underset{ (e_0,S_0)} {\in } }(e_2)_{E_2}$.
	\item $(e_1)_{E_1}$ is {\bf {\itshape equal to $(e_2)_{E_2}$ over $(e_0,S_0)$}} or write $(e_1)_{E_1}\overset{}{\underset{ (e_0,S_0)} {= } }(e_2)_{E_2}$ if $(e_1)_{E_1}\overset{}{\underset{ (e_0,S_0)} {\subseteq } }(e_2)_{E_2}$ and $(e_2)_{E_2}\overset{}{\underset{ (e_0,S_0)} {\subseteq } }(e_1)_{E_1}$. Otherwise we say that $(e_1)_{E_2}$ is {\bf {\itshape not equal to $(e_2)_{E_2}$ over $(e_0,S_0)$}} or write $(e_1)_{E_1}\overset{}{\underset{ (e_0,S_0)} {\neq } }(e_2)_{E_2}$.
	\item $(e_1)_{E_1}$ is a {\bf {\itshape strict subset of $(e_2)_{E_2}$ over $(e_0,S_0)$}} or write $(e_1)_{E_1}\overset{}{\underset{ (e_0,S_0)} {\subset } }(e_2)_{E_2}$ (or $(e_2)_{E_2}\overset{}{\underset{ (e_0,S_0)} {\supset } }(e_1)_{E_1}$) if $(e_1)_{E_1}\overset{}{\underset{ (e_0,S_0)} {\subseteq } }(e_2)_{E_2}$ and $(e_1)_{E_1}\overset{}{\underset{ (e_0,S_0)} {\neq } }(e_2)_{E_2}$.
	\item $(e_3)_{E_3}$ equals the {\bf {\itshape union of $(e_1)_{E}$ and $(e_2)_{E_2}$ over $(e_0,S_0)$ }} or write $(e_3)_{E_3}\overset{}{\underset{ (e_0,S_0)} {= } }(e_1)_{E_1} \cup (e_2)_{E_2}$, if for any word $x$, $x\overset{}{\underset{ (e_0,S_0)} {\in } }(e_3)_{E_3}$ if and only if $x\overset{}{\underset{ (e_0,S_0)} {\in } }(e_1)_{E_1}$ or $x\overset{}{\underset{ (e_0,S_0)} {\in } }(e_2)_{E_2}$. 
	\item $(e_3)_{E_3}$ equals the {\bf {\itshape intersection of $(e_1)_{E}$ and $(e_2)_{E_2}$ over $(e_0,S_0)$ }} or write $(e_3)_{E_3}\overset{}{\underset{ (e_0,S_0)} {= } }(e_1)_{E_1} \cap (e_2)_{E_2}$, if for any word $x$, $x\overset{}{\underset{ (e_0,S_0)} {\in } }(e_3)_{E_3}$ if and only if $x\overset{}{\underset{ (e_0,S_0)} {\in } }(e_1)_{E_1}$ and $x\overset{}{\underset{ (e_0,S_0)} {\in } }(e_2)_{E_2}$. 
	\item $(e_3)_{E_3}$ equals the {\bf {\itshape set difference of $(e_1)_{E}$ by $(e_2)_{E_2}$ over $(e_0,S_0)$ }} or write $(e_3)_{E_3}\overset{}{\underset{ (e_0,S_0)} {= } }(e_1)_{E_1} \setminus (e_2)_{E_2}$, if for any word $x$, $x\overset{}{\underset{ (e_0,S_0)} {\in } }(e_3)_{E_3}$ if and only if $x\overset{}{\underset{ (e_0,S_0)} {\in } }(e_1)_{E_1}$ and $x\overset{}{\underset{ (e_0,S_0)} {\notin } }(e_2)_{E_2}$.
	\item $(e_1)_{E_1}$ equals the {\bf {\itshape empty set over $(e_0,S_0)$}} or write $(e_1)_{E_1}\overset{}{\underset{ (e_0,S_0)} {= } }\emptyset$, if for any word $x$, $x\overset{}{\underset{ (e_0,S_0)} {\notin } }(e_1)_{E_1}$. 
\end{itemize}
\end{definition}

% Definition of set containment
\begin{definition}
Given sets $(s_1)_{S_1}$ and $(s_2)_{S_2}$ over $(e_0,S_0)$, we say that $(s_1)_{S_1}$ {\bf {\itshape contains $(s_2)_{S_2}$ over $(e_0,S_0)$ }} if there exists $x\overset{}{\underset{ (e_0,S_0)} {\in } }(s_1)_{S_1}$ such that $(x)_{S_1}\overset{}{\underset{ (e_0,S_0)} {= } }(s_2)_{S_2}$.
\end{definition}

Note that set containment does not necessarily capture containment of the whole ``set structure'' of interest. Containment will be understood by defining new ``structure equalities'' on a case by case basis. In particular, the next subsections will deal with the set structure of relations, integers, rational numbers and real numbers.

Now we define the notion of power set.
% Definition of power set
\begin{definition}
Given sets $(e)_{E}$ and $(p)_{P}$ over $(e_0,S_0)$, if 
	\begin{itemize}
		\item $(p)_{P}$ contains every subset of $(e)_{E}$ over $(e_0,S_0)$
		\item for any $x\overset{}{\underset{ (e_0,S_0)} {\in } }(p)_{P}$ and any $y\overset{}{\underset{ (e_0,S_0)} {\in } }(p)_{P}$, if $(x)_{P}\overset{}{\underset{ (e_0,S_0)} {= } }(y)_{P}$ then $x=y$
		\item $\forall x\overset{}{\underset{ (e_0,S_0)} {\in } }(p)_{P}((x)_{P}\overset{}{\underset{ (e_0,S_0)} {\subseteq } }(e)_{E})$
	\end{itemize}
then we say that $(p)_{P}$ equals the {\bf {\itshape power set of $(e)_{E}$ over $(e_0,S_0)$}}.
\end{definition}

\subsection{Relations and functions}

% Association
Now we formulate the notion of association.
\begin{definition}
Given
\begin{itemize}
	\item set $(e)_{E}$ over $(e_0,S_0)$
	\item words $x_1$ and $x_2$
\end{itemize}
if $(e)_{E}\overset{}{\underset{ (e_0,S_0)} {= } }\{x_1,x_1x_2\}$ then we say that $(e)_{E}$ {\bf {\itshape associates $x_1$ to $x_2$ over $(e_0,S_0)$}}  and write $(e)_{E}\overset{}{\underset{ (e_0,S_0)} {= } }(x_1,x_2)$.
\end{definition}

Given a set $(s)_{S}$ over $(e_0,S_0)$ and words $x_1$ and $x_2$, if there exists $y$ such that $(y)_{S}\overset{}{\underset{ (e_0,S_0)} {= } }(x_1,x_2)$ and $y\overset{}{\underset{ (e_0,S_0)} {\in } }(s)_{S}$ then we will just write $(x_1,x_2)\overset{}{\underset{ (e_0,S_0)} {\in } }(s)_{S}$.

% Binary relation and Cartesian Product
\begin{definition}
Let $(a)_{A}$, $(b)_{B}$ and $(c)_{C}$ be sets over $(e_0,S_0)$. We say that 
\begin{itemize}
	\item $((a)_{A},(b)_{B},(c)_{C})_{(e_0,S_0)}$ is a {\bf {\itshape (binary) relation }} or write $(c)_{C}\overset{}{\underset{ (e_0,S_0)} {\subseteq } }(a)_{A} \times (b)_{B}$ if 
		\begin{itemize}
			\item for any $c_1\overset{}{\underset{ (e_0,S_0)} {\in } }(c)_{C}$ there exist $a_1\overset{}{\underset{ (e_0,S_0)} {\in } }(a)_{A}$ and $b_1\overset{}{\underset{ (e_0,S_0)} {\in } }(b)_{B}$ such that  $(c_1)_{C}\overset{}{\underset{ (e_0,S_0)} {= } }(a_1,b_1)$ 
			\item for any $a_1\overset{}{\underset{ (e_0,S_0)} {\in } }(a)_{A}$ and any $b_1\overset{}{\underset{ (e_0,S_0)} {\in } }(b)_{B}$, if there exist $c_1\overset{}{\underset{ (e_0,S_0)} {\in } }(c)_{C}$ and $c_2\overset{}{\underset{ (e_0,S_0)} {\in } }(c)_{C}$ such that $(c_1)_{C}\overset{}{\underset{ (e_0,S_0)} {= } }(a_1,b_1)$ and $(c_2)_{C}\overset{}{\underset{ (e_0,S_0)} {= } }(a_1,b_1)$ then $c_1=c_2$
		\end{itemize}
		\item $(c)_{C}$ is the {\bf {\itshape cartesian product of $(a)_{A}$ and $(b)_{B}$ over $(e_0,S_0)$ }} or write  $(c)_{C}\overset{}{\underset{ (e_0,S_0)} {= } }(a)_{A} \times (b)_{B}$ if 
		\begin{itemize}
		
			\item for any $c_1\overset{}{\underset{ (e_0,S_0)} {\in } }(c)_{C}$ there exist $a_1\overset{}{\underset{ (e_0,S_0)} {\in } }(a)_{A}$ and $b_1\overset{}{\underset{ (e_0,S_0)} {\in } }(b)_{B}$ such that  $(c_1)_{C}\overset{}{\underset{ (e_0,S_0)} {= } }(a_1,b_1)$ 
		
		  \item for any $a_1\overset{}{\underset{ (e_0,S_0)} {\in } }(a)_{A}$ and any $b_1\overset{}{\underset{ (e_0,S_0)} {\in } }(b)_{B}$ there exists a unique $c_1\overset{}{\underset{ (e_0,S_0)} {\in } }(c)_{C}$ such that $(c_1)_{C}\overset{}{\underset{ (e_0,S_0)} {= } }(a_1,b_1)$
		  \end{itemize}
\end{itemize}
\end{definition}

\begin{definition}
Let $( (a)_{A},(b)_{B},(c)_{C} )_{(e_0,S_0)}$ be a binary relation and $(d)_{D}$ a set over $(e_0,S_0)$. We say that
\begin{itemize}
	\item $(d)_{D}$ is the {\bf {\itshape left domain of}} $( (a)_{A},(b)_{B},(c)_{C} )_{(e_0,S_0)}$ if $(d)_{D}\overset{}{\underset{ (e_0,S_0)} {= } }\{x \mid \exists b_1( (x,b_1)\overset{}{\underset{ (e_0,S_0)} {\in } }(c)_{C})  \}$
	\item  $(d)_{D}$ is the {\bf {\itshape right domain of}} $( (a)_{A},(b)_{B},(c)_{C} )_{(e_0,S_0)}$ if $(d)_{D}\overset{}{\underset{ (e_0,S_0)} {= } }\{x \mid \exists a_1( (a_1,x)\overset{}{\underset{ (e_0,S_0)} {\in } }(c)_{C})  \}$
\end{itemize}
\end{definition}

% Definition of power set
%\begin{definition}
%Given sets $(e)_{E}$ and $(p)_{P}$ over $(e_0,S_0)$, if 
%	\begin{itemize}
%		\item for any subset $(s)_{S}$ of $(e)_{E}$ over $(e_0,S_0)$ there exists a unique $p_1\overset{}{\underset{ (e_0,S_0)} {\in } }(p)_{P}$ such that $(p_1)_{P}\overset{}{\underset{ (e_0,S_0)} {= } }(s)_{S}$
%		\item every element of $(p)_{P}$ over $(e_0,S_0)$ is a subset of $(e)_{E}$ over $(e_0,S_0)$
%	\end{itemize}
%then we say that $(p)_{P}$ equals the {\bf {\itshape power set of $(e)_{E}$ over $(e_0,S_0)$}}.
%\end{definition}

% Definition of function
\begin{definition}
Given a binary relation $((a)_{A},(b)_{B},(f)_{F})_{(e_0,S_0)}$, if
\begin{itemize}
	% left total
	\item $(a)_{A}$ is the left domain of $( (a)_{A},(b)_{B},(c)_{C} )_{(e_0,S_0)}$
	% unique image
	\item for any $a_1\overset{}{\underset{ (e_0,S_0)} {\in } }(a)_{A}$, any $b_1\overset{}{\underset{ (e_0,S_0)} {\in } }(b)_{B}$ and any $b_2\overset{}{\underset{ (e_0,S_0)} {\in } }(b)_{B}$, if $(a_1,b_1)\overset{}{\underset{ (e_0,S_0)} {\in } }(f)_{F}$ and $(a_1,b_2)\overset{}{\underset{ (e_0,S_0)} {\in } }(f)_{F}$ then $b_1=b_2$
\end{itemize}
then we say that $(f)_{F}:(a)_{A}\underset{ (e_0,S_0)} {\rightarrow } (b)_{B}$ is a {\bf {\itshape function}}. 
\end{definition}

% image, restricting

Let $(f)_{F}:(a)_{A}\underset{ (e_0,S_0)} {\rightarrow} (b)_{B}$ be a function. We say that $(a)_{A}$ is the {\bf {\itshape domain}} of $(f)_{F}:(a)_{A}\underset{ (e_0,S_0)} {\rightarrow} (b)_{B}$.
For any $x$ and any $y$, if $(x,y)\overset{}{\underset{ (e_0,S_0)} {\in } }(f)_{F}$ then we write $(f)_{F}(x)\overset{}{\underset{ (e_0,S_0)} {= } }y$. If $(c)_{C}\overset{}{\underset{ (e_0,S_0)} {= } }\{(f)_{F}(x) \mid x\overset{}{\underset{ (e_0,S_0)} {\in } }(a)_{A}  \}$ then we say that $(c)_{C}$ equals the {\bf{\itshape image }} of $(f)_{F}$ over $(e_0,S_0)$ or write $(c)_{C}\overset{}{\underset{ (e_0,S_0)} {= } }(f)_{F}( (a)_{A} )$. For any function $(g)_{G}:(d)_{D}\underset{ (e_0,S_0)} {\rightarrow }(e)_{E} $, if $(d)_{D}\overset{}{\underset{ (e_0,S_0)} { \subseteq} }(a)_{A}$ and for any element $x$ of $(d)_{D}$ over $(e_0,S_0)$ we have $(g)_{G}(x)\overset{}{\underset{ (e_0,S_0)} {= } }(f)_{F}(x)$ then we say that $(g)_{G}$ is a {\bf{\itshape restriction }} of $(f)_{F}$ to $(d)_{D}$ over $(e_0,S_0)$ or write $(g)_{G}\overset{}{\underset{ (e_0,S_0)} {= } }(f)_{F}|(d)_{D}$.

% Definition of injective, surjective and bijective
\begin{definition}
Given a function $(f)_{F}:(a)_{A}\underset{ (e_0,S_0)} {\rightarrow }(b)_{B} $ we say that
	\begin{itemize}
		% injective
		\item $(f)_{F}:(a)_{A}\underset{ (e_0,S_0)} {\rightarrow }(b)_{B} $ is {\bf {\itshape injective}} if for any $a_1\overset{}{\underset{ (e_0,S_0)} {\in } }(a)_{A}$, any $a_2\overset{}{\underset{ (e_0,S_0)} { \in} }(a)_{A}$ and any $b_1\overset{}{\underset{ (e_0,S_0)} {\in } }(b)_{B}$, if $(f)_{F}(a_1)\overset{}{\underset{ (e_0,S_0)} {= } }b_1$ and $(f)_{F}(a_2)\overset{}{\underset{ (e_0,S_0)} {= } }b_1$ then $a_1=a_2$
		% surjective 
		\item $(f)_{F}:(a)_{A}\underset{ (e_0,S_0)} {\rightarrow }(b)_{B}$ is {\bf {\itshape surjective}} if for any $b_1\overset{}{\underset{ (e_0,S_0)} {\in } }(b)_{B}$ there exists $a_1\overset{}{\underset{ (e_0,S_0)} {\in } }(a)_{A}$ such that $(f)_{F}(a_1)\overset{}{\underset{ (e_0,S_0)} {= } }b_1$
		% bijective
		\item $(f)_{F}:(a)_{A}\underset{ (e_0,S_0)} {\rightarrow }(b)_{B}$ is {\bf {\itshape bijective}} if $(f)_{F}:(a)_{A}\underset{ (e_0,S_0)} {\rightarrow }(b)_{B}$ is both injective and surjective
	\end{itemize}
\end{definition}

% natural number set

Let $(s)_{S}$ be a set over $(e_0,S_0)$. If $(s)_{S}\overset{}{\underset{ (e_0,S_0)} {= } }\{n \mid  n\in \mathbb{N} \}$ then we will write $(s)_{S}\overset{}{\underset{ (e_0,S_0)} {= } }\mathbb{N}$. For a given natural number $n$, if $(s)_{S}\overset{}{\underset{ (e_0,S_0)} {= } }\{m \mid m\in \mathbb{N} \text{ and } m\leq n  \}$ then we will write $(s)_{S}\overset{}{\underset{ (e_0,S_0)} {= } }\mathbb{N}_{\leq n}$.

% Definition of cardinality, finite and infinite set
\begin{definition}
Given a set $(s)_{S}$ over $(e_0,S_0)$ and a natural number $n$, we say that the {\bf {\itshape cardinality of $(s)_{S}$ over $(e_0,S_0)$}} is $n$ or just write $card((s)_{S})\overset{}{\underset{ (e_0,S_0)} {= } }n$ if there exists a bijective function $(f)_{F}:(s)_{S}\overset{}{\underset{ (e_0,S_0)} {\rightarrow } }\mathbb{N}_{\leq n}$.
\end{definition}

If there exists $n$ such that $card((s)_{S})\overset{}{\underset{ (e_0,S_0)} {= } }n$ then we say that $(s)_{S}$ is a {\bf {\itshape finite set over $(e_0,S_0)$}}, otherwise we say that $(s)_{S}$ is an {\bf {\itshape infinite set over $(e_0,S_0)$}}.

% Definition of countable
\begin{definition}
Let $(s)_{S}$ be a set over $(e_0,S_0)$. If there exists a bijective function $(f)_{F}:(s)_{S}\overset{}{\underset{ (e_0,S_0)} {\rightarrow } }(t)_{T}$ where either $(t)_{T}\overset{}{\underset{ (e_0,S_0)} {= } }\mathbb{N}$ or $(t)_{T}\overset{}{\underset{ (e_0,S_0)} {= } }\mathbb{N}_{\leq n}$ for some $n\in \mathbb{N}$ then we say that $(s)_{S}$ is {\bf {\itshape countable over $(e_0,S_0)$}}.
\end{definition}

% Definition of partial order
\begin{definition}
Given a relation $( (a)_{A},(a)_{A},(r)_{R} )_{(e_0,S_0)}$, if for any $a_1\overset{}{\underset{ (e_0,S_0)} {\in } }(a)_{A}$, $a_2\overset{}{\underset{ (e_0,S_0)} {\in } }(a)_{A}$ and $a_3\overset{}{\underset{ (e_0,S_0)} {\in } }(a)_{A}$ we have
\begin{itemize}
	% reflexivity
	\item $(a_1,a_1)\overset{}{\underset{ (e_0,S_0)} {\in } }(r)_{R}$
	% antisymmetry
	\item $((a_1,a_2)\overset{}{\underset{ (e_0,S_0)} {\in } }(r)_{R} \wedge (a_2,a_1)\overset{}{\underset{ (e_0,S_0)} {\in } }(a)_{R}) \Rightarrow a_1=a_2$
	% transitivity
	\item $((a_1,a_2)\overset{}{\underset{ (e_0,S_0)} {\in } }(r)_{R} \wedge (a_2,a_3)\overset{}{\underset{ (e_0,S_0)} {\in } }(r)_{R}) \Rightarrow (a_1,a_3)\overset{}{\underset{ (e_0,S_0)} {\in } }(r)_{R}$
\end{itemize}
then we say that $( (a)_{A},(r)_{R} )_{(e_0,S_0)}$ is a {\bf {\itshape partial order}}.
\end{definition}

Let $( (a)_{A},(r)_{R} )_{(e_0,S_0)}$ be a partial order.

For any $x\overset{}{\underset{ (e_0,S_0)} {\in } }(a)_{A}$ and any $y\overset{}{\underset{ (e_0,S_0)} {\in } }(a)_{A}$, if $(x,y)\overset{}{\underset{ (e_0,S_0)} {\in } }(r)_{R}$ then we write $x\overset{(r)_{R}}{\underset{ (e_0,S_0)} {\leq } }y$ or $y\overset{(r)_{R}}{\underset{ (e_0,S_0)} {\geq } }x$. If $x\overset{(r)_{R}}{\underset{ (e_0,S_0)} {\leq } }y$ and $x\neq y$ then we write $x\overset{(r)_{R}}{\underset{ (e_0,S_0)} {<} }y$ or $y\overset{(r)_{R}}{\underset{ (e_0,S_0)} {> } }x$.

% Definition of dense order
\begin{definition}
Given a partial order $( (a)_{A},(r)_{R} )_{(e_0,S_0)}$, if $\forall a_1\overset{}{\underset{ (e_0,S_0)} {\in } }(a)_{A} \forall a_2\overset{}{\underset{ (e_0,S_0)} {\in } }(a)_{A}( a_1\overset{(r)_{R}}{\underset{ (e_0,S_0)} { <} }a_2 \Rightarrow \exists a_3\overset{}{\underset{ (e_0,S_0)} {\in } }(a)_{A}( a_1\overset{(r)_{R}}{\underset{ (e_0,S_0)} {< } }a_3\overset{(r)_{R}}{\underset{ (e_0,S_0)} {< } }a_2) )$ then we say that $( (a)_{A},(r)_{R} )_{(e_0,S_0)}$ is a {\bf {\itshape dense order}}.
\end{definition}

% Definition of totally ordered
\begin{definition}
Given a partial order $( (a)_{A}, (r)_{R} )_{(e_0,S_0)}$, if for any $a_1\overset{}{\underset{ (e_0,S_0)} {\in } }(a)_{A}$ and any $a_2\overset{}{\underset{ (e_0,S_0)} {\in } }(a)_{A}$ we have $a_1\overset{(r)_{R}}{\underset{ (e_0,S_0)} {\leq } }a_2$ or $a_2\overset{(r)_{R}}{\underset{ (e_0,S_0)} {\leq } }a_1$ then we say that $( (a)_{A},(r)_{R} )_{(e_0,S_0)}$ is a {\bf {\itshape total order}}.
\end{definition}

% Definition of lower bound/upper bound
\begin{definition}
Given 
\begin{itemize}
	\item total order $( (a)_{A},(r)_{R} )_{(e_0,S_0)}$
	\item subset $(b)_{B}$ of $(a)_{A}$ over $(e_0,S_0)$
	\item element $x$ of $(a)_{A}$ over $(e_0,S_0)$
\end{itemize}
if for all $b_1\overset{}{\underset{ (e_0,S_0)} {\in } }(b)_{B}$ we have $x\overset{(r)_{R}}{\underset{ (e_0,S_0)} {\leq } }b_1$ ({\itshape resp. $x\overset{(r)_{R}}{\underset{ (e_0,S_0)} {\geq } }b_1$}) 
then we say that $x$ is a {\bf {\itshape lower bound of $(b)_{B}$ over $( (a)_{A},(r)_{R} )_{(e_0,S_0)}$}} ( {\itshape resp. {\bf upper bound of $(b)_{B}$ over $( (a)_{A},(r)_{R} )_{(e_0,S_0)}$ }}).
\end{definition} 

% Greatest lower bound and Least upper bound *** USE THIS DEFN ***
\begin{definition}
Suppose we are given 
\begin{itemize}
	\item total order $( (a)_{A},(r)_{R} )_{(e_0,S_0)}$
	\item subset $(b)_{B}$ of $(a)_{A}$ over $(e_0,S_0)$
	\item element $x$ of $(a)_{A}$ over $(e_0,S_0)$
\end{itemize}

If $x$ is a lower bound of $(b)_{B}$ over $( (a)_{A},(r)_{R} )_{(e_0,S_0)}$ and for every lower bound $x_1$ of $(b)_{B}$ over $( (a)_{A},(r)_{R} )_{(e_0,S_0)}$ we have $x\overset{(r)_{R}}{\underset{ (e_0,S_0)} {\geq } }x_1$ then we say that $x$ is the {\bf {\itshape greatest lower bound of $(b)_{B}$ over $( (a)_{A},(r)_{R} )_{(e_0,S_0)}$} } or write $x\overset{(r)_{R}}{\underset{ (e_0,S_0)} {= } }inf (b)_{B}$. 

If $x$ is an upper bound of $(b)_{B}$ over $( (a)_{A},(r)_{R} )_{(e_0,S_0)}$ and that for every upper bound $x_1$ of $(b)_{B}$ over $( (a)_{A},(r)_{R} )_{(e_0,S_0)}$ we have $x\overset{(r)_{R}}{\underset{ (e_0,S_0)} {\leq } }x_1$ then we say that $x$ is the {\bf {\itshape least upper bound of $(b)_{B}$ over $( (a)_{A},(r)_{R} )_{(e_0,S_0)}$ }} or write $x\overset{(r)_{R}}{\underset{ (e_0,S_0)} {= } }sup (b)_{B}$.
\end{definition}

% Definition of well ordered
\begin{definition}
Given a total order $( (a)_{A},(r)_{R} )_{(e_0,S_0)}$, if every non-empty subset of $(a)_{A}$ over $(e_0,S_0)$ contains its greatest lower bound over $( (a)_{A},(r)_{R} )_{(e_0,S_0)}$  then we say that $( (a)_{A},(r)_{R} )_{(e_0,S_0)}$ is a {\bf {\itshape well ordering}}.
\end{definition}

% Order morphism
\begin{definition}
Given 
\begin{itemize}
	\item partial orders $( (a_1)_{A_1},(r_1)_{R_1} )_{(e_0,S_0)}$ and $( (a_2)_{A_2},(r_2)_{R_2})_{(e_0,S_0)}$
	\item function $(f)_{F}:(a_1)_{A_1}\overset{}{\underset{ (e_0,S_0)} {\rightarrow } }(a_2)_{A_2}$
\end{itemize}
if for any $x_1$ and any $x_2$ we have $x_1\overset{(r_1)_{R_1}}{\underset{ (e_0,S_0)} {\leq } }x_2 \Leftrightarrow (f)_{F}(x_1)\overset{(r_2)_{R_2}}{\underset{ (e_0,S_0)} {\leq } }(f)_{F}(x_2)$ then we say that $(f)_{F}$ is an {\bf {\itshape order embedding of $( (a_1)_{A_1},(r_1)_{R_1} )_{(e_0,S_0)} $ into $( (a_2)_{A_2},(r_2)_{R_2} )_{(e_0,S_0)}$ }}.
\end{definition}

% Order isomorphic
\begin{definition}
Given partial orders $( (a_1)_{A_1},(r_1)_{R_1} )_{(e_0,S_0)}$ and $( (a_2)_{A_2},(r_2)_{R_2})_{(e_0,S_0)}$, if there exists a bijective function $(f)_{F}:(a_1)_{A_1}\overset{}{\underset{ (e_0,S_0)} {\rightarrow } }(a_2)_{A_2}$ such that $(f)_{F}$ is an order embedding of $( (a_1)_{A_1},(r_1)_{R_1} )_{(e_0,S_0)}$ into $( (a_2)_{A_2},(r_2)_{R_2} )_{(e_0,S_0)}$ then we say that $( (a_1)_{A_1},(r_1)_{R_1} )_{(e_0,S_0)}$ is {\bf {\itshape order isomorphic}} to $( (a_2)_{A_2},(r_2)_{R_2})_{(e_0,S_0)}$.
\end{definition}

% Definition of equivalence relation 
\begin{definition}
Given a relation $( (a)_{A},(a)_{A},(r)_{R} )_{(e_0,S_0)}$, if for any $a_1\overset{}{\underset{ (e_0,S_0)} {\in } }(a)_{A}$, $a_2\overset{}{\underset{ (e_0,S_0)} {\in } }(a)_{A}$ and $a_3\overset{}{\underset{ (e_0,S_0)} {\in } }(a)_{A}$ we have
\begin{itemize}
	% reflexivity
	\item $(a_1,a_1)\overset{}{\underset{ (e_0,S_0)} {\in } }(r)_{R}$
	% symmetry
	\item $(a_1,a_2)\overset{}{\underset{ (e_0,S_0)} {\in } }(r)_{R} \Rightarrow (a_2,a_1)\overset{}{\underset{ (e_0,S_0)} {\in } }(r)_{R}$
	% transitivity
	\item $((a_1,a_2)\overset{}{\underset{ (e_0,S_0)} {\in } }(r)_{R} \wedge (a_2,a_3)\overset{}{\underset{ (e_0,S_0)} {\in } }(r)_{R}) \Rightarrow (a_1,a_3)\overset{}{\underset{ (e_0,S_0)} {\in } }(r)_{R}$
\end{itemize}
then we say that $( (a)_{A},(r)_{R} )_{(e_0,S_0)}$ is a an {\bf {\itshape equivalence relation}}. 
\end{definition}

Let $((a)_{A},(r)_{R} )_{(e_0,S_0)}$ be an equivalence relation. For any $a_1$ and $a_2$, if $(a_1,a_2)\overset{}{\underset{ (e_0,S_0)} {\in } }(r)_{R}$ then we say that $a_1$ is {\bf {\itshape equivalent}} to $a_2$ over $((a)_{A},(r)_{R} )_{(e_0,S_0)}$ or write $a_1\overset{(r)_{R}}{\underset{ (e_0,S_0)} {\sim } }a_2$.

%%%%%% TO MOVE %%%%%%%%%%%

%%%%%%% / TO MOVE %%%%%%%%%%%%%%

\subsection{Integers and rational numbers}

% Definition of integers :

\begin{definition}
Given a set $(z)_{Z}$ over $(e_0,S_0)$, if
\begin{itemize}
	\item $(z)_{Z}\overset{}{\underset{ (e_0,S_0)} {\subseteq } }\mathbb{N} \times \mathbb{N}$
	\item $\forall n_1,n_2,n_3,n_4(((n_1,n_2)\overset{}{\underset{(e_0,S_0) } {\in } }(z)_{Z}\wedge (n_3,n_4)\in (z)_{Z})\Rightarrow n_1+n_4=n_2+n_3)$
	\item $\forall n_1,n_2,n_3,n_4( ( (n_1,n_2)\overset{}{\underset{ (e_0,S_0)} {\in } }(z)_{Z} \wedge n_1+n_4=n_2+n_3 ) \Rightarrow (n_3,n_4)\overset{}{\underset{ (e_0,S_0)} {\in } }(z)_{Z} )$
\end{itemize}
then we say that $(z)_{Z}$ is an {\bf {\itshape integer over}} $(e_0,S_0)$. 
\end{definition}

\begin{definition}
Let $(z_1)_{Z_1}$ and $(z_2)_{Z_2}$ be integers over $(e_0,S_0)$. If
$$ \forall n_1,n_2 ( (n_1,n_2)\overset{}{\underset{(e_0,S_0) } {\in } }(z_1)_{Z_1}\Leftrightarrow (n_1,n_2)\overset{}{\underset{(e_0,S_0) } {\in } }(z_2)_{Z_2})$$
then we say that $(z_1)_{Z_1}$ and $(z_2)_{Z_2}$ are {\bf {\itshape equal as integers over $(e_0,S_0)$}} and write $(z_1)_{Z_1}\overset{\mathbb{Z}}{\underset{(e_0,S_0) } {=} }(z_2)_{Z_2}$. Otherwise we write $(z_1)_{Z_1}\overset{\mathbb{Z}}{\underset{(e_0,S_0) } {\neq } }(z_2)_{Z_2}$.
\end{definition}

To lighten the discourse we will treat $\overset{\mathbb{Z}}{\underset{(e_0,S_0) } {=} }$ as an equality symbol when there is no ambiguity. In particular, the symbol $\overset{\mathbb{Z}}{\underset{(e_0,S_0) } {=} }$ will convey the substitutive property of equality in logical expressions when there is no ambiguity. Generally, any ``decorated'' equality symbol (in this paper) will be treated like an equality symbol (when there is no ambiguity).

% less than or equal for integers

Given integers $(z_1)_{Z_1}$ and $(z_2)_{Z_2}$ over $(e_0,S_0)$, if for any natural numbers $n_1$,$n_2$ and $n_3$ we have $( (n_1,n_2)\overset{}{\underset{ (e_0,S_0)} {\in } }(z_1)_{Z_1} \wedge (n_1,n_3)\overset{}{\underset{ (e_0,S_0)} {\in } }(z_2)_{Z_2} ) \Rightarrow n_3<n_2$ then we write $(z_1)_{Z_1}\overset{\mathbb{Z}}{\underset{ (e_0,S_0)} {< } }(z_2)_{Z_2}$. If $(z_1)_{Z_1}\overset{\mathbb{Z}}{\underset{ (e_0,S_0)} {< } }(z_2)_{Z_2}$ or $(z_1)_{Z_1}\overset{\mathbb{Z}}{\underset{ (e_0,S_0)} {= } }(z_2)_{Z_2}$ then we write $(z_1)_{Z_1}\overset{\mathbb{Z}}{\underset{ (e_0,S_0)} {\leq } }(z_2)_{Z_2}$.

% integer set
\begin{definition}
Given a set $(z_Z)$ over $(e_0,S_0)$, if
\begin{itemize}
	\item for any $x\overset{}{\underset{(e_0,S_0) } {\in } }(z)_{Z}$, $(x)_{Z}$ is an integer over $(e_0,S_0)$
	\item $\forall x_1\overset{}{\underset{ (e_0,S_0)} {\in } }(z)_{Z} \forall x_2\overset{}{\underset{ (e_0,S_0)} {\in } }(z)_{Z}( (x_1)_{Z}\overset{\mathbb{Z}}{\underset{ (e_0,S_0)} {= } }(x_2)_{Z} \Leftrightarrow x_1=x_2 )$
\end{itemize}
then we say that $(z)_{Z}$ is an {\bf {\itshape integer set over $(e_0,S_0)$}}.
\end{definition}

Let $(i)_{I}$ and $(z)_{Z}$ be integer sets over $(e_0,S_0)$. If for every $x\overset{}{\underset{ (e_0,S_0)} {\in } }(i)_{I}$ there exists $y\overset{}{\underset{ (e_0,S_0)} {\in } }(z)_{Z}$ such that $(x)_{I}\overset{\mathbb{Z}}{\underset{ (e_0,S_0)} {= } }(y)_{Z}$ then we say that $(i)_{I}$ is an {\bf {\itshape integer subset of $(z)_{Z}$ over $(e_0,S_0)$}} or write $(i)_{I}\overset{\{\mathbb{Z}\}}{\underset{ (e_0,S_0)} {\subseteq } }(z)_{Z}$. If $(i)_{I}\overset{\{\mathbb{Z}\}}{\underset{ (e_0,S_0)} {\subseteq } }(z)_{Z}$ and $(z)_{Z}\overset{\{ \mathbb{Z} \}}{\underset{ (e_0,S_0)} {\subseteq } }(i)_{I}$ then we say $(i)_{I}$ and $(z)_{Z}$ are {\bf {\itshape equal as integer sets over $(e_0,S_0)$}} or write $(i)_{I}\overset{ \{ \mathbb{Z} \} }{\underset{ (e_0,S_0)} {= } }(z)_{Z}$.

% set of integers

\begin{definition}
Given an integer set $(z)_{Z}$ over $(e_0,S_0)$, if for every natural numbers $n_1$ and $n_2$ there exists $x\overset{}{\underset{(e_0,S_0) } {\in } }(z)_{Z}$ such that $(n_1,n_2)\overset{}{\underset{(e_0,S_0) } {\in } }(x)_{Z}$
then we say that $(z)_{Z}$ {\bf {\itshape represents the set of integers over $(e_0,S_0)$}} or write $(z)_{Z}\overset{}{\underset{(e_0,S_0) } {= } }\mathbb{Z}$
\end{definition}

We can construct a set that represents the set of integers $(e_0,S_0)$ in the following way.

%%%% CONSTRUCTION OF Z %%%%%%%%%%%%%%%%%%%%%%%%%%%%
Let $x\in S_0$ and $y\in S_0$ with $x\neq y$. Let $(x)_{X}$ be a set over $(e_0,S_0)$ such that

\begin{itemize}

	% y(n_1\otimes x)(n_2\otimes y) is an integer representative
	\item for any word $a$, $a\overset{}{\underset{ (e_0,S_0)} {\in } }(x)_{X}$ if and only if there exist natural numbers $n_1$ and $n_2$ with $a=y(n_1\otimes x)(n_2\otimes y)$ such that $n_1=1$ or $n_2=1$
	  
	% (n_3\otimes x)(n_4\otimes y) is a an element of integer representatives  
  \item for any natural numbers $n_1$ and $n_2$, if $y(n_1\otimes x)(n_2\otimes y)\overset{}{\underset{ (e_0,S_0)} {\in } }(x)_{X}$ then for any $b$ we have that $b\overset{}{\underset{ (e_0,S_0)} {\in } }(y(n_1\otimes x)(n_2\otimes y))_{X}$ if and only if there exist natural numbers $n_3$ and $n_4$ such that $n_1 + n_4 = n_2 + n_3$, $b=(n_3 \otimes x)(n_4 \otimes y)$ and $(b)_{X}\overset{}{\underset{ (e_0,S_0)} {= } }(n_3,n_4)$

\end{itemize}

We see that $(x)_{X}$ represents the set of integers over $(e_0,S_0)$.

%%%%%%%%%%%%%%%%%%%%%%%%%%%%%%%%%%%%%%%%%%%%%%%%%%%

Whenever some $(z)_{Z}$ is an integer over $(e_0,S_0)$ we will just write $(z)_{Z}\overset{}{\underset{(e_0,S_0) } {\in } }\mathbb{Z}$.

% addition
Given integers $(z_1)_{Z_1}$, $(z_2)_{Z_2}$ and $(z_3)_{Z_3}$ over $(e_0,S_0)$, if 
$$ \exists n_1,n_2,n_3,n_4( (n_1,n_2)\overset{}{\underset{ (e_0,S_0) } {\in } }(z_1)_{Z_1} \wedge (n_3,n_4)\overset{}{\underset{(e_0,S_0) } {\in } }(z_2)_{Z_2} \wedge (n_1+n_3,n_2+n_4)\overset{}{\underset{(e_0,S_0) } {\in } }(z_3)_{Z_3})$$
then we write $$(z_1)_{Z_1}+(z_2)_{Z_2}\overset{\mathbb{Z}}{\underset{(e_0,S_0) } {= } }(z_3)_{Z_3}$$

% zero element

Given integer $(z_1)_{Z_1}$ over $(e_0,S_0)$, if for any integer $(z_2)_{Z_2}$ over $(e_0,S_0)$ we have $(z_1)_{Z_1} + (z_2)_{Z_2} \overset{\mathbb{Z}}{\underset{(e_0,S_0) } {= } }(z_2)_{Z_2}$ then we write 
$$ (z_1)_{Z_1} \overset{\mathbb{Z}}{\underset{(e_0,S_0) } {= } }0$$

% additive inverse

Let $(z_1)_{Z_1}$, $(z_2)_{Z_2}$ and $(z_3)_{Z_3}$ be integers over $(e_0,S_0)$, if $(z_1)_{Z_1}+(z_2)_{Z_2}\overset{\mathbb{Z}}{\underset{(e_0,S_0) } {= } }(z_3)_{Z_3}$ and $(z_3)_{Z_3}\overset{\mathbb{Z}}{\underset{(e_0,S_0) } {= } }0$ then we say that $(z_2)_{Z_2}$ is the {\bf {\itshape additive integer inverse of $(z_1)_{Z_1}$ over $(e_0,S_0)$}} or we write 
$$ (z_2)_{Z_2}\overset{\mathbb{Z}}{\underset{(e_0,S_0) } {= } }-(z_1)_{Z_1}$$

% multiplication

Given integers $(z_1)_{Z_1}$, $(z_2)_{Z_2}$ and $(z_3)_{Z_3}$ all over $(e_0,S_0)$, if 
$ \exists n_1,n_2,n_3,n_4( (n_1,n_2)\overset{}{\underset{(e_0,S_0) } {\in } }(z_1)_{Z_1} \wedge (n_3,n_4)\overset{}{\underset{(e_0,S_0) } {\in } }(z_2)_{Z_2} \wedge (n_1\cdot n_3 + n_2\cdot n_4,n_1\cdot n_4 + n_2\cdot n_3)\overset{}{\underset{(e_0,S_0) } {\in } }(z_3)_{Z_3})$
then we write 
$$ (z_1)_{Z_1} \cdot (z_2)_{Z_2} \overset{\mathbb{Z}}{\underset{(e_0,S_0) } {= } }(z_3)_{Z_3} \text{ or } (z_1)_{Z_1}(z_2)_{Z_2} \overset{\mathbb{Z}}{\underset{(e_0,S_0) } {= } }(z_3)_{Z_3}$$

% identity element

Given integer $(z_1)_{Z_1}$ over $(e_0,S_0)$, if for any integer $(z_2)_{Z_2}$ over $(e_0,S_0)$ we have  $(z_1)_{Z_1} \cdot (z_2)_{Z_2} \overset{\mathbb{Z}}{\underset{(e_0,S_0) } {= } }(z_2)_{Z_2}$ then we write 
$$ (z_1)_{Z_1} \overset{\mathbb{Z}}{\underset{(e_0,S_0) } {= } }1$$

% multiplicative inverse

Given integers $(z_1)_{Z_1}$, $(z_2)_{Z_2}$ and $(z_3)_{Z_3}$ over $(e_0,S_0)$, if $(z_1)_{Z_1} \cdot (z_2)_{Z_2}\overset{\mathbb{Z}}{\underset{(e_0,S_0) } {= } }(z_3)_{Z_3}$ and $(z_3)_{Z_3}\overset{\mathbb{Z}}{\underset{(e_0,S_0) } {= } }1$ then we say that $(z_2)_{Z_2}$ is the {\bf {\itshape multiplicative integer inverse}} of $(z_1)_{Z_1}$ over $(e_0,S_0)$ or we write 
$$ (z_2)_{Z_2}\overset{\mathbb{Z}}{\underset{(e_0,S_0) } {= } }(z_1)^{-1}_{Z_1}$$

% Rational number

% Definition:
\begin{definition} Given a set $(q)_{Q}$ over $(e_0,S_0)$ we say that $(q)_{Q}$ is a {\bf {\itshape rational number over}} $(e_0,S_0)$ if
\begin{itemize}
\item for any $x\overset{}{\underset{(e_0,S_0) } {\in } }(q)_{Q}$ there exist integers $(z_1)_{Q}$ and $(z_2)_{Q}\overset{\mathbb{Z}}{\underset{ (e_0,S_0)} {\neq } }0$ over $(e_0,S_0)$ such that $(x)_{Q}\overset{}{\underset{(e_0,S_0) } {= } }(z_1,z_2)$
\item $\forall x_1\overset{}{\underset{(e_0,S_0) } {\in } }(q)_{Q} \forall x_2\overset{}{\underset{(e_0,S_0) } {\in } }(q)_{Q} \forall z_1,z_2,z_3,z_4 ( (x_1\overset{}{\underset{ (e_0,S_0)} {= } }(z_1,z_2) \wedge x_2\overset{}{\underset{ (e_0,S_0)} {= } }(z_3,z_4) \wedge (z_1)_{Q}\overset{\mathbb{Z}}{\underset{ (e_0,S_0)} {= } }(z_3)_{Q} \wedge (z_2)_{Q}\overset{\mathbb{Z}}{\underset{ (e_0,S_0)} {= } }(z_4)_{Q}) \Rightarrow x_1=x_2) $ 
\item $\forall z_1,z_2,z_3,z_4(((z_1,z_2)\overset{}{\underset{(e_0,S_0) } {\in } }(q)_{Q}\wedge (z_3,z_4)\in (q)_{Q})\Rightarrow (z_1)_{Q} \cdot (z_4)_{Q} \overset{\mathbb{Z}}{\underset{(e_0,S_0) } {= } }(z_2)_{Q} \cdot (z_3)_{Q})$
\item $\forall (z_1)_{Z_1}\overset{}{\underset{ (e_0,S_0)} {\in } }\mathbb{Z} \forall (z_2)_{Z_2}\overset{}{\underset{ (e_0,S_0)} {\in } }\mathbb{Z} ( ( z_2\overset{\mathbb{Z}}{\underset{ (e_0,S_0)} {\neq } }0 \wedge \exists z_3 \exists z_4( (z_3,z_4)\overset{}{\underset{ (e_0,S_0)} { \in} }(q)_{Q} \wedge (z_1)_{Z_1}(z_4)_{Q}\overset{\mathbb{Z}}{\underset{ (e_0,S_0)} {= } }(z_2)_{Z_2}(z_3)_{Q} ) ) \Rightarrow \exists z_5 \exists z_6((z_5)_{Q}\overset{\mathbb{Z}}{\underset{ (e_0,S_0)} {= } }(z_1)_{Z_1} \wedge (z_6)_{Q}\overset{\mathbb{Z}}{\underset{ (e_0,S_0)} {= } }(z_2)_{Z_2} \wedge (z_5,z_6)\overset{}{\underset{ (e_0,S_0)} {\in } }(z)_{Q}))$
\end{itemize}
\end{definition}

\begin{definition}
Let $(q_1)_{Q_1}$ and $(q_2)_{Q_2}$ be rational numbers over $(e_0,S_0)$. If
\begin{itemize}
\item $\forall z_1,z_2 ( (z_1,z_2)\overset{}{\underset{(e_0,S_0) } {\in } }(q_1)_{Q_1}\Rightarrow \exists z_3,z_4( (z_3)_{Q_2}\overset{\mathbb{Z}}{\underset{(e_0,S_0) } {= } }(z_1)_{Q_1} \wedge (z_4)_{Q_2}\overset{\mathbb{Z}}{\underset{(e_0,S_0) } {= } }(z_2)_{Q_1} \wedge (z_3,z_4)\overset{}{\underset{(e_0,S_0) } {\in } }(q_2)_{Q_2} ) )$ 
\item $\forall z_1,z_2 ( (z_1,z_2)\overset{}{\underset{(e_0,S_0) } {\in } }(q_2)_{Q_2}\Rightarrow \exists z_3,z_4( (z_3)_{Q_1}\overset{\mathbb{Z}}{\underset{(e_0,S_0) } {= } }(z_1)_{Q_2} \wedge (z_4)_{Q_1}\overset{\mathbb{Z}}{\underset{(e_0,S_0) } {=} }(z_2)_{Q_2} \wedge (z_3,z_4)\overset{}{\underset{(e_0,S_0) } {\in } }(q_1)_{Q_1} ) )$
\end{itemize}
then we say that $(q_1)_{Q_1}$ and $(q_2)_{Q_2}$ are {\bf {\itshape equal as rationals over}} $(e_0,S_0)$ or write $(q_1)_{Q_1}\overset{\mathbb{Q}}{\underset{(e_0,S_0) } {=} }(q_2)_{Q_2}$. Otherwise we write $(q_1)_{Q_1}\overset{\mathbb{Q}}{\underset{(e_0,S_0) } {\neq } }(q_2)_{Q_2}$.
\end{definition}

% 
% less than or equal for rational numbers

For any rational numbers $(q_1)_{Q_1}$ and $(q_2)_{Q_2}$ over $(e_0,S_0)$, if for any integers $(z_1)_{Q_1}$, $(z_2)_{Q_1}$, $(z_3)_{Q_2}$ and $(z_4)_{Q_2}$ over $(e_0,S_0)$ we have $( (z_1,z_2)\overset{}{\underset{ (e_0,S_0)} {\in } }(q_1)_{Q_1} \wedge (z_3,z_4)\overset{}{\underset{ (e_0,S_0)} {\in } }(q_2)_{Q_2} \wedge (z_4)_{Q_2}\overset{\mathbb{Z}}{\underset{ (e_0,S_0)} {= } }(z_2)_{Q_1} ) \Rightarrow (z_1)_{Q_1}\overset{\mathbb{Z}}{\underset{ (e_0,S_0)} {< } }(z_3)_{Q_2}$ then we write $(q_1)_{Q_1}\overset{\mathbb{Q}}{\underset{ (e_0,S_0)} {< } }(q_2)_{Q_2}$. Given rational numbers $(q_1)_{Q_1}$ and $(q_2)_{Q_2}$ over $(e_0,S_0)$, if $(q_1)_{Q_1}\overset{\mathbb{Q}}{\underset{ (e_0,S_0)} {< } }(q_2)_{Q_2}$ or $(q_1)_{Q_1}\overset{\mathbb{Q}}{\underset{ (e_0,S_0)} {=} }(q_2)_{Q_2}$ then we write $(q_1)_{Q_1}\overset{\mathbb{Q}}{\underset{ (e_0,S_0)} {\leq } }(q_2)_{Q_2}$.

% rational set
\begin{definition}
Given a set $(q)_{Q}$ over $(e_0,S_0)$, if
\begin{itemize}
	\item for any $x\overset{}{\underset{(e_0,S_0) } {\in } }(q)_{Q}$, $(x)_{Q}$ is a rational number over $(e_0,S_0)$
	\item $\forall x_1\overset{}{\underset{ (e_0,S_0)} {\in } }(q)_{Q} \forall x_2\overset{}{\underset{ (e_0,S_0)} {\in } }(q)_{Q}( (x_1)_{Q}\overset{\mathbb{Q}}{\underset{ (e_0,S_0)} {= } }(x_2)_{Q} \Leftrightarrow x_1=x_2 )$
\end{itemize}
then we say that $(q)_{Q}$ is a {\bf {\itshape rational set over $(e_0,S_0)$}}.
\end{definition}

Let $(q)_{Q}$ and $(r)_{R}$ be rational sets over $(e_0,S_0)$. If for every $x\overset{}{\underset{ (e_0,S_0)} {\in } }(q)_{Q}$ there exists $y\overset{}{\underset{ (e_0,S_0)} {\in } }(r)_{R}$ such that $(x)_{Q}\overset{\mathbb{Q}}{\underset{ (e_0,S_0)} {= } }(y)_{R}$ then we say that $(q)_{Q}$ is a {\bf {\itshape rational subset of $(r)_{R}$ over $(e_0,S_0)$}} or write $(q)_{Q}\overset{\{\mathbb{Q}\}}{\underset{ (e_0,S_0)} {\subseteq } }(r)_{R}$. If $(q)_{Q}\overset{\{\mathbb{Q}\}}{\underset{ (e_0,S_0)} {\subseteq } }(r)_{R}$ and $(r)_{R}\overset{\{\mathbb{Q}\}}{\underset{ (e_0,S_0)} {\subseteq } }(q)_{Q}$ then we say that $(q)_{Q}$ and $(r)_{R}$ are {\bf {\itshape equal as rational sets over $(e_0,S_0)$}} or write $(q)_{Q}\overset{\{\mathbb{Q}\}}{\underset{ (e_0,S_0)} {=} }(r)_{R}$. 

% set of rational numbers
\begin{definition}
Given a rational set $(q)_{Q}$ over $(e_0,S_0)$, if for any rational number $(q_1)_{Q_1}$ over $(e_0,S_0)$ there exists $x\overset{}{\underset{ (e_0,S_0)} {\in } }(q)_{Q}$ such that $(x)_{Q}\overset{}{\underset{(e_0,S_0) } {= } }(q_1)_{Q_1}$ then we say $(q)_{Q}$ {\bf {\itshape represents the set of rational numbers}} over $(e_0,S_0)$ or write $(q)_{Q}\overset{}{\underset{(e_0,S_0) } {= } }\mathbb{Q}$.

\end{definition}

%%%%%%%%%%%%% CONSTRUCTION OF Q %%%%%%%%%%%%%%%%%%%
We can construct a set that represents the set of rational numbers $(e_0,S_0)$ in the following way.

Let $x\in S_0$ and $y\in S_0$ with $x\neq y$. Let $(x)_{X}$ be a set over $(e_0,S_0)$ such that
\begin{itemize}

	% form of elements in (x)_{X} (representative of rational)
	\item for any $a$, $a\overset{}{\underset{ (e_0,S_0)} {\in } }(x)_{X}$ if and only if there exist natural numbers $n_1$, $n_2$, $n_3$ and $n_4$ with $a=y(n_1\otimes x)(n_2\otimes y)(n_3\otimes x)(n_4\otimes y)$ such that
		\begin{itemize}
			
			% Constraints on integer parts
			\item $n_1=1$ or $n_2=1$
			\item $n_1\geq n_2$
			\item $n_3=1$ or $n_4=1$ 
			\item $n_3\neq n_4$
			
			% Constraint on rational representation
			
			\item $\forall n_5, n_6, n_7, n_8 \in \mathbb{N} \forall (z_1)_{Z_1}\overset{}{\underset{ (e_0,S_0)} {\in } }\mathbb{Z} \forall (z_2)_{Z_2}\overset{}{\underset{ (e_0,S_0)} {\in } }\mathbb{Z} \forall  (z_3)_{Z_3}\overset{}{\underset{ (e_0,S_0)} {\in } }\mathbb{Z} \forall (z_4)_{Z_4}\overset{}{\underset{ (e_0,S_0)} {\in } }\mathbb{Z}( (n_5\geq n_6 \wedge n_7\neq n_8 \wedge 
			(n_1,n_2)\overset{}{\underset{ (e_0,S_0)} {\in } }(z_1)_{Z_1} \wedge (n_3,n_4)\overset{}{\underset{ (e_0,S_0)} {\in } }(z_2)_{Z_2} \wedge (n_5,n_6)\overset{}{\underset{ (e_0,S_0)} {\in } }(z_3)_{Z_3} \wedge
			(n_7,n_8)\overset{}{\underset{ (e_0,S_0)} {\in } }(z_4)_{Z_4} \wedge 
			(z_1)_{Z_1} \cdot (z_4)_{Z_4} \overset{\mathbb{Z}}{\underset{(e_0,S_0) } {= } }(z_2)_{Z_2} \cdot (z_3)_{Z_3}
			) \Rightarrow 
			(z_1)_{Z_1}\overset{\mathbb{Z}}{\underset{ (e_0,S_0)} {\leq } }(z_3)_{Z_3}
			)
			$
			
		\end{itemize}
	% Form of elements of rational representations (integers stay standardized - rationals are free)
	% (n_5\otimes x)(n_6\otimes y)(n_7\otimes x)(n_8\otimes y) is integer stand but rat free
	\item for any natural numbers $n_1$,$n_2$,$n_3$ and $n_4$, if $y(n_1\otimes x)(n_2\otimes y)(n_3\otimes x)(n_4\otimes x)\overset{}{\underset{ (e_0,S_0)} {\in } }(x)_{X}$ then for any $b$ we have that $b\overset{}{\underset{ (e_0,S_0)} {\in } }(y(n_1\otimes x)(n_2\otimes y)(n_3\otimes x)(n_4\otimes x))_{X}$ if and only if there exist natural numbers $n_5$, $n_6$, $n_7$ and $n_8$ with $b=(n_5\otimes x)(n_6\otimes y)(n_7\otimes x)(n_8\otimes y)$ such that 
	
	\begin{itemize}
		\item $n_5=1$ or $n_6=1$
		\item $n_7=1$ or $n_8=1$
		\item $n_7\neq n_8$
		\item for any integers $(z_1)_{Z_1}$,$(z_2)_{Z_2}$,$(z_3)_{Z_3}$ and $(z_4)_{Z_4}$ over $(e_0,S_0)$ with $(n_1,n_2)\overset{}{\underset{ (e_0,S_0)} {\in } }(z_1)_{Z_1}$, $(n_3,n_4)\overset{}{\underset{ (e_0,S_0)} {\in } }(z_2)_{Z_2}$, $(n_5,n_6)\overset{}{\underset{ (e_0,S_0)} {\in } }(z_3)_{Z_3}$ and 
			$(n_7,n_8)\overset{}{\underset{ (e_0,S_0)} {\in } }(z_4)_{Z_4}$ we have $(z_1)_{Z_1} \cdot (z_4)_{Z_4} \overset{\mathbb{Z}}{\underset{(e_0,S_0) } {= } }(z_2)_{Z_2} \cdot (z_3)_{Z_3}$
	\end{itemize}
	
	% integers as part of rationals (int is stand, rat not stand)
	\item for any natural numbers $n_1$,$n_2$,$n_3$ and $n_4$, if $n_1=1$ or $n_2=1$ and if $n_3=1$ or $n_4=1$ then 
	$( (n_1\otimes x)(n_2\otimes y)(n_3\otimes x)(n_4\otimes y) )_{X}\overset{}{\underset{ (e_0,S_0)} {= } }( y(n_1\otimes x)(n_2\otimes y),y(n_3\otimes x)(n_4\otimes y) )_{X}$

	% natural number as parts of integers
	\item for any natural numbers $n_1$ and $n_2$, if $n_1=1$ or $n_2=1$ then for any $c$ we have that  $c\overset{}{\underset{ (e_0,S_0)} {\in } }(y(n_1\otimes x)(n_2\otimes y))_{X}$ if and only if there exist natural numbers $n_3$ and $n_4$ such that $n_1 + n_4 = n_2 + n_3$, $c=(n_3\otimes x)(n_4\otimes y)$ and $(c)_{X}\overset{}{\underset{ (e_0,S_0)} {= } }(n_3, n_4)$	

\end{itemize}
We see that $(x)_{X}\overset{}{\underset{ (e_0,S_0)} {= } }\mathbb{Q}$.

%%%%%%%%%%%%%%%%%%%%%%%%%%%%%%%%%%%%%%%%%%%%%%%%%%%

Whenever some $(q_1)_{Q_1}$ is a rational number over $(e_0,S_0)$ we will just write $(q_1)_{Q_1}\overset{}{\underset{(e_0,S_0) } {\in } }\mathbb{Q}$.

% addition
Given rational numbers $(q_1)_{Q_1}$, $(q_2)_{Q_2}$ and $(q_3)_{Q_3}$ all over $(e_0,S_0)$, if 
$$ \exists z_1,z_2,z_3,z_4,z_5,z_6( (z_1,z_2)\overset{}{\underset{ (e_0,S_0) } {\in } }(q_1)_{Q_1} \wedge (z_3,z_4)\overset{}{\underset{(e_0,S_0) } {\in } }(q_2)_{Q_2} \wedge (z_5,z_6)\overset{}{\underset{(e_0,S_0) } {\in } }(q_3)_{Q_3} \wedge $$
$$
(z_5)_{Q_3}\overset{\mathbb{Z}}{\underset{(e_0,S_0) } {= } }(z_1)_{Q_1}\cdot (z_4)_{Q_2}+(z_2)_{Q_1}\cdot (z_3)_{Q_2} \wedge (z_6)_{Q_3}\overset{\mathbb{Z}}{\underset{(e_0,S_0) } {= } }(z_2)_{Q_1}\cdot (z_4)_{Q_2} )$$
then we write $$(q_1)_{Q_1}+(q_2)_{Q_2}\overset{\mathbb{Q}}{\underset{(e_0,S_0) } {= } }(q_3)_{Q_3}$$

% zero element

Given rational number $(q_1)_{Q_1}$ over $(e_0,S_0)$, if for any integer $(q_2)_{Q_2}$ over $(e_0,S_0)$ we have $(q_1)_{Q_1} + (q_2)_{Q_2} \overset{\mathbb{Q}}{\underset{(e_0,S_0) } {= } }(q_2)_{Q_2}$ then we write 
$$ (q_1)_{Q_1} \overset{\mathbb{Q}}{\underset{(e_0,S_0) } {= } }0$$ 

% additive inverse

Given rational numbers $(q_1)_{Q_1}$, $(q_2)_{Q_2}$ and $(q_3)_{Q_3}$ all over $(e_0,S_0)$, if $(q_1)_{Q_1}+(q_2)_{Q_2}\overset{\mathbb{Q}}{\underset{(e_0,S_0) } {= } }(q_3)_{Q_3}$ and $(q_3)_{Q_3}\overset{\mathbb{Q}}{\underset{(e_0,S_0) } {= } }0$ then we say that $(q_2)_{Q_2}$ is the {\bf {\itshape additive rational inverse}} of $(q_1)_{Q_1}$ over $(e_0,S_0)$ or we write 
$$ (q_2)_{Q_2}\overset{\mathbb{Q}}{\underset{(e_0,S_0) } {= } }-(q_1)_{Q_1}$$

% multiplication

Given rational numbers $(q_1)_{Q_1}$, $(q_2)_{Q_2}$ and $(q_3)_{Q_3}$ all over $(e_0,S_0)$, if 
$ \exists z_1,z_2,z_3,z_4,z_5,z_6( \\ (z_1,z_2)\overset{}{\underset{(e_0,S_0) } {\in } }(q_1)_{Q_1} \wedge (z_3,z_4)\overset{}{\underset{(e_0,S_0) } {\in } }(q_2)_{Q_2} \wedge (z_5)_{Q_3}\overset{\mathbb{Z}}{\underset{ (e_0,S_0)} {= } }(z_1)_{Q_1}(z_3)_{Q_2} \wedge
 (z_6)_{Q_3}\overset{\mathbb{Z}}{\underset{ (e_0,S_0)} {= } }(z_2)_{Q_1}(z_4)_{Q_2} \wedge 
(z_5,z_6)\overset{}{\underset{ (e_0,S_0)} {\in } }(q_3)_{Q_3} ))$
then we write 
$$ (q_1)_{Q_1} \cdot (q_2)_{Q_2} \overset{\mathbb{Q}}{\underset{(e_0,S_0) } {= } }(q_3)_{Q_3} \text{ or } (q_1)_{Q_1}(q_2)_{Q_2} \overset{\mathbb{Q}}{\underset{(e_0,S_0) } {= } }(q_3)_{Q_3}$$

% identity element

Given rational number $(q_1)_{Q_1}$ over $(e_0,S_0)$, if for any rational number $(q_2)_{Q_2}$ over $(e_0,S_0)$ we have $(q_1)_{Q_1} \cdot (q_2)_{Q_2} \overset{\mathbb{Q}}{\underset{(e_0,S_0) } {= } }(q_2)_{Q_2}$ then we write 
$$ (q_1)_{Q_1} \overset{\mathbb{Q}}{\underset{(e_0,S_0) } {= } }1$$ 

% multiplicative inverse

Given rational numbers $(q_1)_{Q_1}$, $(q_2)_{Q_2}$ and $(q_3)_{Q_3}$ all over $(e_0,S_0)$, if $(q_1)_{Q_1} \cdot (q_2)_{Q_2}\overset{\mathbb{Q}}{\underset{(e_0,S_0) } {= } }(q_3)_{Q_3}$ and $(q_3)_{Q_3}\overset{\mathbb{Q}}{\underset{(e_0,S_0) } {= } }1$ then we say that $(q_2)_{Q_2}$ is the {\bf {\itshape multiplicative rational inverse}} of $(q_1)_{Q_1}$ over $(e_0,S_0)$ or we write 
$$ (q_2)_{Q_2}\overset{\mathbb{Q}}{\underset{(e_0,S_0) } {= } }(q_1)^{-1}_{Q_1}$$

\subsection{Real numbers}

\begin{definition}
Given a set $(r)_{R}$ over $(e_0,S_0)$, if there exist non-empty rational sets $(l)_{R}$ and $(u)_{R}$ over $(e_0,S_0)$ such that 
\begin{itemize}
	\item $(r)_{R}\overset{}{\underset{ (e_0,S_0)} {= } }\{l,u\}$
	\item $\forall x \forall y( (x\overset{}{\underset{ (e_0,S_0)} {\in } }(l)_{R} \wedge y\overset{}{\underset{ (e_0,S_0)} {\in } } (u)_{R})\Rightarrow (x)_{R}\overset{\mathbb{Q}}{\underset{ (e_0,S_0)} {<} }(y)_{R} )$
	\item $\forall (x)_{X}\overset{}{\underset{ (e_0,S_0)} {\in } }\mathbb{Q} \exists (y)_{R}\overset{\mathbb{Q}}{\underset{ (e_0,S_0)} {= } }(x)_{X}(y\overset{}{\underset{ (e_0,S_0)} {\in } }(l)_{R} \vee y\overset{}{\underset{ (e_0,S_0)} {\in } }(u)_{R} )$
	\item $\forall x\overset{}{\underset{ (e_0,S_0)} {\in } }(l)_{R} \exists y\overset{}{\underset{ (e_0,S_0)} {\in } }(l)_{R}((x)_{R}\overset{\mathbb{Q}}{\underset{ (e_0,S_0)} {< } }(y)_{R})$
\end{itemize}
then we say that $(r)_{R}$ is a {\bf {\itshape real number over $(e_0,S_0)$}}.

\end{definition}

If $(r)_{R}\overset{}{\underset{ (e_0,S_0)} {= } }\{ l,u \}$ is a real number over $(e_0,S_0)$ such that for any $x\overset{}{\underset{ (e_0,S_0)} {\in } }(l)_{R}$ and any $y\overset{}{\underset{ (e_0,S_0)} {\in } }(u)_{R}$ we have $(x)_{R}\overset{\mathbb{Q}}{\underset{ (e_0,S_0)} {< } }(y)_{R}$ then we say that $(l)_{R}$ is the {\bf {\itshape lower segment}} of $(r)_{R}$ over $(e_0,S_0)$ and $(u)_{R}$ is the {\bf {\itshape upper segment}} of $(r)_{R}$ over $(e_0,S_0)$.

% Real equality

\begin{definition}
Let $(r_1)_{R_1}$ and $(r_2)_{R_2}$ be real numbers over $(e_0,S_0)$ having respectively  $(l_1)_{R_1}$ and $(l_2)_{R_2}$ as their lower segments over $(e_0,S_0)$. If $(l_1)_{R_1}\overset{\{\mathbb{Q}\}}{\underset{ (e_0,S_0)} {= } }(l_2)_{R_2}$ then we say $(r_1)_{R_1}$ and $(r_2)_{R_2}$ are {\bf {\itshape equal as reals over $(e_0,S_0)$}} or write $(r_1)_{R_1}\overset{\mathbb{R}}{\underset{ (e_0,S_0)} {= } }(r_2)_{R_2}$, otherwise we write $(r_1)_{R_1}\overset{\mathbb{R}}{\underset{ (e_0,S_0)} {\neq } }(r_2)_{R_2}$.
\end{definition}

% less or equal for reals

\begin{definition}
Let $(r_1)_{R_1}$ and $(r_2)_{R_2}$ be real numbers over $(e_0,S_0)$ having respectively $(l_1)_{R_1}$ and $(l_2)_{R_2}$ as their lower segments over $(e_0,S_0)$. If $(l_1)_{R_1}\overset{\{\mathbb{Q}\}}{\underset{ (e_0,S_0)} {\subseteq } }(l_2)_{R_2}$ then we write $(r_1)_{R_1}\overset{\mathbb{R}}{\underset{ (e_0,S_0)} { \leq} }(r_2)_{R_2}$. If $(r_1)_{R_1}\overset{\mathbb{R}}{\underset{ (e_0,S_0)} {\leq } }(r_2)_{R_2}$ and $(r_1)_{R_1}\overset{\mathbb{R}}{\underset{ (e_0,S_0)} {\neq } }(r_2)_{R_2}$ then we write $(r_1)_{R_1}\overset{\mathbb{R}}{\underset{ (e_0,S_0)} {< } }(r_2)_{R_2}$.
\end{definition}

%%%%%%% TO MOVE %%%%%%%%%%%%%%%%

%%%%%%% / TO MOVE %%%%%%%%%%%%%%%%%%%

% Definition of real set
\begin{definition}
Given a set $(s)_{S}$ over $(e_0,S_0)$, if for every $e\overset{}{\underset{ (e_0,S_0)} {\in } }(s)_{S}$ we have that $(e)_{S}$ is a real number over $(e_0,S_0)$ and $\forall x \forall y((x)_{S}\overset{\mathbb{R}}{\underset{ (e_0,S_0)} {= } }(y)_{S} \Leftrightarrow x=y)$ then we say that $(s)_{S}$ is a {\bf {\itshape real set over $(e_0,S_0)$}}.
\end{definition}

Let $(s)_{S}$ and $(t)_{T}$ be real sets over $(e_0,S_0)$. If for every $x\overset{}{\underset{ (e_0,S_0)} {\in } }(s)_{S}$ there exists $y\overset{}{\underset{ (e_0,S_0)} {\in } }(t)_{T}$ such that $(x)_{S}\overset{\mathbb{R}}{\underset{ (e_0,S_0)} {= } }(y)_{T}$ then we say that $(s)_{S}$ is a {\bf {\itshape real subset}} of $(t)_{T}$ over $(e_0,S_0)$ or write $(s)_{S}\overset{\{\mathbb{R}\}}{\underset{ (e_0,S_0)} {\subseteq } }(t)_{T}$. If $(s)_{S}\overset{\{\mathbb{R}\}}{\underset{ (e_0,S_0)} {\subseteq } }(t)_{T}$ and $(t)_{T}\overset{\{\mathbb{R}\}}{\underset{ (e_0,S_0)} {\subseteq } }(s)_{S}$ then we say that $(s)_{S}$ and $(t)_{T}$ are {\bf {\itshape equal as real sets }} over $(e_0,S_0)$ or write $(s)_{S}\overset{\{\mathbb{R}\}}{\underset{ (e_0,S_0)} {= } }(t)_{T}$.

% Definition of real function
\begin{definition}
Given a function $(f)_{F}:(a)_{A}\overset{}{\underset{ (e_0,S_0)} {\rightarrow } }(r)_{R}$, if $(r)_{R}$ is a real set over $(e_0,S_0)$ then we say that $(f)_{F}:(a)_{A}\overset{}{\underset{ (e_0,S_0)} {\rightarrow } }(r)_{R}$ is a {\bf {\itshape real function}}.
\end{definition}

% Definition of lower bound/upper bound
\begin{definition}
Given 
\begin{itemize}
	\item real set $(a)_{A}$ over $(e_0,S_0)$
	\item real number $(x)_{X}$ over $(e_0,S_0)$
\end{itemize}
if for all $a_1\overset{}{\underset{ (e_0,S_0)} {\in } }(a)_{A}$ we have $(x)_{X}\overset{\mathbb{R}}{\underset{ (e_0,S_0)} {\leq } }(a_1)_{A}$ ({\itshape resp. $(x)_{X}\overset{\mathbb{R}}{\underset{ (e_0,S_0)} {\geq } }a_1$}) 
then we say that $(x)_{X}$ is {\bf {\itshape a lower bound of $(a)_{A}$ over $(e_0,S_0)$}} ({\itshape resp. {\bf an upper bound of $(a)_{A}$ over $(e_0,S_0)$} }).
\end{definition} 

% Greatest lower bound and Least upper bound 
\begin{definition}
Suppose we are given 
\begin{itemize}
	\item real set $(a)_{A}$ over $(e_0,S_0)$
	\item real number $(x)_{X}$ over $(e_0,S_0)$
\end{itemize}
If $(x)_{X}$ is a lower bound of $(a)_{A}$ over $(e_0,S_0)$ and for every lower bound $(x_1)_{X_1}$ of $(a)_{A}$ over $(e_0,S_0)$ we have $(x)_{X}\overset{\mathbb{R}}{\underset{ (e_0,S_0)} {\geq } }(x_1)_{X_1}$ then we say that $(x)_{X}$ is the {\bf {\itshape greatest lower bound of $(a)_{A}$ over $(e_0,S_0)$} } or write $(x)_{X}\overset{\mathbb{R}}{\underset{ (e_0,S_0)} {= } }inf (a)_{A}$. 
If $(x)_{X}$ is an upper bound of $(a)_{A}$ over $(e_0,S_0)$ and for every upper bound $(x_1)_{X_1}$ of $(a)_{A}$ over $(e_0,S_0)$ we have $(x)_{X}\overset{\mathbb{R}}{\underset{ (e_0,S_0)} {\leq } }(x_1)_{X_1}$ then we say that $(x)_{X}$ is the {\bf {\itshape least upper bound of $(a)_{A}$ over $(e_0,S_0)$} } or write $(x)_{X}\overset{\mathbb{R}}{\underset{ (e_0,S_0)} {= } }sup (a)_{A}$.
\end{definition}

% real cast of rational
\begin{definition}
Given
\begin{itemize}
	\item real number $(r)_{R}$ over $(e_0,S_0)$
	\item rational number $(q)_{Q}$ over $(e_0,S_0)$
\end{itemize}
if there exists a rational number $(q_1)_{R}$ in the upper segment of $(r)_{R}$ over $(e_0,S_0)$ such that $(q_1)_{R}\overset{\mathbb{Q}}{\underset{ (e_0,S_0)} {=} }(q)_{Q}$ and every rational number in the upper segment of $(r)_{R}$ is larger or equal to $(q_1)_{R}$ (as rational numbers) over $(e_0,S_0)$ then we write $(r)_{R}\overset{\mathbb{R},\mathbb{Q}}{\underset{ (e_0,S_0)} {=} }(q)_{Q}$ or write $(q)_{Q}\overset{\mathbb{Q},\mathbb{R}}{\underset{ (e_0,S_0)} {=} }(r)_{R}$. 
\end{definition}

If $(q)_{Q}\overset{\mathbb{Q}}{\underset{ (e_0,S_0)} {= } }0$ and $(r)_{R}\overset{\mathbb{R},\mathbb{Q}}{\underset{ (e_0,S_0)} {= } }(q)_{Q}$ then we will write $(r)_{R}\overset{\mathbb{R}}{\underset{ (e_0,S_0)} {= } }0$.

If $(q)_{Q}\overset{\mathbb{Q}}{\underset{ (e_0,S_0)} {= } }1$ and $(r)_{R}\overset{\mathbb{R},\mathbb{Q}}{\underset{ (e_0,S_0)} {= } }(q)_{Q}$ then we will write $(r)_{R}\overset{\mathbb{R}}{\underset{ (e_0,S_0)} {= } }1$.

% addition

Let $(r_1)_{R_1}$, $(r_2)_{R_2}$ and $(r_3)_{R_3}$ be real numbers over $(e_0,S_0)$ where $(l_1)_{R_1}$, $(l_2)_{R_2}$ and $(l_3)_{R_3}$ are the lower segments of $(r_1)_{R_1}$, $(r_2)_{R_2}$ and $(r_3)_{R_3}$ respectively over $(e_0,S_0)$. We say that $(r_3)_{R_3}$ is the {\bf {\itshape sum of $(r_1)_{R_1}$ and $(r_2)_{R_2}$ as reals over $(e_0,S_0)$}} or write  $(r_1)_{R_1}+(r_2)_{R_2}\overset{\mathbb{R}}{\underset{ (e_0,S_0)} {=} }(r_3)_{R_3}$ if
$$\forall (x)_{X}\overset{}{\underset{ (e_0,S_0)} {\in } }\mathbb{Q} (\exists y\overset{}{\underset{ (e_0,S_0)} {\in } }(l_3)_{R_3}((y)_{R_3}\overset{\mathbb{Q}}{\underset{ (e_0,S_0)} {= } }(x)_{X} ) \Leftrightarrow \exists x_1\overset{}{\underset{ (e_0,S_0)} {\in } }(l_1)_{R_1} \exists x_2\overset{}{\underset{ (e_0,S_0)} {\in } }(l_2)_{R_2} ( $$ 
$$(x)_{X}\overset{\mathbb{Q}}{\underset{ (e_0,S_0)} {\leq } }
(x_1)_{R_1}+(x_2)_{R_2} ) ) $$

% additive inverse
Given real numbers $(r_1)_{R_1}$ and $(r_2)_{R_2}$ over $(e_0,S_0)$, if $(r_1)_{R_1} + (r_2)_{R_2} \overset{\mathbb{R}}{\underset{ (e_0,S_0)} {=} } 0$ then we say that $(r_2)_{R_2}$ is the {\bf {\itshape real additive inverse of $(r_1)_{R_1}$ over $(e_0,S_0)$}} or write $(r_1)_{R_1}\overset{\mathbb{R}}{\underset{ (e_0,S_0)} {= } }-(r_2)_{R_2}$.

% multiplication

Let $(r_1)_{R_1}\overset{}{\underset{ (e_0,S_0)} {= } }\{l_1,u_1 \}$, $(r_2)_{R_2}\overset{}{\underset{ (e_0,S_0)} {= } }\{l_2,u_2 \}$ and $(r_3)_{R_3}\overset{}{\underset{ (e_0,S_0)} {= } }\{ l_3,u_3 \}$ be real numbers over $(e_0,S_0)$ where $(l_1)_{R_1}$, $(l_2)_{R_2}$ and $(l_3)_{R_3}$ are the lower segments of $(r_1)_{R_1}$, $(r_2)_{R_2}$ and $(r_3)_{R_3}$ respectively over $(e_0,S_0)$. We say that $(r_3)_{R_3}$ is the {\bf {\itshape product of $(r_1)_{R_1}$ and $(r_2)_{R_2}$ as reals over $(e_0,S_0)$}} or write $(r_1)_{R_1}\cdot (r_2)_{R_2}\overset{\mathbb{R}}{\underset{ (e_0,S_0)} {=} }(r_3)_{R_3}$ or $(r_1)_{R_1}(r_2)_{R_2}\overset{\mathbb{R}}{\underset{ (e_0,S_0)} {=} }(r_3)_{R_3}$ if

\begin{itemize}
	\item when $(r_1)_{R_1}\overset{\mathbb{R}}{\underset{ (e_0,S_0)} {>} }0$ and $(r_2)_{R_2}\overset{\mathbb{R}}{\underset{ (e_0,S_0)} {> } }0$ we have 
$\forall (x)_{X}\overset{}{\underset{ (e_0,S_0)} {\in } }\mathbb{Q}(
\exists y\overset{}{\underset{ (e_0,S_0)} {\in } }(u_3)_{R_3} ( (y)_{R_3}\overset{\mathbb{Q}}{\underset{ (e_0,S_0)} {= } }(x)_{X}) \Leftrightarrow 
\exists x_1\overset{}{\underset{ (e_0,S_0)} {\in } }(u_1)_{R_1} \exists x_2\overset{}{\underset{ (e_0,S_0)} {\in } }(u_2)_{R_2}( (x)_{X}\overset{\mathbb{Q}}{\underset{ (e_0,S_0)} {\geq } }(x_1)_{R_1}(x_2)_{R_2} ) )$
	\item when $(r_1)_{R_1}\overset{\mathbb{R}}{\underset{ (e_0,S_0)} {= } }0$ or $(r_2)_{R_2}\overset{\mathbb{R}}{\underset{ (e_0,S_0)} {= } }0$ we have $(r_3)_{R_3}\overset{\mathbb{R}}{\underset{ (e_0,S_0)} { =} }0$
	\item  when $(r_1)_{R_1}\overset{\mathbb{R}}{\underset{ (e_0,S_0)} {< } }0$ and $(r_2)_{R_2}\overset{\mathbb{R}}{\underset{ (e_0,S_0)} {> } }0$ we have $-(r_1)_{R_2} \cdot (r_2)_{R_2} \overset{\mathbb{R}}{\underset{ (e_0,S_0)} {= } }(r_3)_{R_3}$
	\item when $(r_2)_{R_2}\overset{\mathbb{R}}{\underset{ (e_0,S_0)} {< } }0$ and $(r_1)_{R_1}\overset{\mathbb{R}}{\underset{ (e_0,S_0)} {> } }0$ we have $(r_1)_{R_2} \cdot (-(r_2)_{R_2}) \overset{\mathbb{R}}{\underset{ (e_0,S_0)} {= } }(r_3)_{R_3}$
\end{itemize}

% multiplicative inverse
Given real numbers $(r_1)_{R_1}$ and $(r_2)_{R_2}$ over $(e_0,S_0)$, if $(r_1)_{R_1} \cdot (r_2)_{R_2} \overset{\mathbb{R}}{\underset{ (e_0,S_0)} {=} } 1$ then we say that $(r_2)_{R_2}$ is the {\bf {\itshape real multiplicative inverse of $(r_1)_{R_1}$ over $(e_0,S_0)$}} or write $(r_2)_{R_2}\overset{\mathbb{R}}{\underset{ (e_0,S_0)} {= } }(r_1)_{R_1}^{-1}$.

% absolute value and distance

% absolute value
Given real numbers $(r_1)_{R_1}$ and $(r_2)_{R_2}$ over $(e_0,S_0)$, if $(r_1)_{R_1}\overset{\mathbb{R}}{\underset{ (e_0,S_0)} {\leq } }0$ and $(r_2)_{R_2}\overset{\mathbb{R}}{\underset{ (e_0,S_0)} {= } }-(r_1)_{R_1}$ or
$(r_1)_{R_1}\overset{\mathbb{R}}{\underset{ (e_0,S_0)} {\geq } }0$ and $(r_2)_{R_2}\overset{\mathbb{R}}{\underset{ (e_0,S_0)} {= } }(r_1)_{R_1}$ then we say that $(r_2)_{R_2}$ is the {\bf {\itshape real absolute value of $(r_1)_{R_1}$  over $(e_0,S_0)$}} or write  $(r_2)_{R_2}\overset{\mathbb{R}}{\underset{ (e_0,S_0)} {= } }|(r_1)_{R_1}|$. 

% distance between real numbers

Given real numbers $(r_1)_{R_1}$, $(r_2)_{R_2}$ and $(d)_{D}$ over $(e_0,S_0)$, if $(d)_{D}\overset{\mathbb{R}}{\underset{ (e_0,S_0)} {= } }|(r_2)_{R_2} + -(r_1)_{R_1}|$ then we say that $(d)_{D}$ is the {\bf {\itshape distance between $(r_1)_{R_1}$ and $(r_2)_{R_2}$ over $(e_0,S_0)$}}.

% definition of cauchy convergence for real functions

\begin{definition}
Let $(f)_{F}:\mathbb{N}\overset{}{\underset{ (e_0,S_0)} {\rightarrow } }(r)_{R}$ be a real function. If for every real number $(\epsilon)_{E}\overset{\mathbb{R}}{\underset{ (e_0,S_0)} {> } }0$ there exists a natural number $m$ such that $\forall n_1>m \forall n_2>m ( |((f)_{F}(n_1))_{R} - ((f)_{F}(n_2))_{R}|\overset{\mathbb{R}}{\underset{ (e_0,S_0)} {< } }(\epsilon)_{E})$ then we say that $(f)_{F}:\mathbb{N}\overset{}{\underset{ (e_0,S_0)} {\rightarrow } }(r)_{R}$ is {\bf {\itshape Cauchy-convergent}}.
\end{definition}

% definition of converges to x

\begin{definition}
Let $(f)_{F}:\mathbb{N}\overset{}{\underset{ (e_0,S_0)} {\rightarrow } }(r)_{R}$ be a real function and $(x)_{X}$ be a real number over $(e_0,S_0)$. If for every real number $(\epsilon)_{E}\overset{\mathbb{R}}{\underset{ (e_0,S_0)} {> } }0$ there exists a natural number $m$ such that  $\forall n>m ( |((f)_{F}(n))_{R} - (x)_{X}|\overset{\mathbb{R}}{\underset{ (e_0,S_0)} {< } }(\epsilon)_{E})$ then we say that $(f)_{F}:\mathbb{N}\overset{}{\underset{ (e_0,S_0)} {\rightarrow } }(r)_{R}$ {\bf {\itshape converges to $(x)_{X}$ over $(e_0,S_0)$}}.
\end{definition}

% theorem on existence of limit point of cauchy convergent function

\begin{thm}
If $(f)_{F}:\mathbb{N}\overset{}{\underset{ (e_0,S_0)} {\rightarrow } }(r)_{R}$ is Cauchy-convergent then there exists a real number $(x)_{X}$ over $(e_0,S_0)$ such that $(f)_{F}:\mathbb{N}\overset{}{\underset{ (e_0,S_0)} {\rightarrow } }(r)_{R}$ converges to $(x)_{X}$ over $(e_0,S_0)$.
\end{thm}

\begin{proof}
If $(f)_{F}:\mathbb{N}\overset{}{\underset{ (e_0,S_0)} {\rightarrow } }(r)_{R}$ is Cauchy-convergent then we can construct a real function  $(g)_{G}:\mathbb{N}\overset{}{\underset{ (e_0,S_0)} {\rightarrow } }\mathbb{N}$ such that
\begin{itemize}
	\item $(g)_{G}(1)\overset{}{\underset{ (e_0,S_0)} {= } }m$ where $m$ is the smallest number such that for any $n_1>m$ and any $n_2>m$ we have $ |((f)_{F}(n_1))_{R} - ((f)_{F}(n_2))_{R}|\overset{\mathbb{R}}{\underset{ (e_0,S_0)} {< } }1$
	\item for every $n\in \mathbb{N}$, $(g)_{G}(n+1)\overset{}{\underset{ (e_0,S_0)} {= } }m$ where $m$ is the smallest number larger than $(g)_{G}(n)$ and for any $n_1>m$ and any $n_2>m$ we have $ |((f)_{F}(n_1))_{R} - ((f)_{F}(n_2))_{R}|\overset{\mathbb{R}}{\underset{ (e_0,S_0)} {< } }1/(n+1)$
\end{itemize}
Now let $(h)_{H}:\mathbb{N}\overset{}{\underset{ (e_0,S_0)} {\rightarrow } }(r)_{R}$ be a real function such that for all $n$ we have  $((h)_{H}(n))_{R}\overset{\mathbb{R}}{\underset{ (e_0,S_0)} {= } }((f)_{F}( (g)_{G}(n) ))_{R}$. Let $(s)_{S}\overset{}{\underset{ (e_0,S_0)} {= } } \{l,u \}$ be a real number over $(e_0,S_0)$ (with $(l)_{S}$ as its lower segment) such that

for any rational number $(q)_{Q}$ over $(e_0,S_0)$, there exists $l_1\overset{}{\underset{ (e_0,S_0)} {\in } }(l)_{S}$ with $(q)_{Q}\overset{\mathbb{Q}}{\underset{ (e_0,S_0)} {= } }(l_1)_{S}$ if and only if there exists a natural number $n$ and a rational number $(q_n)_{R}$ in the lower segment of $( (h)_{H}(n))_{R}$ over $(e_0,S_0)$ such that  $(q)_{Q} + 1/n \overset{\mathbb{Q}}{\underset{ (e_0,S_0)} {= } }(q_n)_{R}$.
 
We see that $(f)_{F}:\mathbb{N}\overset{}{\underset{ (e_0,S_0)} {\rightarrow } }(r)_{R}$ converges to $(s)_{S}$ over $(e_0,S_0)$.

\end{proof}

% Real-sufficient

\begin{definition}
Let $(r)_{R}$ be a real set over $(e_0,S_0)$. If for every real number $(r_1)_{R_1}$ over $(e_0,S_0)$ and every real number $(\epsilon)_{R_2}\overset{\mathbb{R}}{\underset{ (e_0,S_0)} {> } }0$ there exists a real number $(r_3)_{R_3}$ over $(e_0,S_0)$ such that the distance between $(r_1)_{R_1}$ and $(r_3)_{R_3}$ is less than $(\epsilon)_{R_2}$ over $(e_0,S_0)$ then we say that $(r)_{R}$ is {\bf {\itshape real-sufficient over $(e_0,S_0)$}}.
\end{definition}

% existence of inf and sup of bounded sets of real numbers
\begin{thm}
Let $(r)_{R}$ be a non-empty real set over $(e_0,S_0)$.
\begin{enumerate}
\item If $(r)_{R}$ has an upper bound over $(e_0,S_0)$ then there exists a real number $(x)_{X}$ over $(e_0,S_0)$ such that $(x)_{X}\overset{\mathbb{R}}{\underset{ (e_0,S_0)} {= } }sup (r)_{R}$.
\item If $(r)_{R}$ has a lower bound over $(e_0,S_0)$ then there exists a real number $(x)_{X}$ over $(e_0,S_0)$ such that $(x)_{X}\overset{\mathbb{R}}{\underset{ (e_0,S_0)} {= } }inf (r)_{R}$.
\end{enumerate}
\end{thm}

\begin{proof}

1) Let $(x)_{X}\overset{}{\underset{ (e_0,S_0)} {= } }\{l,u \}$ be a real number over $(e_0,S_0)$ (with $(l)_{X}$ as  its lower segment) such that 

for any rational number $(q)_{Q}$ over $(e_0,S_0)$, there exists $q_1\overset{}{\underset{ (e_0,S_0)} {\in } }(u)_{X}$ with $(q)_{Q}\overset{\mathbb{Q}}{\underset{ (e_0,S_0)} {= } }(q_1)_{X}$ if and only if
there exists an upper bound $(b)_{B}\overset{}{\underset{ (e_0,S_0)} {= } }\{l_1,u_1 \}$ (where $(u_1)_{B}$ is the upper segment of $(b)_{B}$) of $(r)_{R}$ over $(e_0,S_0)$ and a rational number $(q_2)_{B}$ over $(e_0,S_0)$ with $q_2\overset{}{\underset{ (e_0,S_0)} {\in } }(u_1)_{B}$ such that $(q_2)_{B}\overset{\mathbb{Q}}{\underset{ (e_0,S_0)} {= } }(q)_{Q}$.
 
Since $(r)_{R}$ is non-empty we see that the real number $(x)_{X}$ over $(e_0,S_0)$ exists and $(x)_{X}\overset{\mathbb{R}}{\underset{ (e_0,S_0)} {= } }sup (r)_{R}$.

2) Let $(x)_{X}\overset{}{\underset{ (e_0,S_0)} {= } }\{l,u \}$ be a real number over $(e_0,S_0)$ (where $(l)_{X}$ is its lower segment) such that 

for any rational number $(q)_{Q}$ over $(e_0,S_0)$ there exists $q_1\overset{}{\underset{ (e_0,S_0)} {\in } }(u)_{X}$ with $(q)_{Q}\overset{\mathbb{Q}}{\underset{ (e_0,S_0)} {= } }(q_1)_{X}$ if and only if
there exists a lower bound $(b)_{B}\overset{}{\underset{ (e_0,S_0)} {= } }\{l_1,u_1 \}$ (where $(l_1)_{B}$ is the lower segment of $(b)_{B}$) of $(r)_{R}$ over $(e_0,S_0)$ and a rational number $(q_2)_{B}$ over $(e_0,S_0)$ with $q_2\overset{}{\underset{ (e_0,S_0)} {\in } }(l_1)_{B}$ such that $(q_2)_{B}\overset{\mathbb{Q}}{\underset{ (e_0,S_0)} {= } }(q)_{Q}$.

Since $(r)_{R}$ is non-empty we see that the real number $(x)_{X}$ over $(e_0,S_0)$ exists and $(x)_{X}\overset{\mathbb{R}}{\underset{ (e_0,S_0)} {= } }inf (r)_{R}$.

\end{proof}

\section{Countability and choice}

The main goal of this section is to prove that every set is countable and show that the ``axiom of choice'' becomes a natural theorem.

\begin{thm}
Let $(w)_{W}$ be a set over $(e_0,S_0)$ containing all words. There exists a bijection $(f)_{F}:(w)_{W}\overset{}{\underset{ (e_0,S_0)} {\rightarrow } }\mathbb{N}$.
\end{thm}

\begin{proof}
Let $(w)_{W}$ be the set over $(e_0,S_0)$ of all words and $(f)_{F}:(w)_{W}\overset{}{\underset{ (e_0,S_0)} {\rightarrow } }\mathbb{N}$ a function such that
\begin{itemize}
	\item $(f)_{F}(0)\overset{}{\underset{ (e_0,S_0)} {= } }1$ and $(f)_{F}(1)\overset{}{\underset{ (e_0,S_0)} {= } }11$
	\item for any $k$ and any words $w_1$ and $w_2$ of length $k$ we have 
	$$\forall n\in \mathbb{N} ((f)_{F}(w_1)+ n = (f)_{F}(w_2) \Rightarrow $$
	$$ ( (f)_{F}(w_10)+n =(f)_{F}(w_20) \wedge (f)_{F}(w_11) + n = (f)_{F}(w_21) )) $$
	\item for any natural number $m_1$ and any counting number $m_2$ over $0$, if $m_1+1$ is the natural number associated to $m_2$ then $(f)_{F}(m_2) = (f)_{F}(m_1)+1 $
	\item for any natural number $m_1$ and any counting number $m_2$ over $0$, if $m_1$ is the natural number associated to $m_2$ then $(f)_{F}(m_10) + 1 =(f)_{F}(m_21)$
\end{itemize}

Using (comprehension) and (word induction) we can see that such a function exists and is bijective. 
\end{proof}

% Countability coro ( of any set )
\begin{cor}
Over any set foundation $(e_0,S_0)$, every set is countable.
\end{cor}

\begin{proof}
Let $(e)_{E}$ be a set over $(e_0,S_0)$.
By the previous theorem there exists a bijection  $(f)_{F}:(w)_{W}\overset{}{\underset{ (e_0,S_0)} {\rightarrow } }\mathbb{N}$ where $(w)_{W}$ is a set over $(e_0,S_0)$ which contains every word. Let $(g)_{G}\overset{}{\underset{ (e_0,S_0)} {= } }(f)_{F}|(e)_{E}$. Construct $(h)_{H}:(e)_{E}\overset{}{\underset{ (e_0,S_0)} {\rightarrow } }\mathbb{N}$ such that
\begin{itemize}
	\item if $n_0$ is the smallest number in the image of $(g)_{G}$ and $(g)_{G}(x_0)\overset{}{\underset{ (e_0,S_0)} {= } }n_0$ for some $x_0$ then $(h)_{H}(x_0)\overset{}{\underset{ (e_0,S_0)} {=} }1$
	\item for any $x_1\overset{}{\underset{ (e_0,S_0)} {\in } }(e)_{E}$ and any $x_2\overset{}{\underset{ (e_0,S_0)} {\in } }(e)_{E}$, if there exists no word $x_3$ such that $(g)_{G}(x_3)$ is strictly between $(g)_{G}(x_1)$ and $(g)_{G}(x_2)$ then $(h)_{H}(x_2)=(h)_{H}(x_1)+1$  
\end{itemize}
The image of $(h)_{H}$ over $(e_0,S_0)$ is either $\mathbb{N}$ or $\mathbb{N}_{\leq n}$ for some $n$. So $(e)_{E}$ is countable over $(e_0,S_0)$.
\end{proof}

% Infinite bijection theorem

As expected we can now show that there exists a bijection between any two infinite sets.

\begin{cor}
Over any set foundation $(e_0,S_0)$, if $(e_1)_{E_1}$ and $(e_2)_{E_2}$ are infinite sets then there exists a bijection $(h)_{H}:(e_1)_{E_1}\overset{}{\underset{ (e_0,S_0)} {\rightarrow } }(e_2)_{E_2}$.
\end{cor}

\begin{proof}
By countability of sets there exists bijective functions
$(f)_{F}:(e_1)_{E_1}\overset{}{\underset{ (e_0,S_0)} {\rightarrow } }\mathbb{N}$ and $(g)_{G}:(e_2)_{E_2}\overset{}{\underset{ (e_0,S_0)} {\rightarrow } }\mathbb{N}$.
Since $(e_1)_{E_1}$ and $(e_2)_{E_2}$ are infinite sets over $(e_0,S_0)$ there exists a function $(h)_{H}:(e_1)_{E_1}\overset{}{\underset{ (e_0,S_0)} {\longrightarrow } }(e_2)_{E_2}$ such that for any $x$, $(h)_{H}(x)\overset{}{\underset{ (e_0,S_0)} {= } }y$ if and only if $(f)_{F}(x)\overset{}{\underset{ (e_0,S_0)} {= } }(g)_{G}(y)$. We see that $(h)_{H}$ defines a bijection between $(e_1)_{E_1}$ and $(e_2)_{E_2}$ over $(e_0,S_0)$.
\end{proof}

% Choice thm
\begin{thm}(Choice theorem)
	
Let $(s)_{S}$ be a set over $(e_0,S_0)$ of non-empty sets over $(e_0,S_0)$. There exists a function $(f)_{F}:(s)_{S}\overset{}{\underset{ (e_0,S_0)} {\rightarrow } }(t)_{T}$ such that for any $s_1\overset{}{\underset{ (e_0,S_0)} {\in } }(s)_{S}$, $(f)_{F}(s_1)\overset{}{\underset{ (e_0,S_0)} {\in } }(s_1)_{S}$.

\end{thm}

\begin{proof}
Let $(c)_{C}:(w)_{W}\overset{}{\underset{ (e_0,S_0)} {\rightarrow } }\mathbb{N}$ be a bijection where $(w)_{W}$ is the set (over $(e_0,S_0)$) of all words (existence guaranteed by countability). Define $(f)_{F}:(s)_{S}\overset{}{\underset{ (e_0,S_0)} {\rightarrow } }(t)_{T}$ such that for any $s_1\overset{}{\underset{ (e_0,S_0)} {\in } }(s)_{S}$, $(f)_{F}(s_1)\overset{}{\underset{ (e_0,S_0)} {= } }x$ where $(c)_{C}(x)$ is the smallest natural number in the set $(c)_{C}( (s_1)_{S} )$ over $(e_0,S_0)$. The existence of such a function is guaranteed by (comprehension).

\end{proof}

\section{Set incompleteness}
This section will show that ``proofs of non-bijection between two infinite sets''(in ZFC or other set theories) are essentially ``proofs of non-existence of at least one of the two specified sets''. In particular, proofs of uncountability of a set are essentially proofs of its non-existence. In a certain sense, uncountability is a form of incompleteness.

Now we see that the power set of an infinite set does not exist.
% Th: Power set
\begin{thm}
Over any set foundation $(e_0,S_0)$, if $(x)_{X}$ is an infinite set then it has no power set.
\end{thm}

\begin{proof}
Let $(y)_{Y}$ be a set over $(e_0,S_0)$ such that for any $a\overset{}{\underset{ (e_0,S_0)} {\in } }(y)_{Y}$ we have $(a)_{Y}\overset{}{\underset{ (e_0,S_0)} {\subseteq } }(x)_{X}$.

Let $(c_1)_{C_1}:(x)_{X}\overset{}{\underset{ (e_0,S_0)} {\rightarrow } }\mathbb{N}$ be a bijective function (existence is guaranteed by countability).

Let $(c_2)_{C_2}:(y)_{Y}\overset{}{\underset{ (e_0,S_0)} {\rightarrow } }\mathbb{N}$ be an injective function (existence is guaranteed by countability).
  
%%% Construction
Let $(y_1)_{Y_1}$ be a set over $(e_0,S_0)$ with $(y_1)_{Y_1}\overset{}{\underset{ (e_0,S_0)} {\subseteq } }(x)_{X}$ such that 
$$\forall a\overset{}{\underset{ (e_0,S_0)} {\in } }(y)_{Y} \forall b\overset{}{\underset{ (e_0,S_0)} {\in } }(x)_{X}
((c_2)_{C_2}(a)\overset{}{\underset{ (e_0,S_0)} {= } }(c_1)_{C_1}(b)\Rightarrow( b\overset{}{\underset{ (e_0,S_0)} {\in } }(y_1)_{Y_1} \Leftrightarrow b\overset{}{\underset{ (e_0,S_0)} {\notin } }(a)_{Y}))$$

We see that for every $a\overset{}{\underset{ (e_0,S_0)} {\in } }(y)_{Y}$ we have $(y_1)_{Y_1}\overset{}{\underset{ (e_0,S_0)} { \neq} }(a)_{Y}$. So we must conclude that $(x)_{X}$ has no power set over $(e_0,S_0)$.
%%% end construction
\end{proof}

Since every infinite set is countable and  the power set of an infinite set does not exist we obtain a trivial resolution of the generalized continuum hypothesis (there exists no set of cardinality strictly between the cardinality of an infinite set and the cardinality of its power set). 

Now we can see that a set containing all sets cannot exist.

% Cor: Sets
\begin{cor}
Over any set foundation $(e_0,S_0)$, there does not exist a set containing all sets.
\end{cor}

\begin{proof}

Let $(w)_{W}$ over $(e_0,S_0)$ be the set of all words. Using the previous theorem it follows that the set of all subsets of $(w)_{W}$ over $(e_0,S_0)$ does not exist. So there does not exist a set of all sets over $(e_0,S_0)$. 

\end{proof}

% Th: Least upper
Now we see that the \textit{least upper bound property} does not hold for a densely and totally ordered set containing at least two elements.
\begin{thm}
Let $( (x)_{X},(t)_{T} )_{(e_0,S_0)}$ be a dense total order with $card((x)_{X})\overset{}{\underset{ (e_0,S_0)} {\geq} }2$. There exists a subset $(d)_{D}$ of $(x)_{X}$ over $(e_0,S_0)$ which has no least upper bound over $((x)_{X},(t)_{T} )_{(e_0,S_0)}$.
\end{thm}

\begin{proof}

Let $(c)_{C}:(x)_{X}\overset{}{\underset{ (e_0,S_0)} {\rightarrow } }\mathbb{N}$ be an injective function (existence is guaranteed by countability).  

%%% Construction
Since $card((x)_{X})\overset{}{\underset{ (e_0,S_0)} {\geq} }2$ then there exist $x_1\overset{}{\underset{ (e_0,S_0)} { \in} }(x)_{X}$ and $x_2\overset{}{\underset{ (e_0,S_0)} {\in } }(x)_{X}$ with $x_1\overset{(t)_{T}}{\underset{ (e_0,S_0)} { <} }x_2$.
Let $(f)_{F}:\mathbb{N}\overset{}{\underset{ (e_0,S_0)} {\rightarrow } }\mathbb{N}$ and $(g)_{G}:\mathbb{N}\overset{}{\underset{ (e_0,S_0)} {\rightarrow } }\mathbb{N}$ be functions such that 

\begin{itemize}
	\item $(f)_{F}(1)\overset{}{\underset{ (e_0,S_0)} {= } }m$ if $m$ is the smallest number such that
	$x_1\overset{(t)_{T}}{\underset{ (e_0,S_0)} { <} }y \overset{(t)_{T}}{\underset{ (e_0,S_0)} {< } }x_2$ where  $(c)_{C}(y)\overset{}{\underset{ (e_0,S_0)} {= } }m$
	
	\item $(g)_{G}(1)\overset{}{\underset{ (e_0,S_0)} {= } }m$ if $m$ is the smallest number such that
	$z\overset{(t)_{T}}{\underset{ (e_0,S_0)} {< } }y \overset{(t)_{T}}{\underset{ (e_0,S_0)} {< } } x_2$ where  $(c)_{C}(y)\overset{}{\underset{ (e_0,S_0)} {= } }m$ and $(c)_{C}( z )\overset{}{\underset{ (e_0,S_0)} {= } }(f)_{F}(1)$
	
	\item $(f)_{F}(k+1)\overset{}{\underset{ (e_0,S_0)} {= } }m$ if $m$ is the smallest number such that  $z_1 \overset{(t)_{T}}{\underset{ (e_0,S_0)} { <} }y \overset{(t)_{T}}{\underset{ (e_0,S_0)} {< } }z_2$ where $(c)_{C}(y)\overset{}{\underset{ (e_0,S_0)} {= } }m$, $(c)_{C}( z_1 )\overset{}{\underset{ (e_0,S_0)} {= } }(f)_{F}(k)$ and $(c)_{C}( z_2 )\overset{}{\underset{ (e_0,S_0)} {= } }(g)_{G}(k)$
	
	\item $(g)_{G}(k+1)\overset{}{\underset{ (e_0,S_0)} {= } }m$ if $m$ is the smallest number such that 
$z_1\overset{(t)_{T}}{\underset{ (e_0,S_0)} { <} }y \overset{(t)_{T}}{\underset{ (e_0,S_0)} {< } }z_2$ where $(c)_{C}(y)\overset{}{\underset{ (e_0,S_0)} {= } }m$, $(c)_{C}( z_1 )\overset{}{\underset{ (e_0,S_0)} {= } }(f)_{F}(k+1)$ and  $(c)_{C}( z_2 )\overset{}{\underset{ (e_0,S_0)} {= } }(g)_{G}(k)$

\end{itemize}
Let $(d)_{D}\overset{}{\underset{ (e_0,S_0)} {= } }\{ a\overset{}{\underset{ (e_0,S_0)} {\in } }(x)_{X} \mid \exists n\overset{}{\underset{ (e_0,S_0)} {\in } }(f)_{F}(\mathbb{N})((c)_{C}(a)\overset{}{\underset{ (e_0,S_0)} {= } }n) \} $

We see that $(d)_{D}$ has no least upper bound over $( (x)_{X},(t)_{T} )_{(e_0,S_0)}$

%%% end Construction

\end{proof}

We can now see that there does not exist a set of all real numbers.
% Cor: Reals
\begin{thm}
Given any real set $(r)_{R}$ over $(e_0,S_0)$ there exists a real number $(r_1)_{R_1}$ over $(e_0,S_0)$ such that $\forall r_2\overset{}{\underset{ (e_0,S_0)} {\in } }(r)_{R}((r_1)_{R_1}\overset{\mathbb{R}}{\underset{ (e_0,S_0)} {\neq } }(r_2)_{R})$.
\end{thm}

\begin{proof}

%%% 

If $(r)_{R}$ over $(e_0,S_0)$ does not contain all the rational numbers (seen as real numbers) then trivially $(r)_{R}$ does not contain every real number over $(e_0,S_0)$.

If $(r)_{R}$ over $(e_0,S_0)$ contains every rational number then $( (r)_{R},(t)_{T} )_{(e_0,S_0)}$ is a dense total order where $\forall x_1 \forall x_2 ( x_1\overset{(t)_{T}}{\underset{ (e_0,S_0)} {\leq } }x_2 \Leftrightarrow (x_1)_{R}\overset{\mathbb{R}}{\underset{ (e_0,S_0)} {\leq } }(x_2)_{R})$. Using the previous theorem there exists a subset $(s)_{S}$ of $(r)_{R}$ (over $(e_0,S_0)$) which does not have a least upper bound over $( (r)_{R},(t)_{T} )_{(e_0,S_0)}$. Let $(r_1)_{R_1}$ be a real number over $(e_0,S_0)$ such that  $\forall x\forall s_i\overset{}{\underset{ (e_0,S_0)} {\in } }(s)_{S}( x \text{ is in the lower segment of }$ $(s_i)_{R} \text{ over } (e_0,S_0) \Leftrightarrow  \exists y ( (y)_{R_1}\overset{\mathbb{Q}}{\underset{ (e_0,S_0)} {= } }(x)_{S} \text{ and } y \text{ is in the lower segment of } (r_1)_{R_1} \text{ over } (e_0,S_0)))$

It follows that $\forall r_2\overset{}{\underset{ (e_0,S_0)} {\in } }(r)_{R}( (r_2)_{R}\overset{\mathbb{R}}{\underset{ (e_0,S_0)} {\neq } }(r_1)_{R_1})$.

\end{proof}

The next theorem shows that a set of all ``ordinals'' cannot exist.

% Th: Well ordering

\begin{thm}
Let $(s)_{S}$ be a set over $(e_0,S_0)$  such that for every element $s_i$ of $(s)_{S}$ over $(e_0,S_0)$ there exist $x_i$ and $r_i$ such that $(s_i)_{S}\overset{}{\underset{ (e_0,S_0)} {= } }(x_i,r_i)$ and $( (x_i)_{S},(r_i)_{S})_{(e_0,S_0)}$ is a well ordering. There exists a well ordering $( (y)_{Y},(t)_{T} )_{(e_0,S_0)}$ such that for all $(x_i,r_i)\overset{}{\underset{ (e_0,S_0)} {\in } }(s)_{S}$, $( (y)_{Y},(t)_{T} )_{(e_0,S_0)}$ is not order isomorphic to $( (x_i)_{S},(r_i)_{S})_{(e_0,S_0)}$.
\end{thm}

\begin{proof}

Let $(c)_{C}:(w)_{W}\overset{}{\underset{ (e_0,S_0)} {\rightarrow } }\mathbb{N}$ be a bijection where $(w)_{W}$ is a set over $(e_0,S_0)$ containing every word (existence is guaranteed by countability).  

%%% Construction
Let $( (y)_{Y},(t)_{T} )_{(e_0,S_0)}$ be a well ordering such that 

\begin{itemize}
\item $a\overset{}{\underset{ (e_0,S_0)} {\in } }(y)_{Y}$ if and only if either $a=1$ or there exist $(b,r)\overset{}{\underset{ (e_0,S_0)} {\in } }(s)_{S}$ and $b_0\overset{}{\underset{ (e_0,S_0)} {\in } }(b)_{S}$ such that $a=m_1m_0$ where $m_0$ is a counting number over $0$, $(c)_{C}(b_0)$ is the length of $m_0$ and $m_1\overset{}{\underset{ (e_0,S_0)} {= } }(c)_{C}(b)$.

\item for any $m_1m_0\overset{}{\underset{ (e_0,S_0)} {\in } }(y)_{Y}$ where $m_0$ is a counting number over $0$ and $m_1$ is a natural number, and for any
$n_1n_0\overset{}{\underset{ (e_0,S_0)} {\in } }(y)_{Y}$ where $n_0$ is a counting number over $0$ and $n_1$ is a natural number we have

	\begin{itemize}
		\item if $n_1<m_1$ then $n_1n_0\overset{(t)_{T}}{\underset{ (e_0,S_0)} { < } }m_1m_0$
		\item $n_1n_0\overset{(t)_{T}}{\underset{ (e_0,S_0)} { \leq } }n_1m_0$ if and only if $z_1\overset{(r)_{S}}{\underset{ (e_0,S_0)} {\leq } }z_2$  where $(z,r)\overset{}{\underset{ (e_0,S_0)} {\in } }(s)_{S}$, $(c)_{C}(z)\overset{}{\underset{ (e_0,S_0)} {= } }n_1$, $(c)_{C}(z_1)$ is the length of $n_0$ and $(c)_{C}(z_2)$ is the length of $m_0$.
		\item $m_1m_0\overset{(t)_{T}}{\underset{ (e_0,S_0)} {< } }1$ 
	\end{itemize}
\end{itemize}

We see that $( (y)_{Y},(t)_{T} )_{(e_0,S_0)}$ is order isomorphic to none of the well orderings identified by the elements of $(s)_{S}$ over $(e_0,S_0)$.
%%% end Construction

\end{proof}

% Cor: Orders

\begin{cor}
Let $(s)_{S}$ be a set over $(e_0,S_0)$  such that for every element $s_i$ of $(s)_{S}$ over $(e_0,S_0)$ there exist $x_i$ and $r_i$ such that $(s_i)_{S}\overset{}{\underset{ (e_0,S_0)} {= } }(x_i,r_i)$ and $( (x_i)_{S},(r_i)_{S})_{(e_0,S_0)}$ is a partial order. There exists a partial order $( (y)_{Y},(t)_{T} )_{(e_0,S_0)}$ such that for all $(x_i,r_i)\overset{}{\underset{ (e_0,S_0)} {\in } }(s)_{S}$, $( (y)_{Y},(t)_{T} )_{(e_0,S_0)}$ is not order isomorphic to $( (x_i)_{S},(r_i)_{S})_{(e_0,S_0)}$.
\end{cor}

\begin{proof}

Follows immediately from the previous theorem on well orderings.
\end{proof}

The next theorem lets us see that a set of all functions cannot exist.

% Th: Functions

\begin{thm}
Given
\begin{itemize}
	\item sets $(x)_{X}$ and $(y)_{Y}$ over $(e_0,S_0)$
	\item a set $(s)_{S}$ of functions from $(x)_{X}$ to $(y)_{Y}$ over $(e_0,S_0)$
\end{itemize}
if $(x)_{X}$ is an infinite set over $(e_0,S_0)$ and $card((y)_{Y})\overset{}{\underset{ (e_0,S_0)} {\geq } }2$ then there exists a function $(g)_{G}:(x)_{X}\overset{}{\underset{ (e_0,S_0)} {\rightarrow } }(y)_{Y}$ such for every $f_1\overset{}{\underset{ (e_0,S_0)} {\in } }(s)_{S}$ there exists $x_1\overset{}{\underset{ (e_0,S_0)} {\in } }(x)_{X}$ such that $(g)_{G}(x_1)\overset{}{\underset{ (e_0,S_0)} {\neq } }(f_1)_{S}(x_1)$.

\end{thm}

\begin{proof}

Let $(c_1)_{C_1}:(x)_{X}\overset{}{\underset{ (e_0,S_0)} {\rightarrow } }\mathbb{N}$ be a bijective  function (existence is guaranteed by countability).
Let $(c_2)_{C_2}:(s)_{S}\overset{}{\underset{ (e_0,S_0)} {\rightarrow } }\mathbb{N}$ be an injective function (existence is guaranteed by countability).

Since $card((y)_{Y})\overset{}{\underset{ (e_0,S_0)} {\geq } }2$ there exist $y_1\overset{}{\underset{ (e_0,S_0)} {\in } }(y)_{Y}$ and $y_2\overset{}{\underset{ (e_0,S_0)} {\in } }(y)_{Y}$ such that $y_1\neq y_2$.

%%% Construction
Let $(g)_{G}:(x)_{X}\overset{}{\underset{ (e_0,S_0)} {\rightarrow } }\{y_1,y_2\}$ be a function such that 

$\forall a\overset{}{\underset{ (e_0,S_0)} {\in } }(s)_{S} \forall b\overset{}{\underset{ (e_0,S_0)} {\in } }(x)_{X}
((c_2)_{C_2}(a)\overset{}{\underset{ (e_0,S_0)} {= } }(c_1)_{C_1}(b)\Rightarrow( (g)_{G}(b)\overset{}{\underset{ (e_0,S_0)} {= } }y_1 \Leftrightarrow (a)_{S}(b)\overset{}{\underset{ (e_0,S_0)} {= } }y_2))$

We see that for every $f_1\overset{}{\underset{ (e_0,S_0)} {\in } }(s)_{S}$ there exists $x_1\overset{}{\underset{ (e_0,S_0)} {\in } }(x)_{X}$ such that $(g)_{G}(x_1)\overset{}{\underset{ (e_0,S_0)} {\neq } }(f_1)_{S}(x_1)$.

%%% end construction

\end{proof}

\section{Some topology}
Due to results of ``incompleteness'' mentioned in the previous section we see that a set of subsets cannot be necessarily closed under infinite unions, so the notion of topological space cannot be reformulated in a  straightforward manner. This section will only deal with metric spaces (other generalizations will be considered in subsequent papers). 

The goal of this section is to reformulate some fundamental results in topology by replacing continuity by ``extensibility'', compactness by total boundedness and by redefining connectedness. We will first re-introduce some standard definitions and then develop the language of ``extensions'' to adequately capture the core ideas behind main results in topology.
% definition of metric space
\begin{definition}
Given
\begin{itemize}
	\item set $(x)_{X}$ over $(e_0,S_0)$
	\item set $(y)_{Y}\overset{}{\underset{ (e_0,S_0)} {= } }(x)_{X} \times (x)_{X}$
	\item real set $(r)_{R}$ over $(e_0,S_0)$
	\item function $(d)_{D}:(y)_{Y}\overset{}{\underset{ (e_0,S_0)} {\rightarrow } }(r)_{R}$
\end{itemize}
if for all $x_1\overset{}{\underset{ (e_0,S_0)} {\in } }(x)_{X}$, $x_2\overset{}{\underset{ (e_0,S_0)} {\in } }(x)_{X}$ and $x_3\overset{}{\underset{ (e_0,S_0)} {\in } }(x)_{X}$ we have
\begin{itemize}
	\item $((d)_{D}(x_1,x_2))_{R}\overset{\mathbb{R}}{\underset{ (e_0,S_0)} {= } }0 \Leftrightarrow x_1=x_2$
	\item $((d)_{D}(x_1,x_2))_{R}\overset{\mathbb{R}}{\underset{ (e_0,S_0)} {= } }((d)_{D}(x_2,x_1))_{R}$
	\item $((d)_{D}(x_1,x_3))_{R}\overset{\mathbb{R}}{\underset{ (e_0,S_0)} {\leq } }((d)_{D}(x_1,x_2))_{R} + ((d)_{D}(x_2,x_3))_{R}$
\end{itemize}
then we say that $( (x)_{X},(d)_{D},(r)_{R} )_{(e_0,S_0)}$ is a {\bf {\itshape metric space}}.
\end{definition}

% definition of real-induced metric space

\begin{definition}
Given a metric space $( (x)_{X},(d)_{D},(r)_{R})_{(e_0,S_0)}$, if $(x)_{X}$ is a real set over $(e_0,S_0)$ and for any real numbers $(x_1)_{X}$ and $(x_2)_{X}$ over $(e_0,S_0)$ with $x_1\overset{}{\underset{ (e_0,S_0)} {\in } }(x)_{X}$ and $x_2\overset{}{\underset{ (e_0,S_0)} {\in } }(x)_{X}$ we have $((d)_{D}(x_1,x_2))_{R}\overset{\mathbb{R}}{\underset{ (e_0,S_0)} {= } }|(x_1)_{X}-(x_2)_{X}|$ then we say that $( (x)_{X},(d)_{D},(r)_{R})_{(e_0,S_0)}$ is a {\bf {\itshape real-induced}} metric space.
\end{definition}

% definition of isometry

\begin{definition}
Given
\begin{itemize}
	\item metric spaces $( (x)_{X},(d_1)_{D_1},(r_1)_{R_1} )_{(e_0,S_0)}$ and $( (y)_{Y},(d_2)_{D_2},(r_2)_{R_2} )_{(e_0,S_0)}$
	\item function $(f)_{F}:(x)_{X}\overset{}{\underset{ (e_0,S_0)} {\rightarrow } }(y)_{Y}$
\end{itemize}
if for any $x_1\overset{}{\underset{ (e_0,S_0)} {\in } }(x)_{X}$ and any $x_2\overset{}{\underset{ (e_0,S_0)} {\in } }(x)_{X}$ we have \\ $((d_1)_{D_1}( (x_1,x_2) ))_{R_1}\overset{\mathbb{R}}{\underset{ (e_0,S_0)} {= } }((d_2)_{D_2}( ( (f)_{F}(x_1),(f)_{F}(x_2)) ))_{R_2}$ then we say that \\ $(f)_{F}:( (x)_{X},(d_1)_{D_1},(r_1)_{R_1} )_{(e_0,S_0)}\overset{}{\underset{ (e_0,S_0)} {\rightarrow } }( (y)_{Y},(d_2)_{D_2},(r_2)_{R_2} )_{(e_0,S_0)}$ is an {\bf {\itshape isometry}}.
\end{definition}

% definition of open ball of radius epsilon
\begin{definition}
Given a metric space $( (x)_{X},(d)_{D},(r)_{R} )_{(e_0,S_0)}$ and a subset $(y)_{Y}$ of $(x)_{X}$ over $(e_0,S_0)$, if there exist $a\overset{}{\underset{ (e_0,S_0)} {\in } }(x)_{X}$ and a real number $(\epsilon)_{E}\overset{\mathbb{R}}{\underset{ (e_0,S_0)} {> } }0$ such that 
	$$ (y)_{Y}\overset{}{\underset{ (e_0,S_0)} {= } } \{b\overset{}{\underset{ (e_0,S_0)} {\in } }(x)_{X} \mid ((d)_{D}((a,b)))_{R}\overset{\mathbb{R}}{\underset{ (e_0,S_0)} {< } }(\epsilon)_{E} \}$$
	then we say $(y)_{Y}$ is an {\bf {\itshape open ball of radius $(\epsilon)_{E}$ (centered at $a$) over $( (x)_{X},(d)_{D},(r)_{R} )_{(e_0,S_0)}$ }} or write $(y)_{Y}\overset{}{\underset{ (e_0,S_0)} {= } }\beta(a,(\epsilon)_{E})$ when there is no ambiguity.
\end{definition}

% definition of bounded
\begin{definition}
Given a metric space $( (x)_{X},(d)_{D},(r)_{R})_{(e_0,S_0)}$ and a subset $(y)_{Y}$ of $(x)_{X}$ over $(e_0,S_0)$, if there exists a real number $(\epsilon)_{E}$ over $(e_0,S_0)$ such that for any points $y_1\overset{}{\underset{ (e_0,S_0)} {\in } }(y)_{Y}$ and $y_2\overset{}{\underset{ (e_0,S_0)} {\in } }(y)_{Y}$ we have $((d)_{D}(y_1,y_2))_{R}\overset{\mathbb{R}}{\underset{ (e_0,S_0)} {\leq } }(\epsilon)_{E}$ then we say that $(y)_{Y}$ is {\bf {\itshape bounded over $( (x)_{X},(d)_{D},(r)_{R} )_{(e_0,S_0)}$ }}.
\end{definition}

% definition of totally bounded

\begin{definition}
Given a metric space $( (x)_{X},(d)_{D},(r)_{R} )$ and a subset $(y)_{Y}\overset{}{\underset{ (e_0,S_0)} {\subseteq } }(x)_{X}$, if for any $(\epsilon)_{E} \overset{\mathbb{R}}{\underset{ (e_0,S_0)} {> } }0$ the set $(y)_{Y}$ can be covered by a finite union of open balls of radius less than $(\epsilon)_{E}$ then we say that $(y)_{Y}$ is {\bf {\itshape totally bounded over $( (x)_{X},(d)_{D},(r)_{R} )_{(e_0,S_0)}$}}. If $(y)_{Y}\overset{}{\underset{ (e_0,S_0)} { =} }(x)_{X}$ we will just say that $( (x)_{X},(d)_{D},(r)_{R} )_{(e_0,S_0)}$ is {\bf {\itshape totally bounded}}.
\end{definition}

% definition of locally totally bounded

\begin{definition}
Given a metric space $( (x)_{X},(d)_{D},(r)_{R})_{(e_0,S_0)}$, if every bounded subset of $(x)_{X}$ is totally bounded over $( (x)_{X},(d)_{D},(r)_{R})_{(e_0,S_0)}$ then we say that $( (x)_{X},(d)_{D},(r)_{R})_{(e_0,S_0)}$ is {\bf {\itshape locally totally bounded}}.
\end{definition}

% definition of  distance between sets (Haussdorf distance) 

\begin{definition}
Given
\begin{itemize}
	\item metric space $( (x)_{X},(d)_{D},(r)_{R})_{(e_0,S_0)}$
	\item subsets $(y)_{Y}$ and $(z)_{Z}$ of $(x)_{X}$ over $(e_0,S_0)$
	\item real number $(r_1)_{R_1}$ over $(e_0,S_0)$
\end{itemize}
if for any set $(k)_{K}$ over $(e_0,S_0)$ such that
\begin{itemize}
	\item $\forall j(j\overset{}{\underset{ (e_0,S_0)} {\in } }(k)_{K} \Rightarrow \exists y_1\overset{}{\underset{ (e_0,S_0)} {\in } }(y)_{Y} \exists z_1\overset{}{\underset{ (e_0,S_0)} {\in } }(z)_{Z}  ((j)_{K}\overset{\mathbb{R}}{\underset{ (e_0,S_0)} {= } } ((d)_{D}(y_1,z_1))_{R}) )$
	\item $\forall y_1\overset{}{\underset{ (e_0,S_0)} {\in } }(y)_{Y} \forall z_1\overset{}{\underset{ (e_0,S_0)} {\in } }(z)_{Z} \exists j\overset{}{\underset{ (e_0,S_0)} {\in } }(k)_{K}( (j)_{K}\overset{\mathbb{R}}{\underset{ (e_0,S_0)} {= } }((d)_{D}(y_1,z_1))_{R})$
\end{itemize}
we have 
$(r_1)_{R_1}\overset{\mathbb{R}}{\underset{ (e_0,S_0)} { =} }inf (k)_{K}$
then we say that $(r_1)_{R_1}$ is the {\bf {\itshape (Haussdorf) distance between $(y)_{Y}$ and $(z)_{Z}$ over $( (x)_{X},(d)_{D},(r)_{R})_{(e_0,S_0)}$}}.
\end{definition}

% definition of connected

\begin{definition}
Given a metric space $( (x)_{X},(d)_{D},(r)_{R} )_{(e_0,S_0)}$ and a subset $(y)_{Y}\overset{}{\underset{ (e_0,S_0)} {\subseteq } }(x)_{X}$, if for any non-empty sets $(a)_{A}$ and $(b)_{B}$ over $(e_0,S_0)$ with $(a)_{A}\cup (b)_{B} \overset{}{\underset{ (e_0,S_0)} {= } } (y)_{Y}$ and $(a)_{A}\cap (b)_{B}\overset{}{\underset{ (e_0,S_0)} {= } }\emptyset$ the distance between $(a)_{A}$ and $(b)_{B}$ is $0$ (over $( (x)_{X},(d)_{D},(r)_{R} )_{(e_0,S_0)}$) then we say that $(y)_{Y}$ is {\bf {\itshape connected over $( (x)_{X},(d)_{D},(r)_{R} )_{(e_0,S_0)}$}}. If $(y)_{Y}\overset{}{\underset{ (e_0,S_0)} {= } }(x)_{X}$ we will just say that $( (x)_{X},(d)_{D},(r)_{R} )_{(e_0,S_0)}$ is {\bf {\itshape connected}}.
\end{definition}

% definition of locally connected

\begin{definition}
Given a metric space $( (x)_{X},(d)_{D},(r)_{R} )_{(e_0,S_0)}$, if every bounded subset of $(x)_{X}$ is connected then we say that $( (x)_{X},(d)_{D},(r)_{R})_{(e_0,S_0)}$ is {\bf {\itshape locally connected}}.
\end{definition}

% (8) Theorem: Local connectedness implies connectedness
\begin{thm}
Every locally connected metric space $( (x)_{X},(d)_{D},(r)_{R} )_{(e_0,S_0)}$ is connected.
\end{thm}

\begin{proof}

		%%% IDEA
Let $(c_1)_{C_1}\overset{}{\underset{ (e_0,S_0)} {\neq } }\emptyset$ and $(c_2)_{C_2}\overset{}{\underset{ (e_0,S_0)} {\neq } }\emptyset$ be such that $(c_1)_{C_1}\cap (c_2)_{C_2} \overset{}{\underset{ (e_0,S_0)} {= } }\emptyset$ and $(c_1)_{C_1} \cup (c_2)_{C_2} \overset{}{\underset{ (e_0,S_0)} {= } }(x)_{X}$.
We can choose $x_1\overset{}{\underset{ (e_0,S_0)} {\in } }(x)_{X}$ and a real number $(s)_{S}\overset{\mathbb{R}}{\underset{ (e_0,S_0)} {> } }0$ such that the open ball $(b)_{B}$ of radius $(s)_{S}$ centered at $x_1$ intersects both $(c_1)_{C_1}$ and $(c_2)_{C_2}$.
Since $(x)_{X}$ is locally connected then for any $(\epsilon)_{E}\overset{\mathbb{R}}{\underset{ (e_0,S_0)} {> } }0$ there exist $y_1\overset{}{\underset{ (e_0,S_0)} {\in } }(b)_{B}\cap (c_1)_{C_1}\overset{}{\underset{ (e_0,S_0)} {\subseteq } }(c_1)_{C_1}$ and $y_2\overset{}{\underset{ (e_0,S_0)} {\in } }(b)_{B}\cap (c_2)_{C_2} \overset{}{\underset{ (e_0,S_0)} {\subseteq } }(c_2)_{C_2}$ such that $((d)_{D}(y_1,y_2))_{R}\overset{\mathbb{R}}{\underset{ (e_0,S_0)} {< } }\epsilon$. We conclude that $(x)_{X}$ must be connected.
		
		%%%

\end{proof}

	% connectedness doesn't necessarly imply local connectedness
	
	%%%
	Note that connectedness does not necessarily imply local connectedness.
	
	%%%

%%%%%% TO MOVE %%%%%%%%%

%%%%%%%% / TO MOVE %%%%%%%%%%%%%%

% definition of product metric

\begin{definition}
Let $( (x)_{X},(d_1)_{D_1},(r_1)_{R_1} )_{(e_0,S_0)}$, $( (y)_{Y},(d_2)_{D_2},(r_2)_{R_2})_{(e_0,S_0)}$ and \\ $( (z)_{Z},(d_3)_{D_3},(r_3)_{R_3})_{(e_0,S_0)}$ be metric spaces with $(z)_{Z}\overset{}{\underset{ (e_0,S_0)} {= } }(x)_{X}\times (y)_{Y}$.
If
\begin{itemize}
	\item for any real number $(\epsilon)_{E_1}\overset{\mathbb{R}}{\underset{ (e_0,S_0)} {> } }0$ there exists $(\delta)_{E_2}\overset{\mathbb{R}}{\underset{ (e_0,S_0)} {> } }0$ such that for any $x_1\overset{}{\underset{ (e_0,S_0)} {\in } }(x)_{X}$, $x_2\overset{}{\underset{ (e_0,S_0)} {\in } }(x)_{X}$, $y_1\overset{}{\underset{ (e_0,S_0)} {\in } }(y)_{Y}$ and $y_2\overset{}{\underset{ (e_0,S_0)} {\in } }(y)_{Y}$ we have \\ $(((d_1)_{D_1}(x_1,x_2))_{R_1}\overset{\mathbb{R}}{\underset{ (e_0,S_0)} {< } }(\delta)_{E_2} \wedge  ((d_2)_{D_2}(y_1,y_2))_{R_2}\overset{\mathbb{R}}{\underset{ (e_0,S_0)} {< } }(\delta)_{E_2}) \Rightarrow \\ ((d_3)_{D_3}((x_1,y_1),(x_2,y_2)))_{R_3}\overset{\mathbb{R}}{\underset{ (e_0,S_0)} {< } }(\epsilon)_{E_1}$	\text{ (uniform approachability) }
	\item for any $x_1\overset{}{\underset{ (e_0,S_0)} {\in } }(x)_{X}$, $x_2\overset{}{\underset{ (e_0,S_0)} {\in } }(x)_{X}$, $(y_1)\overset{}{\underset{ (e_0,S_0)} {\in } }(y)_{Y}$ and $(y_2)\overset{}{\underset{ (e_0,S_0)} {\in } }(y)_{Y}$ we have \\ $((d_3)_{D_3}( (x_1,y_1),(x_1,y_2) ))_{R_3}\overset{\mathbb{R}}{\underset{ (e_0,S_0)} {=} }((d_2)_{D_2}(y_1,y_2))_{R_2}$ and
	$((d_3)_{D_3}( (x_1,y_1),(x_2,y_1) ))_{R_3}\overset{\mathbb{R}}{\underset{ (e_0,S_0)} {= } }((d_1)_{D_1}(x_1,x_2))_{R_1}$	\text{ (coordinate reduction) }
	\item for any $x_1\overset{}{\underset{ (e_0,S_0)} {\in } }(x)_{X}$, $x_2\overset{}{\underset{ (e_0,S_0)} {\in } }(x)_{X}$, $y_1\overset{}{\underset{ (e_0,S_0)} {\in } }(y)_{Y}$ and $y_2\overset{}{\underset{ (e_0,S_0)} {\in } }(y)_{Y}$ we have \\ $((d_3)_{D_3}( (x_1,y_1),(x_2,y_2) ))_{R_3}\overset{\mathbb{R}}{\underset{ (e_0,S_0)} {\geq } }((d_1)_{D_1}(x_1,x_2))_{R_1}$ and $((d_3)_{D_3}( (x_1,y_1),(x_2,y_2) ))_{R_3}\overset{\mathbb{R}}{\underset{ (e_0,S_0)} {\geq } }((d_2)_{D_2}(y_1,y_2))_{R_2}$	\text{ (coordinate bounding) }
\end{itemize}
then we say that $( (z)_{Z},(d_3)_{D_3},(r_3)_{R_3})_{(e_0,S_0)}$ is a {\bf {\itshape product metric of $( (x)_{X},(d_1)_{D_1},(r_1)_{R_1})_{(e_0,S_0)}$ and $( (y)_{Y},(d_2)_{D_2},(r_2)_{R_2} )_{(e_0,S_0)}$}}.
\end{definition}

% (2) Theorem: Product of connected spaces is connected

\begin{thm}
Given
\begin{itemize}
	\item metric spaces $( (x)_{X},(d_1)_{D_1},(r_1)_{R_1} )_{(e_0,S_0)}$ and $( (y)_{Y},(d_2)_{D_2},(r_2)_{R_2} )_{(e_0,S_0)}$
	\item product metric $( (p)_{P},(d_3)_{D_3},(r_3)_{R_3})_{(e_0,S_0)}$ of $( (x)_{X},(d_1)_{D_1},(r_1)_{R_1} )_{(e_0,S_0)}$ and \\ $( (y)_{Y},(d_2)_{D_2},(r_2)_{R_2} )_{(e_0,S_0)}$
\end{itemize}
if $( (x)_{X},(d_1)_{D_1},(r_1)_{R_1} )_{(e_0,S_0)}$ and $( (y)_{Y},(d_2)_{D_2},(r_2)_{R_2} )_{(e_0,S_0)}$ are connected then \\ $( (p)_{P},(d_3)_{D_3},(r_3)_{R_3})_{(e_0,S_0)}$ is connected.
\end{thm}

\begin{proof}
Suppose $(p)_{P}\overset{}{\underset{ (e_0,S_0)} {= } }(z_1)_{Z_1} \cup (z_2)_{Z_2} $ where $(z_1)_{Z_1}$ and $(z_2)_{Z_2}$ are both non-empty over $(e_0,S_0)$ and $(z_1)_{Z_1}\cap (z_2)_{Z_2} \overset{}{\underset{ (e_0,S_0)} { =} }\emptyset$. We will consider two cases.

					%%% 
					\underline{Case 1}: There exist $a\overset{}{\underset{ (e_0,S_0)} {\in } }(x)_{X}$ and non-empty (over $(e_0,S_0)$) sets  $(b)_{B}\overset{}{\underset{ (e_0,S_0)} {\subseteq } }(y)_{Y}$ and $(c)_{C}\overset{}{\underset{ (e_0,S_0)} {\subseteq } }(y)_{Y}$ such that $(b)_{B} \cup (c)_{C} \overset{}{\underset{ (e_0,S_0)} {=} }(y)_{Y}$, $(b)_{B} \cap (c)_{C} \overset{}{\underset{ (e_0,S_0)} {= } }\emptyset$ and
					$\forall b_1\overset{}{\underset{ (e_0,S_0)} {\in } }(b)_{B} \exists p_1\overset{}{\underset{ (e_0,S_0)} {\in } }(z_1)_{Z_1}( (p_1)_{P}\overset{}{\underset{ (e_0,S_0)} {= } }(a,b_1)  )$ 
					and $\forall c_1\overset{}{\underset{ (e_0,S_0)} {\in } }(c)_{C} \exists p_1\overset{}{\underset{ (e_0,S_0)} {\in } }(z_2)_{Z_2}( (p_1)_{P}\overset{}{\underset{ (e_0,S_0)} {= } } (a,c_1) )$.
					%%%

					%%%
					\underline{Case 2}: For each $a\overset{}{\underset{ (e_0,S_0)} {\in } }(x)_{X}$ we either have $\forall y_1\overset{}{\underset{ (e_0,S_0)} {\in } }(y)_{Y} ( (a,y_1)\overset{}{\underset{ (e_0,S_0)} {\in } }(z_1)_{Z_1})$ or $\forall y_1\overset{}{\underset{ (e_0,S_0)} {\in } }(y)_{Y} ( (a,y_1)\overset{}{\underset{ (e_0,S_0)} {\in } }(z_2)_{Z_2})$.
					%%%
				   The second case implies that
					
					%%%
					there exist $a\overset{}{\underset{ (e_0,S_0)} {\in } }(y)_{Y}$ and non-empty (over $(e_0,S_0)$) sets  $(b)_{B}\overset{}{\underset{ (e_0,S_0)} {\subseteq } }(x)_{X}$ and $(c)_{C}\overset{}{\underset{ (e_0,S_0)} {\subseteq } }(x)_{X}$ such that $(b)_{B} \cup (c)_{C} \overset{}{\underset{ (e_0,S_0)} {=} }(x)_{X}$, $(b)_{B} \cap (c)_{C} \overset{}{\underset{ (e_0,S_0)} {= } }\emptyset$ and 
					$\forall b_1\overset{}{\underset{ (e_0,S_0)} {\in } }(b)_{B} \exists p_1\overset{}{\underset{ (e_0,S_0)} {\in } }(z_1)_{Z_1}( (p_1)_{P}\overset{}{\underset{ (e_0,S_0)} {= } }(b_1,a)  )$ 
					and $\forall c_1\overset{}{\underset{ (e_0,S_0)} {\in } }(c)_{C} \exists p_1\overset{}{\underset{ (e_0,S_0)} {\in } }(z_2)_{Z_2}( (p_1)_{P}\overset{}{\underset{ (e_0,S_0)} {= } } (c_1,a) )$.
					%%%
					
					%%%
					In the first case, by connectedness of $(y)_{Y}$ and by definition of product metric (using (coordinate reduction)) we see that the distance between $(b)_{B}$ and $(c)_{C}$ is 0.
					In the second case, by connectedness of $(x)_{Y}$ and by definition of product metric (using (coordinate reduction)) we see that the distance between $(b)_{B}$ and $(c)_{C}$ is 0. So we conclude that $( (p)_{P},(d_3)_{D_3},(r_3)_{R_3})_{(e_0,S_0)}$ is connected.
					%%%
					
\end{proof}

% (3) Theorem: Product of totally bounded sets is totally bounded
\begin{thm}
Given
\begin{itemize}
	\item metric spaces $( (x)_{X},(d_1)_{D_1},(r_1)_{R_1} )_{(e_0,S_0)}$ and $( (y)_{Y},(d_2)_{D_2},(r_2)_{R_2} )_{(e_0,S_0)}$
	\item product metric $( (p)_{P},(d_3)_{D_3},(r_3)_{R_3})_{(e_0,S_0)}$ of $( (x)_{X},(d_1)_{D_1},(r_1)_{R_1} )_{(e_0,S_0)}$ and \\ $( (y)_{Y},(d_2)_{D_2},(r_2)_{R_2} )_{(e_0,S_0)}$
	\item subsets $(x_1)_{X_1}\overset{}{\underset{ (e_0,S_0)} {\subseteq } }(x)_{X}$ and $(y_1)_{Y_1}\overset{}{\underset{ (e_0,S_0)} {\subseteq } }(y)_{Y}$
\end{itemize}
if $(x_1)_{X_1}$ is totally bounded over $( (x)_{X},(d_1)_{D_1},(r_1)_{R_1} )_{(e_0,S_0)}$ and $(y_1)_{Y_1}$ is totally bounded over $( (y)_{Y},(d_2)_{D_2},(r_2)_{R_2} )_{(e_0,S_0)}$ then $(z)_{Z}\overset{}{\underset{ (e_0,S_0)} { =} }(x_1)_{X_1} \times (y_1)_{Y_1}$ is totally bounded over \\ $( (p)_{P},(d_3)_{D_3},(r_3)_{R_3} )_{(e_0,S_0)}$.
\end{thm}

\begin{proof}

			%%%
			Let $(\epsilon)_{E_1}\overset{\mathbb{R}}{\underset{ (e_0,S_0)} {> } }0$ be a real number. By (uniform approachability) there exists $(\delta)_{E_2}\overset{}{\underset{ (e_0,S_0)} {> } }0$ such that $\forall a_1\overset{}{\underset{ (e_0,S_0)} {\in } }(x)_{X}$, $\forall a_2\overset{}{\underset{ (e_0,S_0)} {\in } }(x)_{X}$, $\forall b_1\overset{}{\underset{ (e_0,S_0)} {\in } }(y)_{Y}$ and $\forall b_2\overset{}{\underset{ (e_0,S_0)} {\in } }(y)_{Y}$ we have \\ $(((d_1)_{D_1}(a_1,a_2))_{R_1}\overset{\mathbb{R}}{\underset{ (e_0,S_0)} {< } }(\delta)_{E_2} \wedge ( (d_2)_{D_2}(b_1,b_2))_{R_2}\overset{\mathbb{R}}{\underset{ (e_0,S_0)} {< } }(\delta)_{E_2})\Rightarrow ((d_3)_{D_3}( (a_1,b_1),(a_2,b_2)))_{R_3}\overset{\mathbb{R}}{\underset{ (e_0,S_0)} {< } }(\epsilon)_{E_1}$.
			
			If $(x_1)_{X_1}$ is totally bounded over $( (x)_{X},(d_1)_{D_1},(r_1)_{R_1} )_{(e_0,S_0)}$ and $(y_1)_{Y_1}$ is totally bounded over $( (y)_{Y},(d_2)_{D_2},(r_2)_{R_2} )_{(e_0,S_0)}$ then there exists
a finite cover $(c_1)_{C_1}$ of $(x_1)_{X_1}$ by balls of radius less than $(\delta)_{E_2}$ and a finite cover $(c_2)_{C_2}$ of $(y_1)_{Y}$ by balls of radius less than  $(\delta)_{E_2}$. Let $(c)_{C}$ be a set over $(e_0,S_0)$ such that

\begin{itemize}
	\item $\forall k\overset{}{\underset{ (e_0,S_0)} {\in } }(c)_{C} \exists a_1\overset{}{\underset{ (e_0,S_0)} {\in } }(c_1)_{C_1} \exists a_2\overset{}{\underset{ (e_0,S_0)} {\in } }(c_2)_{C_2} ( (k)_{C}\overset{}{\underset{ (e_0,S_0)} {= } }(a_1)_{C_1} \times (a_2)_{C_2}  )$
	\item $\forall a_1\overset{}{\underset{ (e_0,S_0)} {\in } }(c_1)_{C_1} \forall a_2\overset{}{\underset{ (e_0,S_0)} {\in } }(c_2)_{C_2} \exists k\overset{}{\underset{ (e_0,S_0)} {\in } }(c)_{C} ( (k)_{C}\overset{}{\underset{ (e_0,S_0)} {= } }(a_1)_{C_1} \times (a_2)_{C_2}  ) $
	\item $\forall k_1\overset{}{\underset{ (e_0,S_0)} {\in } }(c)_{C} \forall k_2\overset{}{\underset{ (e_0,S_0)} {\in } }(c)_{C} ( (k_1)_{C}\overset{}{\underset{ (e_0,S_0)} {= } }(k_2)_{C} \Rightarrow k_1=k_2 )$
\end{itemize}

 We see that $(c)_{C}$ is a finite cover of $(z)_{Z}\overset{}{\underset{ (e_0,S_0)} {= } }(x_1)_{X_1} \times (y_1)_{Y_1}$ over  $( (p)_{P},(d_3)_{D_3},(r_3)_{R_3} )_{(e_0,S_0)}$, and for any $e\overset{}{\underset{ (e_0,S_0)} {\in } }(c)_{C}$, any $(a_1,a_2)\overset{}{\underset{ (e_0,S_0)} {\in } }(e)_{C}$ and any $(b_1,b_2)\overset{}{\underset{ (e_0,S_0)} {\in } }(e)_{C}$ we have \\ $((d_3)_{D_3}((a_1,a_2),(b_1,b_2)))_{R_3}\overset{\mathbb{R}}{\underset{ (e_0,S_0)} {< } }(\epsilon)_{E_1}$. We conclude that $(z)_{Z}\overset{}{\underset{ (e_0,S_0)} { =} }(x_1)_{X_1} \times (y_1)_{Y_1}$ is totally bounded over $( (p)_{P},(d_3)_{D_3},(r_3)_{R_3} )_{(e_0,S_0)}$.
			%%%

\end{proof}

% definition of limit point
\begin{definition}
For a given metric space $( (x)_{X},(d)_{D},(r)_{R} )_{(e_0,S_0)}$, a subset $(y)_{Y}$ of $(x)_{X}$ over $(e_0,S_0)$ and a point $z_1\overset{}{\underset{ (e_0,S_0)} {\in } }(x)_{X}$, if for any $(\epsilon)_{E} \overset{\mathbb{R}}{\underset{ (e_0,S_0)} {> } }0$ there exists a point $z_2\overset{}{\underset{ (e_0,S_0)} {\in } }(y)_{Y}$ with  $z_2\neq z_1$ such that $((d)_{D}((z_1,z_2)))_{R}\overset{\mathbb{R}}{\underset{ (e_0,S_0)} {< } }(\epsilon)_{E}$ then we say that $z_1$ is a {\bf {\itshape limit point of $(y)_{Y}$ over $( (x)_{X},(d)_{D},(r)_{R} )_{(e_0,S_0)}$}}.
\end{definition}

% closure

\begin{definition}
Given
\begin{itemize} 
	\item metric space $( (x)_{X},(d)_{D},(r)_{R} )_{(e_0,S_0)}$
	\item subsets $(a)_{A}$ and $(c)_{C}$ of $(x)_{X}$ over $(e_0,S_0)$
\end{itemize}
if 
$$ (c)_{C}\overset{}{\underset{ (e_0,S_0)} {= } } \{ p \mid \text{ $p\overset{}{\underset{ (e_0,S_0)} {\in } }(a)_{A}$ or $p$ is a limit point of $(a)_{A}$ over $( (x)_{X},(d)_{D},(r)_{R})_{(e_0,S_0)}$} \} $$
then we say that $(c)_{C}$ is the {\bf {\itshape closure of $(a)_{A}$ over $( (x)_{X},(d)_{D},(r)_{R})_{(e_0,S_0)}$}} or write  $(c)_{C}\overset{}{\underset{ (e_0,S_0)} {=} }[(a)_{A}]_{( (x)_{X},(d)_{D},(r)_{R} )_{(e_0,S_0)}}$ or $(c)_{C}\overset{}{\underset{ (e_0,S_0)} {=} }[(a)_{A}]$ when there is no ambiguity.
\end{definition}

% definition of Cauchy-convergent
\begin{definition}
Given a metric space $( (x)_{X},(d)_{D},(r)_{R} )_{(e_0,S_0)}$ and a function $(f)_{F}:\mathbb{N}\overset{}{\underset{ (e_0,S_0)} {\rightarrow } }(y)_{Y}$ with $(y)_{Y}\overset{}{\underset{ (e_0,S_0)} {\subseteq } }(x)_{X}$, if for every real number $(\epsilon)_{E}\overset{\mathbb{R}}{\underset{ (e_0,S_0)} { >} }0$ there exists $m$ such that $\forall n_1 \forall n_2 ( (n_1>m \wedge n_2>m)\Rightarrow ((d)_{D} ( (f)_{F}(n_1),(f)_{F}(n_2) ))_{R} \overset{\mathbb{R}}{\underset{ (e_0,S_0)} {< } }(\epsilon)_{E} )$ then we say that $(f)_{F}:\mathbb{N}\overset{}{\underset{ (e_0,S_0)} {\rightarrow } }(y)_{Y}$ is {\bf {\itshape Cauchy-convergent over \\ $( (x)_{X},(d)_{D},(r)_{R} )_{(e_0,S_0)}$}}.
\end{definition}

% Definition of convergence

\begin{definition}
Given
\begin{itemize} 
	\item metric space $( (x)_{X},(d)_{D},(r)_{R} )_{(e_0,S_0)}$ 
	\item function $(f)_{F}:\mathbb{N}\overset{}{\underset{ (e_0,S_0)} {\rightarrow } }(y)_{Y}$ with  $(y)_{Y}\overset{}{\underset{ (e_0,S_0)} {\subseteq } }(x)_{X}$
	\item $x_1\overset{}{\underset{ (e_0,S_0)} {\in } }(x)_{X}$
\end{itemize} 

if for every real number $(\epsilon)_{E}\overset{\mathbb{R}}{\underset{ (e_0,S_0)} { >} }0$ there exists $m$ such that \\ $\forall n( n>m \Rightarrow ((d)_{D} ( (f)_{F}(n),x_1 ))_{R} \overset{\mathbb{R}}{\underset{ (e_0,S_0)} {< } }(\epsilon)_{E} )$ then we say that $(f)_{F}:\mathbb{N}\overset{}{\underset{ (e_0,S_0)} {\rightarrow } }(y)_{Y}$ {\bf {\itshape converges to $x_1$ over $( (x)_{X},(d)_{D},(r)_{R} )_{(e_0,S_0)}$}}.
\end{definition}

%%%%%%%%%% TO MOVE %%%%%%%%%%%

%%%%%%%% / TO MOVE %%%%%%%%%

% metric extension

\begin{definition}
Given an isometry \\ $(f)_{F}:( (x)_{X},(d_1)_{D_1},(r_1)_{R_1} )_{(e_0,S_0)}\overset{}{\underset{ (e_0,S_0)} {\rightarrow } }( (y)_{Y},(d_2)_{D_2},(r_2)_{R_2} )_{(e_0,S_0)}$, if 
$$\forall y_1\overset{}{\underset{ (e_0,S_0)} {\in } }(y)_{Y}( y_1\overset{}{\underset{ (e_0,S_0)} {\notin } }(f)_{F}((x)_{X}) \Rightarrow \text{ $y_1$ is a limit point of $(f)_{F}((x)_{X})$ } )$$
then we say that $( (y)_{Y},(d_2)_{D_2},(r_2)_{R_2} )_{(e_0,S_0)}$ is an {\bf {\itshape extension of $( (x)_{X},(d_1)_{D_1},(r_1)_{R_1} )_{(e_0,S_0)}$ via $(f)_{F}$ }}.
\end{definition}

Now we define the concept of extension for relations.

% definition of relation extension
\begin{definition}

Given
\begin{itemize}
	\item metric spaces $((x_1)_{X_1},(d_1)_{D_1},(r_1)_{R_1})_{(e_0,S_0)}$ and $((y_1)_{Y_1},(d_2)_{D_2},(r_2)_{R_2})_{(e_0,S_0)}$
	\item extension $((x_2)_{X_2},(d_3)_{D_3},(r_3)_{R_3})_{(e_0,S_0)}$ of $((x_1)_{X_1},(d_1)_{D_1},(r_1)_{R_1})_{(e_0,S_0)}$ via $(f)_{F}$
	\item extension $((y_2)_{Y_2},(d_4)_{D_4},(r_4)_{R_4})_{(e_0,S_0)}$ of $((y_1)_{Y_1},(d_2)_{D_2},(r_2)_{R_2})_{(e_0,S_0)}$ via $(g)_{G}$
	\item relations $((x_1)_{X_1},(y_1)_{Y_1},(t_1)_{T_1})_{(e_0,S_0)}$ and $((x_2)_{X_2},(y_2)_{Y_2},(t_2)_{T_2})_{(e_0,S_0)}$
\end{itemize}
if
\begin{itemize}
	\item $\forall a_1 \forall b_1 ( (a_1,b_1)\overset{}{\underset{ (e_0,S_0)} {\in } }(t_1)_{T_1}\Rightarrow ( (f)_{F}(a_1),(g)_{G}(b_1))\overset{}{\underset{ (e_0,S_0)} {\in } }(t_2)_{T_2} )$
	\item for any $(a_2,b_2)\overset{}{\underset{ (e_0,S_0)} {\in } }(t_2)_{T_2}$, if $\forall a_1 \forall b_1 (((f)_{F}(a_1),(g)_{G}(b_1))\overset{}{\underset{ (e_0,S_0)} {\neq } }(a_2,b_2))$ \\ then for any $(\epsilon)_{E}\overset{\mathbb{R}}{\underset{ (e_0,S_0)} {> } }0$ there exists $(a_1,b_1)\overset{}{\underset{ (e_0,S_0)} {\in } }(t_1)_{T_1}$ with $((f)_{F}(a_1),(g)_{G}(b_1))\overset{}{\underset{ (e_0,S_0)} {\neq } }(a_2,b_2)$ such that  $((d_3)_{D_3} ( (f)_{F}(a_1),a_2 ))_{R_3}\overset{\mathbb{R}}{\underset{ (e_0,S_0)} {< } }(\epsilon)_{E}$ and $((d_4)_{D_4} ( (g)_{G}(b_1),b_2 ))_{R_4}\overset{\mathbb{R}}{\underset{ (e_0,S_0)} {< } }(\epsilon)_{E}$

\end{itemize}
then we say that
$( (t_2)_{T_2},( (x_2)_{X_2},(d_3)_{D_3},(r_3)_{R_3} ),( (y_2)_{Y_2},(d_4)_{D_4},(r_4)_{R_4} ))_{(e_0,S_0)}$ is a  {\bf {\itshape (binary) relation extension of }} $( (t_1)_{T_1},( (x_1)_{X_1},(d_1)_{D_1},(r_1)_{R_1} ),( (y_1)_{Y_1},(d_2)_{D_2},(r_2)_{R_2}) )_{(e_0,S_0)}$  {\bf {\itshape via}} $( (f)_{F},(g)_{G} )$.

\end{definition}

% (1.1) Theorem: Product metric on extensions implies existence of extension of product metric
\begin{thm}
Given
\begin{itemize}
	\item metric spaces $( (x_1)_{X_1},(d_1)_{D_1},(r_1)_{R_1} )_{(e_0,S_0)}$ and $( (y_1)_{Y_1},(d_2)_{D_2},(r_2)_{R_2} )_{(e_0,S_0)}$
	\item extension $( (x_2)_{X_2},(d_3)_{D_3},(r_3)_{R_3} )_{(e_0,S_0)}$ of $( (x_1)_{X_1},(d_1)_{D_1},(r_1)_{R_1} )_{(e_0,S_0)}$ via $(f)_{F}$
	\item extension $( (y_2)_{Y_2},(d_4)_{D_4},(r_4)_{R_4} )_{(e_0,S_0)}$ of $((y_1)_{Y_1},(d_2)_{D_2},(r_2)_{R_2})_{(e_0,S_0)}$ via $(g)_{G}$
	\item product metric $( (p_1)_{P_1},(d_5)_{D_5},(r_5)_{R_5} )_{(e_0,S_0)}$ of $( (x_1)_{X_1},(d_1)_{D_1},(r_1)_{R_1} )_{(e_0,S_0)}$ and \\ $( (y_1)_{Y_1},(d_2)_{D_2},(r_2)_{R_2} )_{(e_0,S_0)}$
\end{itemize}
there exists a product metric $( (p_2)_{P_2},(d_6)_{D_6},(r_6)_{R_6} )_{(e_0,S_0)}$ of $( (x_2)_{X_2},(d_3)_{D_3},(r_3)_{R_3} )_{(e_0,S_0)}$ and \\ $( (y_2)_{Y_2},(d_4)_{D_4},(r_4)_{R_4} )_{(e_0,S_0)}$ such that $( (p_2)_{P_2},(d_6)_{D_6},(r_6)_{R_6})_{(e_0,S_0)}$ is an extension of \\ $( (p_1)_{P_1},(d_5)_{D_5},(r_5)_{R_5} )_{(e_0,S_0)}$ via some $(i)_{I}$ where for all $a\overset{}{\underset{ (e_0,S_0)} {\in } }(x_1)_{X_1}$ and all $b\overset{}{\underset{ (e_0,S_0)} {\in } }(y_1)_{Y_1}$ we have
$$ (i)_{I}(a,b)\overset{}{\underset{ (e_0,S_0)} {= } }( (f)_{F}(a),(g)_{G}(b) ) $$
\end{thm}

\begin{proof}
Let $(p_2)_{P_2}\overset{}{\underset{ (e_0,S_0)} {= } }(x_2)_{X_2} \times (y_2)_{Y_2}$ and $(i)_{I}$ be such that
for all $a\overset{}{\underset{ (e_0,S_0)} {\in } }(x_1)_{X_1}$ and all $b\overset{}{\underset{ (e_0,S_0)} {\in } }(y_1)_{Y_1}$ we have
$$ (i)_{I}(a,b)\overset{}{\underset{ (e_0,S_0)} {= } }( (f)_{F}(a),(g)_{G}(b) ) $$ 
We define $(d_6)_{D_6}$ and use a suitable real set $(r_6)_{R_6}$ over $(e_0,S_0)$ such that for any  $(v)_{P_2}\overset{}{\underset{ (e_0,S_0)} {= } }(v_1,v_2)\overset{}{\underset{ (e_0,S_0)} {\in } }(p_2)_{P_2}$ and any $(w)_{P_2}\overset{}{\underset{ (e_0,S_0)} { =} }(w_1,w_2)\overset{}{\underset{ (e_0,S_0)} {\in } }(p_2)_{P_2}$ we have
				  $$ ((d_6)_{D_6}(v,w))_{R_6}\overset{\mathbb{R}}{\underset{ (e_0,S_0)} {= } }  \underset{n\rightarrow +\infty}{\text{lim}} sup\{((d_5)_{D_5}((t_1,t_2),(u_1,u_2)))_{R_5}\mid $$
				  
				 $ ((d_3)_{D_3}( v_1,(f)_{F}(t_1) ))_{R_3}\overset{\mathbb{R}}{\underset{ (e_0,S_0)} {< } }1/n,   
				  ((d_4)_{D_4}( v_2,(g)_{G}(t_2) ))_{R_4}\overset{\mathbb{R}}{\underset{ (e_0,S_0)} {< } }1/n, \\
				  ((d_3)_{D_3}( w_1,(f)_{F}(u_1) ))_{R_3}\overset{\mathbb{R}}{\underset{ (e_0,S_0)} {< } }1/n \text{ and }
				  ((d_4)_{D_4}( w_2,(g)_{G}(u_2) ))_{R_4}\overset{\mathbb{R}}{\underset{ (e_0,S_0)} {< } }1/n
				  \}$
				  
By (uniform approachability) we get that for any real number $(\epsilon)_{E_1}\overset{\mathbb{R}}{\underset{ (e_0,S_0)} {> } }0$ there exists $(\delta)_{E_2}\overset{\mathbb{R}}{\underset{ (e_0,S_0)} {> } }0$ such that for any $a_1\overset{}{\underset{ (e_0,S_0)} {\in } }(x_1)_{X_1}$, $a_2\overset{}{\underset{ (e_0,S_0)} {\in } }(x_1)_{X_1}$, $b_1\overset{}{\underset{ (e_0,S_0)} {\in } }(y_1)_{Y_1}$ and $b_2\overset{}{\underset{ (e_0,S_0)} {\in } }(y_1)_{Y_1}$ we have 
$$ ( ((d_1)_{D_1}(a_1,a_2))_{R_1}\overset{\mathbb{R}}{\underset{ (e_0,S_0)} {< } }(\delta)_{E_2} \wedge ((d_2)_{D_2}(b_1,b_2))_{R_2}\overset{\mathbb{R}}{\underset{ (e_0,S_0)} {< } }(\delta)_{E_2} )\Rightarrow $$
$$ ((d_5)_{D_5}((a_1,b_1),(a_2,b_2)))_{R_5}\overset{\mathbb{R}}{\underset{ (e_0,S_0)} {< } }(\epsilon)_{E_1}$$ So using the triangle inequality it follows that $(d_6)_{D_6}$ is well defined and $( (p_2)_{P_2},(d_6)_{D_6},(r_6)_{R_6} )_{(e_0,S_0)}$ is a product metric of $( (x_2)_{X_2},(d_3)_{D_3},(r_3)_{R_3} )_{(e_0,S_0)}$ and $( (y_2)_{Y_2},(d_4)_{D_4},(r_4)_{R_4} )_{(e_0,S_0)}$.
\end{proof}

The next two theorems show that relation extensions can be seen as extensions in the context of a product metric. In a certain sense, these theorems represent the core idea behind the closed graph theorem.

% (1.2) Theorem: Relation extension implies extension in product metric 
\begin{thm}

Given
\begin{itemize}
	\item metric spaces $( (x_1)_{X_1},(d_1)_{D_1},(r_1)_{R_1} )_{(e_0,S_0)}$ and $( (y_1)_{Y_1},(d_2)_{D_2},(r_2)_{R_2} )_{(e_0,S_0)}$
	\item extension $( (x_2)_{X_2},(d_3)_{D_3},(r_3)_{R_3} )_{(e_0,S_0)}$ of $( (x_1)_{X_1},(d_1)_{D_1},(r_1)_{R_1})_{(e_0,S_0)}$ via $(f)_{F}$
	\item extension $( (y_2)_{Y_2},(d_4)_{D_4},(r_4)_{R_4} )_{(e_0,S_0)}$ of $( (y_1)_{Y_1},(d_2)_{D_2},(r_2)_{R_2})_{(e_0,S_0)}$ via $(g)_{G}$
	\item relation $ ( (x_1)_{X_1},(y_1)_{Y_1},(t_1)_{T_1} )_{(e_0,S_0)}$
	\item relation extension $( (t_2)_{T_2},((x_2)_{X_2},(d_3)_{D_3},(r_3)_{R_3}),((y_2)_{Y_2},(d_4)_{D_4},(r_4)_{R_4}) )_{(e_0,S_0)}$ of \\ $( (t_1)_{T_1},((x_1)_{X_1},(d_1)_{D_1},(r_1)_{R_1}),((y_1)_{Y_1},(d_2)_{D_2},(r_2)_{R_2}) )_{(e_0,S_0)}$ via $((f)_{F},(g)_{G})$ 
\end{itemize}
for any product metric $( (p_1)_{P_1},(d_5)_{D_5},(r_5)_{R_5} )_{(e_0,S_0)}$ of $( (x_1)_{X_1},(d_1)_{D_1},(r_1)_{R_1} )_{(e_0,S_0)}$ and \\ $( (y_1)_{Y_1},(d_2)_{D_2},(r_2)_{R_2} )_{(e_0,S_0)}$ and any product metric $( (p_2)_{P_2},(d_6)_{D_6},(r_6)_{R_6})_{(e_0,S_0)}$ of \\ $( (x_2)_{X_2},(d_3)_{D_3},(r_3)_{R_3})_{(e_0,S_0)}$ and
$( (y_2)_{Y_2},(d_4)_{D_4},(r_4)_{R_4} )_{(e_0,S_0)}$ where \\ $( (p_2)_{P_2},(d_6)_{D_6},(r_6)_{R_6} )_{(e_0,S_0)}$ is an extension of $( (p_1)_{P_1},(d_5)_{D_5},(r_5)_{R_5} )_{(e_0,S_0)}$ via $(i)_{I}$ such that for all $a\overset{}{\underset{ (e_0,S_0)} {\in } }(x_1)_{X_1}$ and all $b\overset{}{\underset{ (e_0,S_0)} {\in } }(y_1)_{Y_1}$ we have
$$ (i)_{I}(a,b)\overset{}{\underset{ (e_0,S_0)} {= } }( (f)_{F}(a),(g)_{G}(b) ) $$
			
we must get that $( (t_2)_{T_2},(d_6)_{D_6}|(t_2)_{T_2},(r_6)_{R_6} )_{(e_0,S_0)}$ is an extension of \\ $( (t_1)_{T_1},(d_5)_{D_5}|(t_1)_{T_1},(r_5)_{R_5} )_{(e_0,S_0)}$ via $(i)_{I}|(t_1)_{T_1}$.

\end{thm}

\begin{proof}
Let $( (p_1)_{P_1},(d_5)_{D_5},(r_5)_{R_5} )_{(e_0,S_0)}$ be a product metric of $( (x_1)_{X_1},(d_1)_{D_1},(r_1)_{R_1} )_{(e_0,S_0)}$ and $( (y_1)_{Y_1},(d_2)_{D_2},(r_2)_{R_2} )_{(e_0,S_0)}$ and let $( (p_2)_{P_2},(d_6)_{D_6},(r_6)_{R_6})_{(e_0,S_0)}$ be a product metric of \\ $( (x_2)_{X_2},(d_3)_{D_3},(r_3)_{R_3})_{(e_0,S_0)}$ and
$( (y_2)_{Y_2},(d_4)_{D_4},(r_4)_{R_4} )_{(e_0,S_0)}$ such that $ ( (p_2)_{P_2},(d_6)_{D_6},(r_6)_{R_6} )_{(e_0,S_0)}$ is an extension of $( (p_1)_{P_1},(d_5)_{D_5},(r_5)_{R_5} )_{(e_0,S_0)}$ via $(i)_{I}$ where for all $a\overset{}{\underset{ (e_0,S_0)} {\in } }(x_1)_{X_1}$ and all $b\overset{}{\underset{ (e_0,S_0)} {\in } }(y_1)_{Y_1}$ we have  $ (i)_{I}(a,b)\overset{}{\underset{ (e_0,S_0)} {= } }( (f)_{F}(a),(g)_{G}(b) )$.

Let $(a_2,b_2)\overset{}{\underset{ (e_0,S_0)} {\in } }(t_2)_{T_2}$ such that $\forall a_1 \forall b_1 ( ((f)_{F}(a_1),(g)_{G}(b_1))\overset{}{\underset{ (e_0,S_0)} {\neq } }(a_2,b_2))$.

Let $(\epsilon)_{E_1}\overset{\mathbb{R}}{\underset{ (e_0,S_0)} {> } }0$ be an arbitrary real number. By (uniform approachability) there exists $(\delta)_{E_2}\overset{\mathbb{R}}{\underset{ (e_0,S_0)} {> } }0$ such that for any $v_1\overset{}{\underset{ (e_0,S_0)} {\in } }(x_2)_{X_2}$, $v_2\overset{}{\underset{ (e_0,S_0)} {\in } }(x_2)_{X_2}$, $w_1\overset{}{\underset{ (e_0,S_0)} {\in } }(y_2)_{Y_2}$ and $w_2\overset{}{\underset{ (e_0,S_0)} {\in } }(y_2)_{Y_2}$, if $((d_3)_{D_3}(v_1,v_2))_{R_3}\overset{\mathbb{R}}{\underset{ (e_0,S_0)} {< } }(\delta)_{E_2}$ and $((d_4)_{D_4}(w_1,w_2))_{R_4}\overset{\mathbb{R}}{\underset{ (e_0,S_0)} {< } }(\delta)_{E_2}$ then \\ $((d_6)_{D_6}((v_1,w_1),(v_2,w_2)))_{R_6}\overset{\mathbb{R}}{\underset{ (e_0,S_0)} {< } }(\epsilon)_{E_1}$.

By definition of relation extension there exists $(z_1,z_2)\overset{}{\underset{ (e_0,S_0)} {\in } }(t_1)_{T_1}$ with $((f)_{F}(z_1),(g)_{G}(z_2))\overset{}{\underset{ (e_0,S_0)} {\neq } }(a_2,b_2)$ such that   $((d_3)_{D_3} ( a_2,(f)_{F}(z_1) ))_{R_3}\overset{\mathbb{R}}{\underset{ (e_0,S_0)} {< } }(\delta)_{E_2}$ and $((d_4)_{D_4} ( b_2,(g)_{G}(z_2) ))_{R_4}\overset{\mathbb{R}}{\underset{ (e_0,S_0)} {< } }(\delta)_{E_2}$. So we obtain  $( (d_6)_{D_6}( (a_2,b_2),( (f)_{F}(z_1),(g)_{G}(z_2) )))_{R_6}\overset{}{\underset{ (e_0,S_0)} {< } }(\epsilon)_{E_1}$. 

It follows that $(a_2,b_2)$ is a limit point of $(t_2)_{T_2}$  over $( (t_2)_{T_2},(d_6)_{D_6}|(t_2)_{T_2},(r_6)_{R_6} )_{(e_0,S_0)}$ and we conclude that \\ $( (t_2)_{T_2},(d_6)_{D_6}|(t_2)_{T_2},(r_6)_{R_6} )_{(e_0,S_0)}$ is an extension of $( (t_1)_{T_1},(d_5)_{D_5}|(t_1)_{T_1},(r_5)_{R_5} )_{(e_0,S_0)}$ via $(i)_{I}|(t_1)_{T_1}$.

\end{proof}

% (1.3) Theorem: Relation equipped with extension of product metric implies relation is an extension
\begin{thm}

Given
\begin{itemize}
	\item metric spaces $( (x_1)_{X_1},(d_1)_{D_1},(r_1)_{R_1} )_{(e_0,S_0)}$ and $( (y_1)_{Y_1},(d_2)_{D_2},(r_2)_{R_2} )_{(e_0,S_0)}$
	\item extension $( (x_2)_{X_2},(d_3)_{D_3},(r_3)_{R_3} )_{(e_0,S_0)}$ of $( (x_1)_{X_1},(d_1)_{D_1},(r_1)_{R_1})_{(e_0,S_0)}$ via $(f)_{F}$
	\item extension $( (y_2)_{Y_2},(d_4)_{D_4},(r_4)_{R_4} )_{(e_0,S_0)}$ of $( (y_1)_{Y_1},(d_2)_{D_2},(r_2)_{R_2})_{(e_0,S_0)}$ via $(g)_{G}$
	\item relations $ ( (x_1)_{X_1},(y_1)_{Y_1},(t_1)_{T_1} )_{(e_0,S_0)}$ and
	$( (x_2)_{X_2},(y_2)_{Y_2},(t_2)_{T_2} )_{(e_0,S_0)}$ 
\end{itemize}
if
\begin{itemize}
	\item $\forall a_1 \forall b_1 ( (a_1,b_1)\overset{}{\underset{ (e_0,S_0)} {\in } }(t_1)_{T_1}\Rightarrow ( (f)_{F}(a_1),(g)_{G}(b_1))\overset{}{\underset{ (e_0,S_0)} {\in } }(t_2)_{T_2} )$
	\item for any product metric $( (p_1)_{P_1},(d_5)_{D_5},(r_5)_{R_5} )_{(e_0,S_0)}$ of $( (x_1)_{X_1},(d_1)_{D_1},(r_1)_{R_1} )_{(e_0,S_0)}$ and \\ $( (y_1)_{Y_1},(d_2)_{D_2},(r_2)_{R_2} )_{(e_0,S_0)}$ and any product metric $( (p_2)_{P_2},(d_6)_{D_6},(r_6)_{R_6})_{(e_0,S_0)}$ of \\ $( (x_2)_{X_2},(d_3)_{D_3},(r_3)_{R_3})_{(e_0,S_0)}$ and
$( (y_2)_{Y_2},(d_4)_{D_4},(r_4)_{R_4} )_{(e_0,S_0)}$ such that \\ $ ( (p_2)_{P_2},(d_6)_{D_6},(r_6)_{R_6} )_{(e_0,S_0)}$ is an extension of $( (p_1)_{P_1},(d_5)_{D_5},(r_5)_{R_5} )_{(e_0,S_0)}$ via $(i)_{I}$ where \\ $ \forall a\overset{}{\underset{ (e_0,S_0)} {\in } }(x_1)_{X_1} \forall b\overset{}{\underset{ (e_0,S_0)} {\in } }(y_1)_{Y_1}((i)_{I}(a,b)\overset{}{\underset{ (e_0,S_0)} {= } }( (f)_{F}(a),(g)_{G}(b) ))$ we have that \\ $( (t_2)_{T_2},(d_6)_{D_6}|(t_2)_{T_2},(r_6)_{R_6} )_{(e_0,S_0)}$ is an extension of $( (t_1)_{T_1},(d_5)_{D_5}|(t_1)_{T_1},(r_5)_{R_5} )_{(e_0,S_0)}$ via \\ $(i)_{I}|(t_1)_{T_1}$
\end{itemize}
then $( (t_2)_{T_2},((x_2)_{X_2},(d_3)_{D_3},(r_3)_{R_3}),((y_2)_{Y_2},(d_4)_{D_4},(r_4)_{R_4}) )_{(e_0,S_0)}$ is a relation extension of \\ $( (t_1)_{T_1},((x_1)_{X_1},(d_1)_{D_1},(r_1)_{R_1}),((y_1)_{Y_1},(d_2)_{D_2},(r_2)_{R_2}) )_{(e_0,S_0)}$ via $((f)_{F},(g)_{G})$
\end{thm}

\begin{proof}
Let $( (p_1)_{P_1},(d_5)_{D_5},(r_5)_{R_5} )_{(e_0,S_0)}$ be a product metric of $( (x_1)_{X_1},(d_1)_{D_1},(r_1)_{R_1} )_{(e_0,S_0)}$ and $( (y_1)_{Y_1},(d_2)_{D_2},(r_2)_{R_2} )_{(e_0,S_0)}$ and let  $( (p_2)_{P_2},(d_6)_{D_6},(r_6)_{R_6})_{(e_0,S_0)}$  be a product metric of \\ $( (x_2)_{X_2},(d_3)_{D_3},(r_3)_{R_3})_{(e_0,S_0)}$ and
$( (y_2)_{Y_2},(d_4)_{D_4},(r_4)_{R_4} )_{(e_0,S_0)}$ such that $ ( (p_2)_{P_2},(d_6)_{D_6},(r_6)_{R_6} )_{(e_0,S_0)}$ is an extension of $( (p_1)_{P_1},(d_5)_{D_5},(r_5)_{R_5} )_{(e_0,S_0)}$ via $(i)_{I}$ where \\ $ \forall a\overset{}{\underset{ (e_0,S_0)} {\in } }(x_1)_{X_1} \forall b\overset{}{\underset{ (e_0,S_0)} {\in } }(y_1)_{Y_1}((i)_{I}(a,b)\overset{}{\underset{ (e_0,S_0)} {= } }( (f)_{F}(a),(g)_{G}(b) ))$ with \\ $( (t_2)_{T_2},(d_6)_{D_6}|(t_2)_{T_2},(r_6)_{R_6} )_{(e_0,S_0)}$ an extension of $( (t_1)_{T_1},(d_5)_{D_5}|(t_1)_{T_1},(r_5)_{R_5} )_{(e_0,S_0)}$ via $(i)_{I}|(t_1)_{T_1}$

It will suffice to consider limit points of $(i)_{I}( (t_1)_{T_1} )$ over $( (t_2)_{T_2},(d_6)_{D_6}|(t_2)_{T_2},(r_6)_{R_6} )_{(e_0,S_0)}$ which are not in $(i)_{I}( (t_1)_{T_1} )$.

If $(w_1,w_2)$ is a limit point of $(i)_{I}( (t_1)_{T_1} )$ over $( (t_2)_{T_2},(d_6)_{D_6}|(t_2)_{T_2},(r_6)_{R_6} )$ with $(w_1,w_2)\overset{}{\underset{ (e_0,S_0)} {\notin } }(i)_{I}( (t_1)_{T_1} )$ then for any  $(\epsilon)_{E}\overset{\mathbb{R}}{\underset{ (e_0,S_0)} {> } }0$ there exists $(z_1,z_2)\overset{}{\underset{ (e_0,S_0)} {\in } }(i)_{I}( (t_1)_{T_1} )$ with $(w_1,w_2)\overset{}{\underset{ (e_0,S_0)} {\neq } }(z_1,z_2)$ such that $((d_6)_{D_6}((w_1,w_2),(z_1,z_2)))_{R_6}\overset{\mathbb{R}}{\underset{ (e_0,S_0)} {< } }(\epsilon)_{E}$ and since $ ( (p_2)_{P_2},(d_6)_{D_6},(r_6)_{R_6} )_{(e_0,S_0)}$ is a product metric we obtain by (coordinate bounding) that $((d_3)_{D_3}(w_1,z_1))_{R_3}\overset{\mathbb{R}}{\underset{ (e_0,S_0)} {< } }(\epsilon)_{E}$ and $((d_4)_{D_4}(w_2,z_2))_{R_4}\overset{\mathbb{R}}{\underset{ (e_0,S_0)} {< } }(\epsilon)_{E}$. So we conclude that \\ $( (t_2)_{T_2},((x_2)_{X_2},(d_3)_{D_3},(r_3)_{R_3}),((y_2)_{Y_2},(d_4)_{D_4},(r_4)_{R_4}) )_{(e_0,S_0)}$ is a relation extension of \\ $( (t_1)_{T_1},((x_1)_{X_1},(d_1)_{D_1},(r_1)_{R_1}),((y_1)_{Y_1},(d_2)_{D_2},(r_2)_{R_2}) )_{(e_0,S_0)}$ via $((f)_{F},(g)_{G})$.

\end{proof}

% Definition of continuity at x_1
\begin{definition}
Given
\begin{itemize}
	\item metric spaces $( (x)_{X},(d_1)_{D_1},(r_1)_{R_1} )_{(e_0,S_0)}$ and $( (y)_{Y},(d_2)_{D_2},(r_2)_{R_2} )_{(e_0,S_0)}$
	\item $x_1\overset{}{\underset{ (e_0,S_0)} {\in } }(x)_{X}$
	\item function $(f)_{F}:(x)_{X}\overset{}{\underset{ (e_0,S_0)} {\rightarrow } }(y)_{Y}$
\end{itemize}
if for any real number $(\epsilon)_{E_1}\overset{\mathbb{R}}{\underset{ (e_0,S_0)} {> } }0$ over $(e_0,S_0)$ there exists a real number $(\delta)_{E_2}\overset{\mathbb{R}}{\underset{ (e_0,S_0)} {> } }0$ over $(e_0,S_0)$ such that for any $x_2\overset{}{\underset{ (e_0,S_0)} {\in } }(x)_{X}$ we have
$$ ((d_1)_{D_1}( x_1,x_2 ))_{R_1}\overset{\mathbb{R}}{\underset{ (e_0,S_0)} {< } }(\delta)_{E_2} \Rightarrow ((d_2)_{D_2}( (f)_{F}(x_1),(f)_{F}(x_2)))_{R_2}\overset{\mathbb{R}}{\underset{ (e_0,S_0)} {< } }(\epsilon)_{E_1} $$
then we say that $(f)_{F}:(x)_{X}\overset{}{\underset{ (e_0,S_0)} {\rightarrow } }(y)_{Y}$ is {\bf {\itshape continuous at $x_1$ over \\ $(( (x)_{X},(d_1)_{D_1},(r_1)_{R_1} ),( (y)_{Y},(d_2)_{D_2},(r_2)_{R_2}))_{(e_0,S_0)}$}}.
\end{definition}

% Definition of continuity
\begin{definition}
Given
\begin{itemize}
	\item metric spaces $( (x)_{X},(d_1)_{D_1},(r_1)_{R_1} )_{(e_0,S_0)}$ and $( (y)_{Y},(d_2)_{D_2},(r_2)_{R_2} )_{(e_0,S_0)}$
	\item function $(f)_{F}:(x)_{X}\overset{}{\underset{ (e_0,S_0)} {\rightarrow } }(y)_{Y}$
\end{itemize}
if for every $x_1\overset{}{\underset{ (e_0,S_0)} {\in } }(x)_{X}$, $(f)_{F}:(x)_{X}\overset{}{\underset{ (e_0,S_0)} {\rightarrow } }(y)_{Y}$ is continuous at $x_1$ over \\ $(( (x)_{X},(d_1)_{D_1},(r_1)_{R_1} ),( (y)_{Y},(d_2)_{D_2},(r_2)_{R_2} ) )_{(e_0,S_0)}$ then we say that $(f)_{F}:(x)_{X}\overset{}{\underset{ (e_0,S_0)} {\rightarrow } }(y)_{Y}$ is {\bf {\itshape continuous over $(( (x)_{X},(d_1)_{D_1},(r_1)_{R_1} ),( (y)_{Y},(d_2)_{D_2},(r_2)_{R_2}))_{(e_0,S_0)}$}}.
\end{definition}

% Definition of uniform continuity
\begin{definition}
Given
\begin{itemize}
	\item metric spaces $( (x)_{X},(d_1)_{D_1},(r_1)_{R_1} )_{(e_0,S_0)}$ and $( (y)_{Y},(d_2)_{D_2},(r_2)_{R_2} )_{(e_0,S_0)}$
	\item function $(f)_{F}:(x)_{X}\overset{}{\underset{ (e_0,S_0)} {\rightarrow } }(y)_{Y}$
\end{itemize}
if for any real number $(\epsilon)_{E_1}\overset{\mathbb{R}}{\underset{ (e_0,S_0)} {> } }0$ over $(e_0,S_0)$ there exists a real number $(\delta)_{E_2}\overset{\mathbb{R}}{\underset{ (e_0,S_0)} {> } }0$ over $(e_0,S_0)$ such that for any $x_1\overset{}{\underset{ (e_0,S_0)} {\in } }(x)_{X}$ and any $x_2\overset{}{\underset{ (e_0,S_0)} {\in } }(x)_{X}$ we have
$$ ((d_1)_{D_1}( x_1,x_2 ))_{R_1}\overset{\mathbb{R}}{\underset{ (e_0,S_0)} {< } }(\delta)_{E_2} \Rightarrow ((d_2)_{D_2}( (f)_{F}(x_1),(f)_{F}(x_2)))_{R_2}\overset{\mathbb{R}}{\underset{ (e_0,S_0)} {< } }(\epsilon)_{E_1} $$
then we say that $(f)_{F}:(x)_{X}\overset{}{\underset{ (e_0,S_0)} {\rightarrow } }(y)_{Y}$ is {\bf {\itshape uniformly continuous over \\ $(( (x)_{X},(d_1)_{D_1},(r_1)_{R_1} ),( (y)_{Y},(d_2)_{D_2},(r_2)_{R_2}))_{(e_0,S_0)}$}}.
\end{definition}

% function persistent [12]

\begin{definition}
Given
\begin{itemize}
	\item metric spaces $((x_1)_{X_1},(d_1)_{D_1},(r_1)_{R_1})_{(e_0,S_0)}$ and $((y_1)_{Y_1},(d_2)_{D_2},(r_2)_{R_2})_{(e_0,S_0)}$
	\item function $(f)_{F}:(x_1)_{X_1}\overset{}{\underset{ (e_0,S_0)} {\rightarrow } } (y_1)_{Y_1}$
\end{itemize}
if for any relation extension
$( (t)_{T},((x_2)_{X_2},(d_3)_{D_3},(r_3)_{R_3}),((y_2)_{Y_2},(d_4)_{D_4},(r_4)_{R_4}) )_{(e_0,S_0)}$ of \\
$( (f)_{F},((x_1)_{X_1},(d_1)_{D_1},(r_1)_{R_1}),((y_1)_{Y_1},(d_2)_{D_2},(r_2)_{R_2}) )_{(e_0,S_0)}$ via some $( (g)_{G},(h)_{H})$ we have that 
$ (t)_{T}:(x_2)_{X_2}\overset{}{\underset{ (e_0,S_0)} {\rightarrow } }(y_2)_{Y_2}$ is a function then we say that 
$ (f)_{F}:(x_1)_{X_1}\overset{}{\underset{ (e_0,S_0)} {\rightarrow } }(y_1)_{Y_1}$ is {\bf {\itshape function persistent over  $( ((x_1)_{X_1},(d_1)_{D_1},(r_1)_{R_1}),( (y_1)_{Y_1},(d_2)_{D_2},(r_2)_{R_2} ) )_{(e_0,S_0)}$}}.
\end{definition}

% function extensible

\begin{definition}
Given
\begin{itemize}
	\item metric spaces $((x_1)_{X_1},(d_1)_{D_1},(r_1)_{R_1})_{(e_0,S_0)}$ and $((y_1)_{Y_1},(d_2)_{D_2},(r_2)_{R_2})_{(e_0,S_0)}$
	\item function $(f)_{F}:(x_1)_{X_1}\overset{}{\underset{ (e_0,S_0)} {\rightarrow } } (y_1)_{Y_1}$
\end{itemize}
if  
\begin{itemize}
	\item for any relation extension
$( (t)_{T},((x_2)_{X_2},(d_3)_{D_3},(r_3)_{R_3}),((y_2)_{Y_2},(d_4)_{D_4},(r_4)_{R_4}) )_{(e_0,S_0)}$ of \\
$( (f)_{F},((x_1)_{X_1},(d_1)_{D_1},(r_1)_{R_1}),((y_1)_{Y_1},(d_2)_{D_2},(r_2)_{R_2}) )_{(e_0,S_0)}$ via some $( (g)_{G},(h)_{H})$ we have that 
$ (t)_{T}:(x_2)_{X_2}\overset{}{\underset{ (e_0,S_0)} {\rightarrow } }(y_2)_{Y_2}$ is a continuous function
	\item for any extension $( (x_2)_{X_2},(d_3)_{D_3},(r_3)_{R_3} )_{(e_0,S_0)}$ of $((x_1)_{X_1},(d_1)_{D_1},(r_1)_{R_1})_{(e_0,S_0)}$ via some $(g)_{G}$ there exists a relation extension \\ $((t)_{T},((x_2)_{X_2},(d_3)_{D_3},(r_3)_{R_3}),((y_2)_{Y_2},(d_4)_{D_4},(r_4)_{R_4}))_{(e_0,S_0)}$ of \\ $((f)_{F},((x_1)_{X_1},(d_1)_{D_1},(r_1)_{R_1}),((y_1)_{Y_1},(d_2)_{D_2},(r_2)_{R_2}))_{(e_0,S_0)}$ via some $( (g)_{G},(h)_{H})$
\end{itemize}
then we say that
$$ (f)_{F}:(x_1)_{X_1}\overset{}{\underset{ (e_0,S_0)} {\rightarrow } } (y_1)_{Y_1}$$ is {\bf {\itshape extensible over $( ((x_1)_{X_1},(d_1)_{D_1},(r_1)_{R_1}),( (y_1)_{Y_1},(d_2)_{D_2},(r_2)_{R_2} ) )_{(e_0,S_0)}$}}.
\end{definition}

% image extensible [15]
\begin{definition}
Given
\begin{itemize}
	\item metric spaces $((x_1)_{X_1},(d_1)_{D_1},(r_1)_{R_1})_{(e_0,S_0)}$ and $((y_1)_{Y_1},(d_2)_{D_2},(r_2)_{R_2})_{(e_0,S_0)}$
	\item function $(f)_{F}:(x_1)_{X_1}\overset{}{\underset{ (e_0,S_0)} {\rightarrow } } (y_1)_{Y_1}$
\end{itemize}
if  
\begin{itemize}
	\item $(f)_{F}:(x_1)_{X_1} \overset{}{\underset{ (e_0,S_0)} {\rightarrow } }(y_1)_{Y_1}$ is function persistent over \\ $( ((x_1)_{X_1},(d_1)_{D_1},(r_1)_{R_1} ),( (y_1)_{Y_1},(d_2)_{D_2},(r_2)_{R_2} ) )_{(e_0,S_0)}$
	\item for any extension $( (y_2)_{Y_2},(d_3)_{D_3},(r_3)_{R_3})_{(e_0,S_0)}$ of \\ $( (f)_{F}((x_1)_{X_1}),(d_2)_{D_2}| (f)_{F}((x_1)_{X_1}), (r_2)_{R_2})_{(e_0,S_0)}$ via some $(g)_{G}$ there exists an extension \\ $( (x_2)_{X_2},(d_4)_{D_4},(r_4)_{R_4})_{(e_0,S_0)}$ of  $( (x_1)_{X_1},(d_1)_{D_1},(r_1)_{R_1})_{(e_0,S_0)}$ via some $(h)_{H}$ and a surjective function $(z)_{Z}:(x_2)_{X_2}\overset{}{\underset{ (e_0,S_0)} {\rightarrow } } (y_2)_{Y_2}$ such that \\ $( (z)_{Z},( (x_2)_{X_2},(d_4)_{D_4},(r_4)_{R_4}),( (y_2)_{Y_2},(d_3)_{D_3},(r_3)_{R_3}))_{(e_0,S_0)}$ is a relation extension of \\ $( (f)_{F}, ( (x_1)_{X_1},(d_1)_{D_1},(r_1)_{R_1}),( (y_1)_{Y_1},(d_2)_{D_2},(r_2)_{R_2}))_{(e_0,S_0)}$ via $( (h)_{H},(g)_{G})$
\end{itemize}
then we say that 
	$$ (f)_{F}:(x_1)_{X_1}\overset{}{\underset{ (e_0,S_0)} {\rightarrow } }(y_1)_{Y_1}$$ 
is {\bf {\itshape image extensible over $( ((x_1)_{X_1},(d_1)_{D_1},(r_1)_{R_1} ),( (y_1)_{Y_2},(d_2)_{D_2},(r_2)_{R_2} ) )_{(e_0,S_0)}$}}.

\end{definition}

% (4) Theorem: Uniform continuity implies extensibility

\begin{thm}
Given
\begin{itemize}
	\item metric spaces $( (x)_{X},(d_1)_{D_1},(r_1)_{R_1} )_{(e_0,S_0)}$ and $( (y)_{Y},(d_2)_{D_2},(r_2)_{R_2} )_{(e_0,S_0)}$
	\item function $(f)_{F}:(x)_{X}\overset{}{\underset{ (e_0,S_0)} {\rightarrow } }(y)_{Y}$
\end{itemize}
if $(f)_{F}:(x)_{X}\overset{}{\underset{ (e_0,S_0)} {\rightarrow } }(y)_{Y}$ is uniformly continuous over \\ $(( (x)_{X},(d_1)_{D_1},(r_1)_{R_1} ),( (y)_{Y},(d_2)_{D_2},(r_2)_{R_2} ))_{(e_0,S_0)}$ then it is also extensible over \\ $(( (x)_{X},(d_1)_{D_1},(r_1)_{R_1} ),( (y)_{Y},(d_2)_{D_2},(r_2)_{R_2} ))_{(e_0,S_0)}$.
\end{thm}

\begin{proof}
We see that continuity of extensions follows almost immediately from uniform continuity. Now let $( (x_1)_{X_1},(d_3)_{D_3},(r_3)_{R_3} )_{(e_0,S_0)}$ be an extension of  $( (x)_{X},(d_1)_{D_1},(r_1)_{R_1} )_{(e_0,S_0)}$ via some $(g)_{G}$. We can construct an extension  $( (y_1)_{Y_1},(d_4)_{D_4},(r_4)_{R_4} )_{(e_0,S_0)}$ of $( (y)_{Y},(d_2)_{D_2},(r_2)_{R_2} )_{(e_0,S_0)}$ via some $(h)_{H}$ and a function $(f_1)_{F_1}:(x_1)_{X_1}\overset{}{\underset{ (e_0,S_0)} {\rightarrow } }(y_1)_{Y_1}$ such that 

for any $a\overset{}{\underset{ (e_0,S_0)} {\in } }(x_1)_{X_1}$ and 
any function $(k)_{K}:\mathbb{N} \overset{}{\underset{ (e_0,S_0)} {\rightarrow } }(x)_{X}$ with \\
$((d_3)_{D_3}(a,(g)_{G}((k)_{K}(n))))_{R_3}\overset{\mathbb{R}}{\underset{ (e_0,S_0)} {< } }1/n$ for all $n$,  
we have that
$ (f_1)_{F_1}(a)$ is the unique limit point of  $(h)_{H}((f)_{F}((k)_{K}(\mathbb{N})))$ over $( (y_1)_{Y_1},(d_4)_{D_4},(r_4)_{R_4} )_{(e_0,S_0)}$ (by triangle inequality).

By uniform continuity of $(f)_{F}$ we see that $(f_1)_{F_1}:(x_1)_{X_1}\overset{}{\underset{ (e_0,S_0)} {\rightarrow } }(y_1)_{Y_1}$ and $(d_4)_{D_4}$ are well defined so we can say that $(f)_{F}:(x)_{X}\overset{}{\underset{ (e_0,S_0)} {\rightarrow } }(y)_{Y}$ is extensible over \\ $(( (x)_{X},(d_1)_{D_1},(r_1)_{R_1} ),( (y)_{Y},(d_2)_{D_2},(r_2)_{R_2} ))_{(e_0,S_0)}$.

\end{proof}

The next theorem can be seen as a reformulation of the Bolzano-Weierstrass theorem.

% (4.5)(43) Theorem: Total boundedness, extension, infinite set and limit point
\begin{thm} \label{bolzano}
Given
\begin{itemize}
	\item metric space $( (x)_{X},(d)_{D},(r)_{R} )_{(e_0,S_0)}$
	\item totally bounded set $(a)_{A}$ over $( (x)_{X},(d)_{D},(r)_{R} )_{(e_0,S_0)}$
\end{itemize}
if $(a)_{A}$ is infinite over $(e_0,S_0)$ then there exists an extension \\ $( (y)_{Y},(d_1)_{D_1},(r_1)_{R_1})_{(e_0,S_0)}$ of  $( (x)_{X},(d)_{D},(r)_{R} )_{(e_0,S_0)}$ via some $(f)_{F}$ and a point $v\overset{}{\underset{ (e_0,S_0)} {\in } }(y)_{Y}$ such that $v$ is a limit point of $(f)_{F}((a)_{A})$ over $( (y)_{Y},(d_1)_{D_1},(r_1)_{R_1} )_{(e_0,S_0)}$.
\end{thm}

\begin{proof}

	%%% 
Since $(a)_{A}$ is totally bounded and $(r)_{R}$ is connected over any (real induced) metric space it induces then we see that $\forall (\epsilon)_{E}\overset{\mathbb{R}}{\underset{ (e_0,S_0)} {> } }0$ there exists a finite cover of $(a)_{A}$ by open balls of radius in $(r)_{R}$ less than $(\epsilon)_{E}$ centered at points in $(x)_{X}$. Since there exists a set $(t)_{T}$ (over $(e_0,S_0)$) of all finite sets of open balls centered at points in $(x)_{X}$ of radius in $(r)_{R}$ then using the theorem of choice we can construct a function $(q)_{Q}:\mathbb{N}\overset{}{\underset{ (e_0,S_0)} {\rightarrow} }(t)_{T} $ such that $((q)_{Q}(n))_{T}$ is a finite cover of $(a)_{A}$ by open balls of radius less than $1/n$. Because $(a)_{A}$ is infinite we can use countability of sets, the well ordering of naturals and the theorem of choice to construct a function $(s)_{S}:\mathbb{N}\overset{}{\underset{ (e_0,S_0)} {\rightarrow } }(a)_{A}$ where for all $m$ there exists $k_m\overset{}{\underset{ (e_0,S_0)} {\in } }((q)_{Q}(m))_{T}$ such that
for all $n\geq m$, 
 $(s)_{S}(n)\overset{}{\underset{ (e_0,S_0)} {\in } }(k_m)_{T}$. So we see that $(s)_{S}$ is Cauchy-convergent. 
We can construct an extension $( (y)_{Y},(d_1)_{D_1},(r_1)_{R_1} )_{(e_0,S_0)}$ of $( (x)_{X},(d)_{D},(r)_{R} )_{(e_0,S_0)}$ via some $(f)_{F}$ where $(y_1)_{Y}\overset{}{\underset{ (e_0,S_0)} {= } }(f)_{F}( (x)_{X} ) \cup \{ v \}$ and where for any $z\overset{}{\underset{ (e_0,S_0)} {\in } }(x)_{X}$ and any $z_1\overset{}{\underset{ (e_0,S_0)} {= } }(f)_{F}(z)$ we have 
$$((d_1)_{D_1}(v,z_1))_{R_1}\overset{}{\underset{ (e_0,S_0)} {= } }{\underset{n\rightarrow +\infty}{\text{lim}}}((d)_{D}((s)_{S}(n),z))_{R}$$
So we see that $v$ is a limit point of $(f)_{F}((a)_{A})$ over $( (y)_{Y},(d_1)_{D_1},(r_1)_{R_1} )_{(e_0,S_0)}$.
	%%%

\end{proof}

In the following theorem we see that an extensible function on a totally bounded domain must be uniformly continuous.

% (5) Theorem: Extensibility + [domain] total boundedness implies uniform continuity
\begin{thm} \label{exttotaluni}
Given
\begin{itemize}
	\item function $(f)_{F}:(x)_{X}\overset{}{\underset{ (e_0,S_0)} {\rightarrow } }(y)_{Y}$
	\item metric spaces $( (x)_{X},(d_1)_{D_1},(r_1)_{R_1} )_{(e_0,S_0)}$ and $( (y)_{Y},(d_2)_{D_2},(r_2)_{R_2} )_{(e_0,S_0)}$
\end{itemize}
if $( (x)_{X},(d_1)_{D_1},(r_1)_{R_1} )_{(e_0,S_0)}$ is totally bounded and $(f)_{F}$ is extensible over \\ $(( (x)_{X},(d_1)_{D_1},(r_1)_{R_1} ),( (y)_{Y},(d_2)_{D_2},(r_2)_{R_2} ))_{(e_0,S_0)}$ then $(f)_{F}$ is uniformly continuous over \\ $(( (x)_{X},(d_1)_{D_1},(r_1)_{R_1} ),( (y)_{Y},(d_2)_{D_2},(r_2)_{R_2} ))_{(e_0,S_0)}$.

\end{thm}

\begin{proof}

%%%
If $(x)_{X}$ is finite the desired result follows immediately so we will just consider the case where $(x)_{X}$ is infinite.

Suppose $(f)_{F}$ not uniformly continuous.
There exists $(\epsilon)_{E_1}\overset{\mathbb{R}}{\underset{ (e_0,S_0)} {> } }0$ such that for all $(\delta)_{E_2}\overset{\mathbb{R}}{\underset{ (e_0,S_0)} {> } }0$ there exist $a$ and $b$ with $((d_1)_{D_1}(a,b))_{R_1}\overset{\mathbb{R}}{\underset{ (e_0,S_0)} {<} }(\delta)_{E_2}$ with  $((d_2)_{D_2}((f)_{F}(a),(f)_{F}(b)))_{R_2}\overset{\mathbb{R}}{\underset{ (e_0,S_0)} {\geq } }(\epsilon)_{E_1}$.
There exists a function $(s)_{S}:\mathbb{N}\overset{}{\underset{ (e_0,S_0)} {\rightarrow } }(x)_{X} \times (x)_{X}$ such that if $(s)_{S}(n)\overset{}{\underset{ (e_0,S_0)} {= } }(a_n,b_n)$ then $((d_1)_{D_1}(a_n,b_n))_{R_1}\overset{\mathbb{R}}{\underset{ (e_0,S_0)} {<} }1/n$ and  $((d_2)_{D_2}((f)_{F}(a_n),(f)_{F}(b_n)))_{R_2}\overset{\mathbb{R}}{\underset{ (e_0,S_0)} {\geq } }(\epsilon)_{E_1}$.  
Since $(x)_{X}$ is totally bounded then by theorem \ref{bolzano} the set $(l)_{L}\overset{}{\underset{ (e_0,S_0)} {= } }\{ a_n \mid (s)_{S}(n)\overset{}{\underset{ (e_0,S_0)} {= } }(a_n,b_n)$   \text{ for some } $b_n\overset{}{\underset{ (e_0,S_0)} {\in } }(x)_{X}\}$ has a limit point $t$ in some extension $((x_1)_{X_1},(d_3)_{D_3},(r_3)_{R_3} )_{(e_0,S_0)}$ of  $((x)_{X},(d_1)_{D_1},(r_1)_{R_1})_{(e_0,S_0)}$ via some $(g)_{G}$. So by extensibility of $(f)_{F}$ we obtain a relation extension \\ $( (f_1)_{F_1},( (x_1)_{X_1},(d_3)_{D_3},(r_3)_{R_3}),( (y_1)_{Y_1},(d_4)_{D_4},(r_4)_{R_4} ))_{(e_0,S_0)}$ of \\ $( (f)_{F},( (x)_{X},(d_1)_{D_1},(r_1)_{R_1}),( (y)_{Y},(d_2)_{D_2},(r_2)_{R_2} ))_{(e_0,S_0)}$ via some $( (g)_{G},(h)_{H} )$. Since extensions of $(f)_{F}$ are continuous, we see that there exists $(\delta)_{E_3}\overset{\mathbb{R}}{\underset{ (e_0,S_0)} {> } }0$ such that for any $c$ in a ball of radius $(\delta)_{E_3}$ centered at $t$ we have $((d_4)_{D_4}((f_1)_{F_1}(t),(f_1)_{F_1}(c)))_{R_4}\overset{\mathbb{R}}{\underset{ (e_0,S_0)} {< } }(\epsilon)_{E_1}/2$. Since $t$ is a limit point of $(l)_{L}$ there exist $c_1\overset{}{\underset{ (e_0,S_0)} {= } }(g)_{G}(k_1)$ and $c_2\overset{}{\underset{ (e_0,S_0)} {= } }(g)_{G}(k_2)$ with 
$k_1\overset{}{\underset{ (e_0,S_0)} {\in } }(l)_{L}$ and $(s)_{S}(n_1)\overset{}{\underset{ (e_0,S_0)} { =} }(k_1,k_2)$ for some $n_1$ such that $((d_3)_{D_3}(c_1,t))_{R_3}\overset{\mathbb{R}}{\underset{ (e_0,S_0)} {< } }(\delta)_{E_3}/2$ and $((d_3)_{D_3}(c_2,t))_{R_3}\overset{\mathbb{R}}{\underset{ (e_0,S_0)} {< } }(\delta)_{E_3}/2$ which implies that $((d_1)_{D_1}(k_1,k_2))_{R_1}\overset{}{\underset{ (e_0,S_0)} {< } }(\delta)_{E_3}$ and $((d_2)_{D_2}((f)_{F}(k_1),(f)_{F}(k_2)))_{R_2}\overset{\mathbb{R}}{\underset{ (e_0,S_0)} {< } }(\epsilon)_{E_1}$  which gives us a contradiction. So $(f)_{F}$ must be uniformly continuous.
%%%

\end{proof}

% (6) Theorem: Uniform continuity + [domain] connectedness implies [image] connectedness

\begin{thm} \label{uniconnected}
Given
\begin{itemize}
	\item metric spaces $( (x)_{X},(d_1)_{D_1},(r_1)_{R_1} )_{(e_0,S_0)}$ and $( (y)_{Y},(d_2)_{D_2},(r_2)_{R_2} )_{(e_0,S_0)}$
	\item function $(f)_{F}:(x)_{X}\overset{}{\underset{ (e_0,S_0)} {\rightarrow } }(y)_{Y}$
\end{itemize}
if $(f)_{F}$ is uniformly continuous over $(( (x)_{X},(d_1)_{D_2},(r_1)_{R_1} ),( (y)_{Y},(d_2)_{D_2},(r_2)_{R_2} ))_{(e_0,S_0)}$ and \\ $( (x)_{X},(d_1)_{D_1},(r_1)_{R_1} )_{(e_0,S_0)}$ is connected  then $(f)_{F}((x)_{X})$ is connected over $( (y)_{Y},(d_2)_{D_2},(r_2)_{R_2} )_{(e_0,S_0)}$.
\end{thm}

\begin{proof}
	
	%%%
Suppose $(f)_{F}( (x)_{X} )$ is not connected. Then there exist $(c_1)_{C_1}\overset{}{\underset{ (e_0,S_0)} {\neq } }\emptyset$ and $(c_2)_{C_2}\overset{}{\underset{ (e_0,S_0)} {\neq } }\emptyset$ such that $(c_1)_{C_1}\cup (c_2)_{C_2} \overset{}{\underset{ (e_0,S_0)} {= } }(f)_{F}( (x)_{X} )$, $(c_1)_{C_1}\cap (c_2)_{C_2} \overset{}{\underset{ (e_0,S_0)} {= } }\emptyset$ and such that the distance between $(c_1)_{C_1}$ and $(c_2)_{C_2}$ is equal to some real number $(s)_{S}\overset{\mathbb{R}}{\underset{ (e_0,S_0)} {> } }0$. 
Let $(p_1)_{P_1}\overset{}{\underset{ (e_0,S_0)} {= } }(f)_{F}^{-1}( (c_1)_{C_1} )$ and $(p_2)_{P_2}\overset{}{\underset{ (e_0,S_0)} {= } }(f)_{F}^{-1}( (c_2)_{C_2} )$.
Evidently $(p_1)_{P_1}\overset{}{\underset{ (e_0,S_0)} {\neq } }\emptyset$, $(p_2)_{P_2}\overset{}{\underset{ (e_0,S_0)} {\neq } }\emptyset$, $(p_1)_{P_1}\cap (p_2)_{P_2} \overset{}{\underset{ (e_0,S_0)} {= } }\emptyset$  and $(p_1)_{P_1}\cup (p_2)_{P_2}\overset{}{\underset{ (e_0,S_0)} {= } }(x)_{X}$.
By uniform continuity there exists a real number $(t)_{T}\overset{\mathbb{R}}{\underset{ (e_0,S_0)} {> } }0$ such that for any $a_1\overset{}{\underset{ (e_0,S_0)} {\in } }(x)_{X}$ and any $a_2\overset{}{\underset{ (e_0,S_0)} {\in } }(x)_{X}$  with $((d_1)_{D_1}(a_1,a_2))_{R_1}\overset{\mathbb{R}}{\underset{ (e_0,S_0)} {< } }(t)_{T} $ we have $((d_2)_{D_2}((f)_{F}(a_1),(f)_{F}(a_2)))_{R_2}\overset{\mathbb{R}}{\underset{ (e_0,S_0)} {< } }(s)_{S}$.
By connectedness of $(x)_{X}$ there exist $a_1\overset{}{\underset{ (e_0,S_0)} {\in } }(p_1)_{P_1}$ and $a_2\overset{}{\underset{ (e_0,S_0)} {\in } }(p_2)_{P_2}$ such that $((d_1)_{D_1}(a_1,a_2))_{R_1}\overset{\mathbb{R}}{\underset{ (e_0,S_0)} {< } }(t)_{T} $ so we get a contradiction and must conclude that $(f)_{F}( (x)_{X} )$ is connected.
	%%%
	
\end{proof}

% (7) Theorem: Uniform continuity + [domain] total boundedness implies [image] total boundedness

\begin{thm} \label{unitotal}
Given
\begin{itemize}
	\item metric spaces $( (x)_{X},(d_1)_{D_1},(r_1)_{R_1} )_{(e_0,S_0)}$ and $( (y)_{Y},(d_2)_{D_2},(r_2)_{R_2} )_{(e_0,S_0)}$
	\item function $(f)_{F}:(x)_{X}\overset{}{\underset{ (e_0,S_0)} {\rightarrow } }(y)_{Y}$
\end{itemize}
if $(f)_{F}$ is uniformly continuous over $(( (x)_{X},(d_1)_{D_1},(r_1)_{R_1} ),( (y)_{Y},(d_2)_{D_2},(r_2)_{R_2} ))_{(e_0,S_0)}$ and \\ $( (x)_{X},(d_1)_{D_1},(r_1)_{R_1} )_{(e_0,S_0)}$ is totally bounded then $(f)_{F}( (x)_{X} )$ is totally bounded over \\ $( (y)_{Y},(d_2)_{D_2},(r_2)_{R_2} )_{(e_0,S_0)}$.

\end{thm}

\begin{proof}

		%%%
Let $(\epsilon)_{E_1}\overset{\mathbb{R}}{\underset{ (e_0,S_0)} {> } }0$ be some real number over $(e_0,S_0)$. By uniform continuity there exists $(\delta)_{E_2}\overset{\mathbb{R}}{\underset{ (e_0,S_0)} {> } }0$ such that for any $x_1\overset{}{\underset{ (e_0,S_0)} {\in } }(x)_{X}$ and any $x_2\overset{}{\underset{ (e_0,S_0)} {\in } }(x)_{X}$ with $((d_1)_{D_1}(x_1,x_2))_{R_1}\overset{\mathbb{R}}{\underset{ (e_0,S_0)} {< } }(\delta)_{E_2}$ we have $((d_2)_{D_2}((f)_{F}(x_1),(f)_{F}(x_2)))_{R_2}\overset{\mathbb{R}}{\underset{ (e_0,S_0)} {< } }(\epsilon)_{E_1}$.
Since $(x)_{X}$ is totally bounded then there exists $(c_1)_{C_1}$ which covers $(x)_{X}$ by a finite amount of balls of radius less than $(\delta)_{E_1}$. Let $(c_2)_{C_2}$ be a set such that

\begin{itemize}
	\item $\forall t\overset{}{\underset{ (e_0,S_0)} {\in } }(c_1)_{C_1} \exists k\overset{}{\underset{ (e_0,S_0)} {\in } }(c_2)_{C_2} ( (k)_{C_2}\overset{}{\underset{ (e_0,S_0)} {= } }(f)_{F}( (t)_{C_1} ) )$
	\item $\forall k\overset{}{\underset{ (e_0,S_0)} {\in } }(c_2)_{C_2} \exists t\overset{}{\underset{ (e_0,S_0)} {\in } }(c_1)_{C_1} ( (k)_{C_2}\overset{}{\underset{ (e_0,S_0)} {= } }(f)_{F}( (t)_{C_1} ) )$
	\item $\forall k_1\overset{}{\underset{ (e_0,S_0)} {\in } }(c_2)_{C_2} \forall k_2\overset{}{\underset{ (e_0,S_0)} {\in } }(c_2)_{C_2} ( (k_1)_{C_2}\overset{}{\underset{ (e_0,S_0)} {= } }(k_2)_{C_2} \Rightarrow k_1=k_2)$
\end{itemize}

We see that $(c_2)_{C_2}$ covers $(f)_{F}( (x)_{X} )$ by a finite amount of balls of radius less than $(\epsilon)_{E_1}$. So $(f)_{F}( (x)_{X} )$ is totally bounded.  
		
		%%%

\end{proof}

% (9) Theorem: Function extensibility + [domain] local connectedness + [domain] local total-boundededness implies [image] connectedness

\begin{thm} \label{locallocal}
Given
\begin{itemize}
	\item metric spaces $( (x)_{X},(d_1)_{D_1},(r_1)_{R_1} )_{(e_0,S_)}$ and $( (y)_{Y},(d_2)_{D_2},(r_2)_{R_2} )_{(e_0,S_0)}$
	\item function $(f)_{F}:(x)_{X}\overset{}{\underset{ (e_0,S_0)} {\rightarrow } }(y)_{Y}$
\end{itemize}
if $( (x)_{X},(d_1)_{D_1},(r_1)_{R_1} )_{(e_0,S_)}$ is both locally connected and locally totally bounded, and \\ $(f)_{F}:(x)_{X}\overset{}{\underset{ (e_0,S_0)} {\rightarrow } }(y)_{Y}$ is extensible then $(f)_{F}((x)_{X})$ is connected over $( (y)_{Y},(d_2)_{D_2},(r_2)_{R_2} )_{(e_0,S_0)}$.

\end{thm}

\begin{proof}

			%%%
Let $(c_1)_{C_1}\overset{}{\underset{ (e_0,S_0)} {\neq } }\emptyset$ and $(c_2)_{C_2}\overset{}{\underset{ (e_0,S_0)} {\neq } }\emptyset$ be such that $(c_1)_{C_1}\cap (c_2)_{C_2} \overset{}{\underset{ (e_0,S_0)} {= } }\emptyset$ and $(c_1)_{C_1} \cup (c_2)_{C_2} \overset{}{\underset{ (e_0,S_0)} {= } }(f)_{F}((x)_{X})$.
Let $(p_1)_{P_1}\overset{}{\underset{ (e_0,S_0)} {= } }(f)_{F}^{-1}( (c_1)_{C_1} )$ and $(p_2)_{P_2}\overset{}{\underset{ (e_0,S_0)} {= } }(f)_{F}^{-1}( (c_2)_{C_2} )$.
We can choose $x_1\overset{}{\underset{ (e_0,S_0)} {\in } }(x)_{X}$ and real number $(s)_{S}\overset{\mathbb{R}}{\underset{ (e_0,S_0)} {> } }0$ such that the open ball $(b)_{B}$ of radius $(s)_{S}$ centered at $x_1$ intersects both $(p_1)_{P_1}$ and $(p_2)_{P_2}$.
Since $(x)_{X}$ is locally connected then $(b)_{B}$ is connected and since $(x)_{X}$ is locally totally bounded then $(b)_{B}$ is totally bounded.
Let $(g)_{G}:(b)_{B} \overset{}{\underset{ (e_0,S_0)} {\rightarrow } }(y)_{Y}$ be the function obtained by restricting $(f)_{F}$ to $(b)_{B}$. By theorem \ref{exttotaluni} it follows that $(g)_{G}$ is uniformly continuous and by theorem \ref{uniconnected} it follows that $(g)_{G}( (b)_{B} )$ is connected which implies that the distance between $(c_1)_{C_1}$ and $(c_2)_{C_2}$ is 0. So $(f)_{F}( (x)_{X} )$ must be connected.

			%%%

\end{proof}

The next theorem shows that an extensible function on a totally bounded domain must be image extensible.

% (10) Theorem: Extensibility + [domain] total boundedness implies image extensible
\begin{thm} \label{imgextensible}
Given
\begin{itemize}
	\item metric spaces $( (x)_{X},(d_1)_{D_1},(r_1)_{R_1} )_{(e_0,S_0)}$ and $( (y)_{Y},(d_2)_{D_2},(r_2)_{R_2} )_{(e_0,S_0)}$
	\item function $(f)_{F}:(x)_{X}\overset{}{\underset{ (e_0,S_0)} {\rightarrow } }(y)_{Y}$
\end{itemize}
if $( (x)_{X},(d_1)_{D_1},(r_1)_{R_1} )_{(e_0,S_0)}$ is totally bounded and $(f)_{F}:(x)_{X}\overset{}{\underset{ (e_0,S_0)} {\rightarrow } }(y)_{Y}$ is extensible over $(( (x)_{X},(d_1)_{D_1},(r_1)_{R_1} ),( (y)_{Y},(d_2)_{D_2},(r_2)_{R_2} ))_{(e_0,S_0)}$ then $(f)_{F}:(x)_{X}\overset{}{\underset{ (e_0,S_0)} {\rightarrow } }(y)_{Y}$ is image extensible over $(( (x)_{X},(d_1)_{D_1},(r_1)_{R_1} ),( (y)_{Y},(d_2)_{D_2},(r_2)_{R_2} ))_{(e_0,S_0)}$.
\end{thm}

\begin{proof}

		%%% 
Let $( (y_1)_{Y_1},(d_3)_{D_3},(r_3)_{R_3} )_{(e_0,S_0)}$ be an extension of \\ $( (f)_{F}((y)_{Y}),(d_2)_{D_2}|(f)_{F}( (y)_{Y} ),(r_2)_{R_2} )_{(e_0,S_0)}$ via some $(g)_{G}$ and let $(p)_{P}$ be such that $(y_1)_{Y_1}\overset{}{\underset{ (e_0,S_0)} {= } }(f)_{F}( (y)_{Y} ) \cup (p)_{P}$. 

Since $(x)_{X}$ is totally bounded then there exists (by choice and countability of words) a function $(q)_{Q}:(p)_{P}\overset{}{\underset{ (e_0,S_0)} {\rightarrow } }(s)_{S}$ such that for every $b\overset{}{\underset{ (e_0,S_0)} {\in } }(p)_{P}$, $((q)_{Q}(b))_{S}:\mathbb{N}\overset{}{\underset{ (e_0,S_0)} {\rightarrow } }(x)_{X}$ is Cauchy-convergent and $(g)_{G} ( (f)_{F} ( ( (q)_{Q}(b))_{S} ) ) : \mathbb{N} \overset{}{\underset{ (e_0,S_0)} {\rightarrow } }(y_1)_{Y_1}$ converges to $b$.

Since $(f)_{F}$ is extensible then it is also function persistent so any of its relation extensions must be a function. So we can construct a relation extension \\ $((f_1)_{F_1}, ((x_1)_{X_1},(d_4)_{D_4},(r_4)_{R_4} ),((y_1)_{Y_1},(d_3)_{D_3},(r_3)_{R_3} ))_{(e_0,S_0)}$ of \\ $( (f)_{F},((x)_{X},(d_1)_{D_1},(r_1)_{R_1} ),((f)_{F}((y)_{Y}),(d_2)_{D_2}|(f)_{F}( (y)_{Y} ),(r_2)_{R_2}))_{(e_0,S_0)}$ via some $((h)_{H},(g)_{G})$
\\ with $(x_1)_{X_1} \overset{}{\underset{ (e_0,S_0)} {= } }(h)_{H}( (x)_{X} ) \cup (t)_{T}$ where
\begin{itemize}
	\item $(f_1)_{F_1}|(t)_{T}:(t)_{T}\overset{}{\underset{ (e_0,S_0)} {\rightarrow } }(p)_{P}$ is a bijective function
	\item $\forall a\overset{}{\underset{ (e_0,S_0)} {\in } }(x)_{X} \forall b\overset{}{\underset{ (e_0,S_0)} {\in } }(p)_{P} ( ((d_4)_{D_4}( (h)_{H}(a), ((f_1)_{F_1}|(t)_{T})^{-1}(b) ))_{R_4} \overset{\mathbb{R}}{\underset{ (e_0,S_0)} {= } }$ \\ ${\underset{n\rightarrow +\infty}{\text{lim}}} ((d_1)_{D_1}( a,((q)_{Q}(b))_{S}(n) ) )_{R_1} )$
	\item $\forall b_1\overset{}{\underset{ (e_0,S_0)} {\in } }(p)_{P} \forall b_2\overset{}{\underset{ (e_0,S_0)} {\in } }(p)_{P} ( ((d_4)_{D_4}( ((f_1)_{F_1}|(t)_{T})^{-1}(b_1), ((f_1)_{F_1}|(t)_{T})^{-1}(b_2) ))_{R_4} \overset{\mathbb{R}}{\underset{ (e_0,S_0)} {= } }$ \\ $  {\underset{n\rightarrow +\infty}{\text{lim}}} ((d_1)_{D_1}( ((q)_{Q}(b_1))_{S}(n),((q)_{Q}(b_2))_{S}(n) ))_{R_1} )$
\end{itemize} 
We conclude that $(f)_{F}:(x)_{X}\overset{}{\underset{ (e_0,S_0)} {\rightarrow } }(y)_{Y}$ is image extensible over \\ $(( (x)_{X},(d_1)_{D_1},(r_1)_{R_1} ),( (y)_{Y},(d_2)_{D_2},(r_2)_{R_2} ))_{(e_0,S_0)}$.

		%%%

\end{proof}

% (11) (Intermediate value theorem)
%  Corollary: (for real numbers) Function extensibility + [domain] total-boundedness + [domain] connectedness => if a and b are in the image and there exists c such that a<c<b then there exists an extension of the function such that c is in the image

\begin{cor}(Intermediate value theorem)
Given
\begin{itemize}
	\item metric space $( (x)_{X},(d_1)_{D_1},(r_1)_{R_1} )_{(e_0,S_0)}$
	\item real-induced metric space $( (r_2)_{R_2},(d_2)_{D_2},(r_2)_{R_2} )_{(e_0,S_0)}$
	\item real function $(f)_{F}:(x)_{X}\overset{}{\underset{ (e_0,S_0)} {\rightarrow } }(r_2)_{R_2}$
\end{itemize}
if
\begin{itemize}
	\item $( (x)_{X},(d_1)_{D_2},(r_1)_{R_1} )_{(e_0,S_0)}$ is totally bounded
	\item $( (x)_{X},(d_1)_{D_2},(r_1)_{R_1} )_{(e_0,S_0)}$ is connected
	\item $(f)_{F}:(x)_{X}\overset{}{\underset{ (e_0,S_0)} {\rightarrow } }(r_2)_{R_2}$ is extensible over $(( (x)_{X},(d_1)_{D_1},(r_1)_{R_1} ),( (r_2)_{R_2},(d_2)_{D_2},(r_2)_{R_2} ))_{(e_0,S_0)}$
\end{itemize}
then if $(c)_{C}$ is a real number over $(e_0,S_0)$ such that there exist $a$ and $b$ in the image of $(f)_{F}$ with $(a)_{R_2}\overset{\mathbb{R}}{\underset{ (e_0,S_0)} {< } }(c)_{C}\overset{\mathbb{R}}{\underset{ (e_0,S_0)} {< } }(b)_{R_2}$ then there exist $(c_1)_{R_3}\overset{\mathbb{R}}{\underset{ (e_0,S_0)} {= } }(c)_{C}$ and a relation extension $( (f_1)_{F_1},( (x_1)_{X_1},(d_4)_{D_4},(r_4)_{R_4}),( (r_3)_{R_3},(d_3)_{D_3},(r_3)_{R_3} ))_{(e_0,S_0)}$ of \\ $((f)_{F},( (x)_{X},(d_1)_{D_1},(r_1)_{R_1} ),( (f)_{F}( (x)_{X} ),(d_2)_{D_2}|(f)_{F}( (x)_{X} ),(r_2)_{R_2} ))_{(e_0,S_0)}$ via some $( (g)_{G},(h)_{H} )$ \\ with $c_1\overset{}{\underset{ (e_0,S_0)} {\in } }(f_1)_{F_1}( (x_1)_{X_1} )$ and where $( (r_3)_{R_3},(d_3)_{D_3},(r_3)_{R_3} )_{(e_0,S_0)}$ is a real-induced metric space.
\end{cor}

\begin{proof}

		%%%
By theorem \ref{exttotaluni} and by theorem \ref{uniconnected} we see that $(f)_{F}( (x)_{X})$ is connected. Let $(c)_{C}$ be a real number over $(e_0,S_0)$ such that there exist $a$ and $b$ in the image of $(f)_{F}$ with $(a)_{R_2}\overset{\mathbb{R}}{\underset{ (e_0,S_0)} {< } }(c)_{C}\overset{\mathbb{R}}{\underset{ (e_0,S_0)} {< } }(b)_{R_2}$. Let $(s_1)_{S_1}\overset{}{\underset{ (e_0,S_0)} { =} }\{ e\overset{}{\underset{ (e_0,S_0)} {\in } }(f)_{F}( (x)_{X} )  \mid (e)_{R_2}\overset{\mathbb{R}}{\underset{ (e_0,S_0)} {< } }(c)_{C} \}$ and let $(s_2)_{S_2}\overset{}{\underset{ (e_0,S_0)} { =} }\{ e\overset{}{\underset{ (e_0,S_0)} {\in } }(f)_{F}( (x)_{X} )  \mid (e)_{R_2}\overset{\mathbb{R}}{\underset{ (e_0,S_0)} {\geq } }(c)_{C} \}$. Since $(f)_{F} ( (x)_{X} )$ is connected then the distance between $(s_1)_{S_1}$ and $(s_2)_{S_2}$ is 0, so we can construct an extension
 $((r_3)_{R_3},(d_3)_{D_3},(r_3)_{R_3} )_{(e_0,S_0)}$ of $( (f)_{F}( (x)_{X} ),(d_2)_{D_2}|(f)_{F}( (x)_{X} ),(r_2)_{R_2} )_{(e_0,S_0)}$ via some $(h)_{H}$ such that there exists $c_1\overset{}{\underset{ (e_0,S_0)} {\in } }(r_3)_{R_3}$ with $(c_1)_{R_3}\overset{\mathbb{R}}{\underset{ (e_0,S_0)} {= } }(c)_{C}$.
Using theorem \ref{imgextensible} we obtain a relation extension \\ $( (f_1)_{F_1},( (x_1)_{X_1},(d_4)_{D_4},(r_4)_{R_4}),( (r_3)_{R_3},(d_3)_{D_3},(r_3)_{R_3} ))_{(e_0,S_0)}$ of \\ $((f)_{F},( (x_1)_{X_1},(d_1)_{D_1},(r_1)_{R_1} ),( (f)_{F}( (x)_{X} ),(d_2)_{D_2}|(f)_{F}( (x)_{X} ),(r_2)_{R_2} ))_{(e_0,S_0)}$ via some $( (g)_{G},(h)_{H} )$ with $c_1\overset{}{\underset{ (e_0,S_0)} {\in } }(f_1)_{F_1}( (x_1)_{X_1} )$.
		
		%%%

\end{proof}

% (12) (Extreme value theorem)
% Corrolary: (for real numbers) Function extensibility + [domain] total boundedness => there exists an extension of the function for which the image contain its supremum
 
 \begin{cor} (Extreme value theorem)
 Given
 \begin{itemize}
 	\item metric space $( (x)_{X},(d_1)_{D_1},(r_1)_{R_1} )_{(e_0,S_0)}$
 	\item real-induced metric space $( (r_2)_{R_2},(d_2)_{D_2},(r_2)_{R_2} )_{(e_0,S_0)}$
 	\item real function $(f)_{F}:(x)_{X}\overset{}{\underset{ (e_0,S_0)} {\rightarrow } }(r_2)_{R_2}$
 \end{itemize}
 if $( (x)_{X},(d_1)_{D_1},(r_1)_{R_1} )_{(e_0,S_0)}$ is totally bounded and $(f)_{F}$ is extensible over \\ $(( (x)_{X},(d_1)_{D_1},(r_1)_{R_1} ),( (r_2)_{R_2},(d_2)_{D_2},(r_2)_{R_2}))_{(e_0,S_0)}$ then there exists a relation extension \\ $((f_1)_{F_1},((x_1)_{X_1},(d_4)_{D_4},(r_4)_{R_4}),((r_3)_{R_3},(d_3)_{D_3},(r_3)_{R_3}) )_{(e_0,S_0)}$ of \\ $((f)_{F},( (x)_{X},(d_1)_{D_1},(r_1)_{R_1} ),( (r_2)_{R_2},(d_2)_{D_2},(r_2)_{R_2} ) )_{(e_0,S_0)}$ via some $( (g)_{G},(h)_{H} )$ such that there exists $s\overset{}{\underset{ (e_0,S_0)} {\in } }(f_1)_{F_1}( (x_1)_{X_1} )$ with $(s)_{Y_1}\overset{\mathbb{R}}{\underset{ (e_0,S_0)} {= } }sup (f)_{F}( (x)_{X} )$ where $( (r_3)_{R_3},(d_3)_{D_3},(r_3)_{R_3} )_{(e_0,S_0)}$ is a real-induced metric space.  
 \end{cor}

 \begin{proof}
 
 			%%%
 			From extensibility and total boundedness we see by theorem \ref{exttotaluni} that $(f)_{F}$ is uniformly continuous. From uniform continuity and total boundedness we see by theorem \ref{unitotal} that $(f)_{F}((x)_{X})$ is totally bounded so $(f)_{F}( (x)_{X} )$ is bounded so $sup (f)_{F}( (x)_{X} )$ exists. Using theorem \ref{imgextensible} we obtain the desired result.
 			%%% 

 \end{proof}

\section{Measure and integration}

We take the view that if two sets have the same closure (in a given metric space) then they convey the ``same information'' in a measure-theoretic setting. Since completeness of certain metric spaces is unachievable we see why that view is suitable. Now we adapt that view to define ``metric rings'' (of sets) and reformulate the notion of measure.

\begin{definition}
Let $((x)_{X},(d)_{D},(r)_{R})_{(e_0,S_0)}$ be a metric space and $(g)_{G}$ a set of subsets of $(x)_{X}$ over $(e_0,S_0)$. If
\begin{itemize}
	\item $\forall e\overset{}{\underset{ (e_0,S_0)} {\in } }(g)_{G} \forall f\overset{}{\underset{ (e_0,S_0)} {\in } }(g)_{G} ( (e)_{G}\overset{}{\underset{ (e_0,S_0)} {= } }(f)_{G} \Rightarrow e=f )$
	\item $\forall e \forall f ( (e\overset{}{\underset{ (e_0,S_0)} {\in } }(g)_{G} \wedge f\overset{}{\underset{ (e_0,S_0)} {\in } }(g)_{G} ) \Rightarrow \exists h\overset{}{\underset{ (e_0,S_0)} {\in } }(g)_{G}( (h)_{G} \overset{}{\underset{ (e_0,S_0)} {= } } (e)_{G}\setminus (f)_{G}))$
	\item $\forall e \forall f ( (e\overset{}{\underset{ (e_0,S_0)} {\in } }(g)_{G} \wedge f\overset{}{\underset{ (e_0,S_0)} {\in } }(g)_{G} ) \Rightarrow \exists h\overset{}{\underset{ (e_0,S_0)} {\in } }(g)_{G}( (h)_{G} \overset{}{\underset{ (e_0,S_0)} {= } }(e)_{G}\cup (f)_{G}))$
	\item $\forall e(e\overset{}{\underset{ (e_0,S_0)} {\in } }(g)_{G}\Rightarrow \exists h\overset{}{\underset{ (e_0,S_0)} {\in } }(g)_{G} ( (h)_{G} \overset{}{\underset{ (e_0,S_0)} {= } } [(e)_{G}]))$
\end{itemize}
then we say that $( (g)_{G},(x)_{X},(d)_{D},(r)_{R})_{(e_0,S_0)}$ is a {\bf {\itshape metric ring}}.
\end{definition}

\begin{definition}
Given
\begin{itemize}
	\item metric ring $((g)_{G},(x)_{X},(d)_{D},(r_1)_{R_1})_{(e_0,S_0)}$
	\item real function $(m)_{M}:(g)_{G}\overset{}{\underset{ (e_0,S_0)} { \rightarrow} }(r_2)_{R_2}$
\end{itemize}
if
\begin{itemize}
	\item $((m)_{M}(a))_{R_2}\overset{\mathbb{R}}{\underset{ (e_0,S_0)} {\geq } }0$ for any $a\overset{}{\underset{ (e_0,S_0)} {\in } }(g)_{G}$
	\item $((m)_{M}([(a_1)_{G}]\cup [(a_2)_{G}]))_{R_2}+((m)_{M}([(a_1)_{G}]\cap [(a_2)_{G}]))_{R_2}\overset{\mathbb{R}}{\underset{ (e_0,S_0)} {= } }$ \\ $ ((m)_{M}([(a_1)_{G}]))_{R_2}+((m)_{M}([(a_2)_{G}]))_{R_2}$ for any $a_1\overset{}{\underset{ (e_0,S_0)} {\in } }(g)_{G}$ and any $a_2\overset{}{\underset{ (e_0,S_0)} {\in } }(g)_{G}$
	\item $((m)_{M}([(a_1)_{G}\setminus (a_2)_{G}]\cap [(a_2)_{G}]))_{R_2}\overset{\mathbb{R}}{\underset{ (e_0,S_0)} {= } }0$ for any $a_1\overset{}{\underset{ (e_0,S_0)} {\in } }(g)_{G}$ and any $a_2\overset{}{\underset{ (e_0,S_0)} {\in } }(g)_{G}$
	\item $((m)_{M}(a))_{R_2}\overset{\mathbb{R}}{\underset{ (e_0,S_0)}{=} }((m)_{M}([(a)_{G}]))_{R_2}$ for any $a\overset{}{\underset{ (e_0,S_0)} {\in } }(g)_{G}$
\end{itemize}
then we say that $( (m)_{M},(g)_{G},(r_2)_{R_2},(x)_{X},(d)_{D},(r_1)_{R_1} )_{e_0,S_0}$ is a {\bf {\itshape measure space}}.
\end{definition}

Now we define measurable sets.
\begin{definition}
Let $( (m)_{M},(g)_{G},(r_2)_{R_2},(x)_{X},(d)_{D},(r_1)_{R_1} )_{(e_0,S_0)}$ be a measure space and $(x_1)_{X_1}$ a subset of $(x)_{X}$ over $(e_0,S_0)$. If there exists a real number $(r_3)_{R_3}$ over $(e_0,S_0)$ such that for all $(\epsilon)_{E}\overset{\mathbb{R}}{\underset{ (e_0,S_0)} {> } }0$ there exist $a_1\overset{}{\underset{ (e_0,S_0)} {\in } }(g)_{G}$ and $a_2\overset{}{\underset{ (e_0,S_0)} {\in } }(g)_{G}$ with $[(a_1)_{G}]\overset{}{\underset{ (e_0,S_0)} {\subseteq } }[(x_1)_{X_1}]\overset{}{\underset{ (e_0,S_0)} {\subseteq } }[(a_2)_{G}]$ such that  $((m)_{M}( a_2 ))_{R_2}\overset{\mathbb{R}}{\underset{ (e_0,S_0)} {\leq } }((m)_{M}( a_1 ))_{R_2} + (\epsilon)_{E} $ and  
$((m)_{M}( a_1 ))_{R_2}\overset{\mathbb{R}}{\underset{ (e_0,S_0)} {\leq } }(r_3)_{R_3} \overset{\mathbb{R}}{\underset{ (e_0,S_0)} {\leq } } ((m)_{M}( a_2 ))_{R_2}$  
 then we say $(x_1)_{X_1}$ is {\bf {\itshape measurable over  $( (m)_{M},(g)_{G},(r_2)_{R_2},(x)_{X},(d)_{D},(r_1)_{R_1} )_{(e_0,S_0)}$}} and $(r_3)_{R_3}$ is the {\bf {\itshape measure of $(x_1)_{X_1}$ over  $( (m)_{M},(g)_{G},(r_2)_{R_2},(x)_{X},(d)_{D},(r_1)_{R_1} )_{(e_0,S_0)}$ }}.
\end{definition}

We see that the closure of a measurable set is measurable and the union of two measurable sets is also measurable. On the other hand, the set difference of two measurable sets is not necessarily measurable. For example it is possible to choose a measure space containing only finite unions of intervals and construct a Smith-Volterra-Cantor type of set which is non-measurable but is the set difference of two measurable sets. So it is not possible to ensure that ``adding'' measurable sets to a measure space will still induce a measure space.
 
%\begin{definition}
%Let $( (m)_{M},(g)_{G},(r_2)_{R_2},(x)_{X},(d)_{D},(r_1)_{R_1} )_{(e_0,S_0)}$ be a measure space. If there exists $g_1\overset{}{\underset{ (e_0,S_0)} {\in } }(g)_{G}$ such that %$(g_1)_{G}\overset{}{\underset{ (e_0,S_0)} {= } }(x)_{X}$ and $((m)_{M}( g_1 ))_{R_2}\overset{\mathbb{R}}{\underset{ (e_0,S_0)} {= } }1$ then we say that  $( %(m)_{M},(g)_{G},(r_2)_{R_2},(x)_{X},(d)_{D},(r_1)_{R_1} )_{(e_0,S_0)}$ is a {\bf {\itshape probability space}}.
%\end{definition}

The next definitions will be used to define integrability and state a simple integrability theorem.

\begin{definition}
Let $((m)_{M},(g)_{G},(r_2)_{R_2},(x)_{X},(d)_{D},(r_1)_{R_1})_{(e_0,S_0)}$ be a measure space. If
\begin{itemize}
	\item there exists a real-sufficient set $(r)_{R}$ over $(e_0,S_0)$ such that for any $x_1 \overset{}{\underset{ (e_0,S_0)} {\in } }(x)_{X}$ and any real $\epsilon \overset{}{\underset{ (e_0,S_0)} {\in } }(r)_{R}$ with $(\epsilon)_{R} \overset{\mathbb{R}}{\underset{ (e_0,S_0)} {> } }0$ we have  $\beta(x_1,(\epsilon)_{R})\overset{}{\underset{ (e_0,S_0)} {\in } }(g)_{G}$  (recall that $\beta(x_1,(\epsilon)_{R})$ is an open ball of radius $(\epsilon)_{R}$ centered at $x_1$)
	\item for any real number $(\epsilon)_{R_3}\overset{\mathbb{R}}{\underset{ (e_0,S_0)} {> } }0$ there exists a real number $(\delta)_{R_4}\overset{\mathbb{R}}{\underset{ (e_0,S_0)} {> } }0$ such that for any $x_1 \overset{}{\underset{ (e_0,S_0)} {\in } }(x)_{X}$ we have $((m)_{M}(\beta(x_1,(\delta)_{R_4})))_{R_2}\overset{\mathbb{R}}{\underset{ (e_0,S_0)} {< } }(\epsilon)_{R_3}$
\end{itemize}
then we say that $((m)_{M},(g)_{G},(r_2)_{R_2},(x)_{X},(d)_{D},(r_1)_{R_1})_{(e_0,S_0)}$ is a {\bf {\itshape standard measure space}}.
\end{definition}

% Definition of partition (AM 27) 
\begin{definition}
Given
\begin{itemize}
	\item measure space $((m)_{M},(g)_{G},(r_2)_{R_2},(x)_{X},(d)_{D},(r_1)_{R_1})_{(e_0,S_0)}$
	\item subset $(a)_{A}$ of $(x)_{X}$ over $(e_0,S_0)$
	\item finite set $(p)_{P}$ over $(e_0,S_0)$
\end{itemize}
if
\begin{itemize}
	\item $\forall x_1 \overset{}{\underset{ (e_0,S_0)} {\in } }(p)_{P}(x_1 \overset{}{\underset{ (e_0,S_0)} {\in } }(g)_{G})$
	\item $[\underset{x_1\overset{}{\underset{ (e_0,S_0)} {\in } }(p)_{P}}{\bigcup}(x_1)_{G}]\overset{}{\underset{ (e_0,S_0)} {= } }[(a)_{A}]$
	\item $\forall x_1\overset{}{\underset{ (e_0,S_0)} {\in } }(p)_{P} \forall x_2\overset{}{\underset{ (e_0,S_0)} {\in } }(p)_{P}(x_1\neq x_2 \Rightarrow ((m)_{M}([x_1]\cap[x_2]))_{R_2}\overset{\mathbb{R}}{\underset{ (e_0,S_0)} {= } }0)$
\end{itemize}
then we say that $(p)_{P}$ is a {\bf {\itshape partition of $(a)_{A}$  over   $((m)_{M},(g)_{G},(r_2)_{R_2},(x)_{X},(d)_{D},(r_1)_{R_1})_{(e_0,S_0)}$}}.
\end{definition}

Note that if $(p)_{P}$ is a partition of $(a)_{A}$ over $((m)_{M},(g)_{G},(r_2)_{R_2},(x)_{X},(d)_{D},(r_1)_{R_1})_{(e_0,S_0)}$ then we see that $\sum_{p_i\overset{}{\underset{ (e_0,S_0)} {\in } }(p)_{P}} ( (m)_{M}( p_i ) )_{R_2} \overset{\mathbb{R}}{\underset{ (e_0,S_0)} {= } } ( (m)_{M} (a_1) )_{R_2}$ where $[(a_1)_{G}]\overset{}{\underset{ (e_0,S_0)} {= } }[(a)_{A}]$.

\begin{definition}
Given
\begin{itemize} 
	\item real function $(f)_{F}:(a)_{A}\overset{}{\underset{ (e_0,S_0)} { \rightarrow} }(r_1)_{R_1}$
	\item measure space $((m)_{M},(g)_{G},(r_3)_{R_3},(x)_{X},(d)_{D},(r_2)_{R_2})_{(e_0,S_0)}$
\end{itemize}
if for any real number $(\epsilon)_{E}\overset{\mathbb{R}}{\underset{ (e_0,S_0)} {> } }0$ there exists a partition $(p)_{P}\overset{}{\underset{ (e_0,S_0)} {=} }\{p_1,p_2,\ldots ,p_n\}$ of $(a)_{A}$ over $((m)_{M},(g)_{G},(r_3)_{R_3},(x)_{X},(d)_{D},(r_2)_{R_2})_{(e_0,S_0)}$ and real numbers \\  $(u)_{U},(u_1)_{U},(u_2)_{U},\dots ,(u_n)_{U}, (l)_{L},(l_1)_{L},(l_2)_{L}, \ldots ,(l_n)_{L}$ all over $(e_0,S_0)$  such that  
$$(u_i)_{U}\overset{\mathbb{R}}{\underset{ (e_0,S_0)} {= } }\underset{b \underset{(e_0,S_0)}{\in}(p_i)_{G}}{sup}((f)_{F}(b))_{R_1} \text{ and } (l_i)_{L}\overset{\mathbb{R}}{\underset{ (e_0,S_0)} {= } }\underset{b \underset{(e_0,S_0)}{\in}(p_i)_{G}}{inf}((f)_{F}(b))_{R_1}$$

$$(u)_{U}\overset{\mathbb{R}}{\underset{ (e_0,S_0)} {= } }\sum_{i=1}^{n}{(u_i)_{U}((m)_{M}(p_i))_{R_3}} \text{ and } (l)_{L}\overset{\mathbb{R}}{\underset{ (e_0,S_0)} {= } }\sum_{i=1}^{n}{(l_i)_{L}((m)_{M}(p_i))_{R_3}}$$

$$(u)_{U}-(l)_{L}\overset{\mathbb{R}}{\underset{ (e_0,S_0)} {< } }(\epsilon)_{E}$$
	
then we say that $(f)_{F}:(a)_{A}\overset{}{\underset{ (e_0,S_0)} { \rightarrow} }(r_1)_{R_1}$ is {\bf {\itshape simply integrable over }}\\ $((m)_{M},(g)_{G},(r_3)_{R_3},(x)_{X},(d)_{D},(r_2)_{R_2})_{(e_0,S_0)}$.
\end{definition}

% TO ADD DEFINITION OF EXTENSIBLE FOR REAL FUNCTIONS

% Integration theorem (AM 29 to AM 39)
\begin{thm}
Given
\begin{itemize} 
	\item measure space $((m)_{M},(g)_{G},(r_3)_{R_3},(x)_{X},(d)_{D},(r_2)_{R_2})_{(e_0,S_0)}$
	\item $(a)_{A}\overset{}{\underset{ (e_0,S_0)} {\subseteq } }(x)_{X}$
	\item real-induced metric space $( (r_1)_{R_1},(d_1)_{D_1},(r_1)_{R_1} )_{(e_0,S_0)}$
	\item real function $(f)_{F}:(a)_{A}\overset{}{\underset{ (e_0,S_0)} {\rightarrow} }(r_1)_{R_1}$
\end{itemize}
if
\begin{itemize}
	\item $(f)_{F}:(a)_{A}\overset{}{\underset{ (e_0,S_0)} { \rightarrow} }(r_1)_{R_1}$ is extensible over \\ $( ( (a)_{A},(d)_{D}|(a)_{A},(r_2)_{R_2} ),( (r_1)_{R_1},(d_1)_{D_1},(r_1)_{R_1} ))_{(e_0,S_0)}$
	\item $((m)_{M},(g)_{G},(r_3)_{R_3},(x)_{X},(d)_{D},(r_2)_{R_2})_{(e_0,S_0)}$ is a standard measure space
	\item there exists $a_1\overset{}{\underset{ (e_0,S_0)} {\in } }(g)_{G}$ such that $(a_1)_{G}\overset{}{\underset{ (e_0,S_0)} {= } }[(a)_{A}]$
	\item $(a)_{A}$ is totally bounded over $((x)_{X},(d)_{D},(r_2)_{R_2})_{(e_0,S_0)}$
\end{itemize}
then $(f)_{F}:(a)_{A}\overset{}{\underset{ (e_0,S_0)} { \rightarrow} }(r_1)_{R_1}$ is simply integrable over \\ $((m)_{M},(g)_{G},(r_3)_{R_3},(x)_{X},(d)_{D},(r_2)_{R_2})_{(e_0,S_0)}$.
\end{thm}
\begin{proof}

% expression of theorem (AM 35)
Let $a_1\overset{}{\underset{ (e_0,S_0)} {\in } }(g)_{G}$ such that $(a_1)_{G}\overset{}{\underset{ (e_0,S_0)} {= } }[(a)_{A}]$. Now lets consider an arbitrary real number   
$(\epsilon)_{E}\overset{\mathbb{R}}{\underset{ (e_0,S_0)} { >} }0$. By total boundedness of $(a)_{A}$ and extensibility we get that $(f)_{F}$ is uniformly continuous (using theorem \ref{exttotaluni}). Let $(\delta)_{E_1}\overset{\mathbb{R}}{\underset{ (e_0,S_0)} {> } }0$ be such that 

$\forall x_1 \forall x_2( ((d)_{D}(x_1,x_2))_{R_2}\overset{\mathbb{R}}{\underset{ (e_0,S_0)} {< } }(\delta)_{E_1} \Rightarrow  |((f)_{F}(x_1))_{R_1}-((f)_{F}(x_2))_{R_1}|\overset{\mathbb{R}}{\underset{ (e_0,S_0)} {< } }(\epsilon)_{E}/((m)_{M}(a_1))_{R_3} )$ 

By total boundedness there exists a finite cover of $(a)_{A}$ by balls of radius less than $(\delta)_{E_1}/2$. Since $((m)_{M},(r_3)_{R_3},(g)_{G},(x)_{X},(d)_{D},(r_2)_{R_2})_{(e_0,S_0)}$ is standard then by connectedness of real sets (and the triangle inequality) there exists a finite cover $(t)_{T}\overset{}{\underset{ (e_0,S_0)} { =} }\{ t_1,t_2,\dots ,t_n \}$ of $(a)_{A}$ such that, for $i=1,2,\dots , n$, $(t_i)_{G}$ is in $(g)_{G}$ and is a ball of radius less than or equal to $(\delta)_{E_1}$.
Let $(b)_{B} \overset{}{\underset{(e_0,S_0) } {= } }\{b_1,b_2,\dots , b_n\}$ such that  $(b_i)_{G}\overset{}{\underset{ (e_0,S_0)} {= } }(t_i)_{G}\cap (a_1)_{G}$ for $i=1,2,\dots , n$.

Let 

$ (k)_{K}\overset{}{\underset{ (e_0,S_0)} {= } }\{ (\overset{}{\underset{i{\underset{(e_0,S_0) } {\in } }(c)_{C} } {\bigcap } }(b_i)_{G})\setminus \overset{}{\underset{j{\underset{(e_0,S_0) } {\in } }\{1,2,\dots ,n\}\setminus (c)_{C} } {\bigcup } }(b_j)_{G} \mid (c)_{C}{\underset{(e_0,S_0) } {\subseteq} }\{1,2, \dots , n\} \text{ and } \\ (c)_{C}{\underset{(e_0,S_0) } {\neq } }\emptyset \}$

We now verify that $(k)_{K}$ is a partition of $(a_1)_{G}$.

% AM 35 - AM 36
1)\underline{\textit{Measure of intersections is zero}}:

Let $ (c_1)_{C_1}{\underset{(e_0,S_0) } {\subseteq} }\{1,2, \dots , n\} $ and $ (c_2)_{C_2}{\underset{(e_0,S_0) } {\subseteq} }\{1,2, \dots , n\} $ with $(c_1)_{C_1}\overset{}{\underset{ (e_0,S_0)} { \neq} }\emptyset$, $(c_2)_{C_2}\overset{}{\underset{ (e_0,S_0)} { \neq} }\emptyset$ and $(c_1)_{C_1}\overset{}{\underset{ (e_0,S_0)} {\neq } }(c_2)_{C_2}$.

Either there exists $q$ such that $q\overset{}{\underset{ (e_0,S_0)} {\in } }(c_1)_{C_1}$ and $q\overset{}{\underset{ (e_0,S_0)} {\notin } }(c_2)_{C_2}$ or there exists $q$ such that $q\overset{}{\underset{ (e_0,S_0)} {\in } }(c_2)_{C_2}$ and $q\overset{}{\underset{ (e_0,S_0)} {\notin } }(c_1)_{C_1}$.

If there exists $q$ such that $q\overset{}{\underset{ (e_0,S_0)} {\in } }(c_1)_{C_1}$ and $q\overset{}{\underset{ (e_0,S_0)} {\notin } }(c_2)_{C_2}$ then $q\overset{}{\underset{ (e_0,S_0)} {\in } }\{1,2, \dots , n \}\setminus(c_2)_{C_2}$. So

$$ (\overset{}{\underset{i{\underset{(e_0,S_0) } {\in } }(c_1)_{C_1} } {\bigcap } }(b_i)_{G})\setminus \overset{}{\underset{j{\underset{(e_0,S_0) } {\in } }\{1,2,\dots ,n\}\setminus (c_1)_{C_1} } {\bigcup } }(b_j)_{G} \overset{}{\underset{ (e_0,S_0)} {\subseteq } } (b_q)_{G} \overset{}{\underset{ (e_0,S_0)} {\subseteq } } 
\overset{}{\underset{j{\underset{(e_0,S_0) } {\in } }\{1,2,\dots ,n\}\setminus (c_2)_{C_2} } {\bigcup } }(b_j)_{G} $$  

which implies 
$$[(\overset{}{\underset{i{\underset{(e_0,S_0) } {\in } }(c_2)_{C_2} } {\bigcap } }(b_i)_{G})\setminus \overset{}{\underset{j{\underset{(e_0,S_0) } {\in } }\{1,2,\dots ,n\}\setminus (c_2)_{C_2} } {\bigcup } }(b_j)_{G} ]\cap 
[(\overset{}{\underset{i{\underset{(e_0,S_0) } {\in } }(c_1)_{C_1} } {\bigcap } }(b_i)_{G})\setminus \overset{}{\underset{j{\underset{(e_0,S_0) } {\in } }\{1,2,\dots ,n\}\setminus (c_1)_{C_1} } {\bigcup } }(b_j)_{G} ] $$
$$
\overset{}{\underset{ (e_0,S_0)} {\subseteq } }
[(\overset{}{\underset{i{\underset{(e_0,S_0) } {\in } }(c_2)_{C_2} } {\bigcap } }b_i)\setminus \overset{}{\underset{j{\underset{(e_0,S_0) } {\in } }\{1,2,\dots ,n\}\setminus (c_2)_{C_2} } {\bigcup } }(b_j)_{G} ]\cap
[ \overset{}{\underset{j{\underset{(e_0,S_0) } {\in } }\{1,2,\dots ,n\}\setminus (c_2)_{C_2} } {\bigcup } }(b_j)_{G} ]
 $$
 
 so 
 
 $$ ((m)_{M}([(\overset{}{\underset{i{\underset{(e_0,S_0) } {\in } }(c_2)_{C_2} } {\bigcap } }(b_i)_{G})\setminus \overset{}{\underset{j{\underset{(e_0,S_0) } {\in } }\{1,2,\dots ,n\}\setminus (c_2)_{C_2} } {\bigcup } }(b_j)_{G} ]\cap$$  $$
[(\overset{}{\underset{i{\underset{(e_0,S_0) } {\in } }(c_1)_{C_1} } {\bigcap } }(b_i)_{G})\setminus \overset{}{\underset{j{\underset{(e_0,S_0) } {\in } }\{1,2,\dots ,n\}\setminus (c_1)_{C_1} } {\bigcup } }(b_j)_{G} ] ))_{R_3} \overset{\mathbb{R}}{\underset{ (e_0,S_0)} {\leq } }  $$ 
$$ ((m)_{M}( [(\overset{}{\underset{i{\underset{(e_0,S_0) } {\in } }(c_2)_{C_2} } {\bigcap } }(b_i)_{G})\setminus \overset{}{\underset{j{\underset{(e_0,S_0) } {\in } }\{1,2,\dots ,n\}\setminus (c_2)_{C_2} } {\bigcup } }(b_j)_{G} ]\cap
[ \overset{}{\underset{j{\underset{(e_0,S_0) } {\in } }\{1,2,\dots ,n\}\setminus (c_2)_{C_2} } {\bigcup } }(b_j)_{G} ] ))_{R_3}\overset{\mathbb{R}}{\underset{ (e_0,S_0)} {= } }0 $$

It follows that 

$$ ((m)_{M}([(\overset{}{\underset{i{\underset{(e_0,S_0) } {\in } }(c_2)_{C_2} } {\bigcap } }(b_i)_{G})\setminus \overset{}{\underset{j{\underset{(e_0,S_0) } {\in } }\{1,2,\dots ,n\}\setminus (c_2)_{C_2} } {\bigcup } }(b_j)_{G} ]\cap$$ 
$$ 
[(\overset{}{\underset{i{\underset{(e_0,S_0) } {\in } }(c_1)_{C_1} } {\bigcap } }(b_i)_{G})\setminus \overset{}{\underset{j{\underset{(e_0,S_0) } {\in } }\{1,2,\dots ,n\}\setminus (c_1)_{C_1} } {\bigcup } }(b_j)_{G} ]))_{R_3} \overset{\mathbb{R}}{\underset{ (e_0,S_0)} {= } } 0  $$

By the same logic, if there exists $q$ such that $q\overset{}{\underset{ (e_0,S_0)} {\in } }(c_2)_{C_2}$ and $q\overset{}{\underset{ (e_0,S_0)} {\notin } }(c_1)_{C_1}$ we also obtain  $$ ((m)_{M}([(\overset{}{\underset{i{\underset{(e_0,S_0) } {\in } }(c_2)_{C_2} } {\bigcap } } (b_i)_{G})\setminus \overset{}{\underset{j{\underset{(e_0,S_0) } {\in } }\{1,2,\dots ,n\}\setminus (c_2)_{C_2} } {\bigcup } }(b_j)_{G} ]\cap$$
$$
[(\overset{}{\underset{i{\underset{(e_0,S_0) } {\in } }(c_1)_{C_1} } {\bigcap } }(b_i)_{G})\setminus \overset{}{\underset{j{\underset{(e_0,S_0) } {\in } }\{1,2,\dots ,n\}\setminus (c_1)_{C_1} } {\bigcup } }(b_j)_{G} ] ))_{R_3}  \overset{\mathbb{R}}{\underset{ (e_0,S_0)} {= } } 0  $$

% AM 37 - AM 38
2)\underline{{\textit Closure of union of partition elements is $[(a)_{A}]$}}:

We first prove that 

%% Replace e_i by b_i 

$$\overset{}{\underset{ i\overset{}{\underset{ (e_0,S_0)} {\in } }\{1,2,\dots ,n\} } {\bigcup } }(b_i)_{G}  \overset{}{\underset{ (e_0,S_0)} {\subseteq } }\overset{}{\underset{j\overset{}{\underset{ (e_0,S_0)} {\in } }(k)_{K} } {\bigcup } }(j)_{K} $$

It will suffice to show that for any $(c)_{C}\overset{}{\underset{ (e_0,S_0)} {\subseteq } }\{1,2,\dots ,n \}$ we have 

$$ \overset{}{\underset{i\overset{}{\underset{ (e_0,S_0)} {\in } }(c)_{C} } {\bigcap } }(b_i)_{G} \overset{}{\underset{ (e_0,S_0)} {\subseteq } } \overset{}{\underset{j\overset{}{\underset{ (e_0,S_0)} {\in } }(k)_{K} } {\bigcup } }(j)_{K}$$

Proceed by induction on $n-card((c)_{C})+1$ where $(c)_{C}\overset{}{\underset{ (e_0,S_0)} {\subseteq } }\{1,2,\dots ,n  \}$.

If $card( (c)_{C} )=n$ then we easily see that $(c)_{C}\overset{}{\underset{ (e_0,S_0)} {= } }\{1,2,\dots ,n \}$ and 

$$ \overset{}{\underset{i\overset{}{\underset{ (e_0,S_0)} {\in } }\{1,2,\dots ,n  \} } {\bigcap } }(b_i)_{G} \overset{}{\underset{ (e_0,S_0)} {\subseteq } } \overset{}{\underset{j\overset{}{\underset{ (e_0,S_0)} {\in } }(k)_{K} } {\bigcup } }(j)_{K} $$

Suppose that for some $m$, if $card( (c)_{C} )=m$ and $m>1$ we have  

$$ \overset{}{\underset{i\overset{}{\underset{ (e_0,S_0)} {\in } }(c)_{C} } {\bigcap } }(b_i)_{G} \overset{}{\underset{ (e_0,S_0)} {\subseteq } } \overset{}{\underset{j\overset{}{\underset{ (e_0,S_0)} {\in } }(k)_{K} } {\bigcup } }(j)_{K}$$

Consider a set $(c_1)_{C_1}\overset{}{\underset{ (e_0,S_0)} {\subseteq } }\{ 1,2,\dots ,n \}$ with $card( (c_1)_{C_1} )=m-1$. We have

$$( (\overset{}{\underset{i{\underset{(e_0,S_0) } {\in } }(c_1)_{C_1} } {\bigcap } }(b_i)_{G})\setminus \overset{}{\underset{j{\underset{(e_0,S_0) } {\in } }\{1,2,\dots ,n\}\setminus (c_1)_{C_1} } {\bigcup } }(b_j)_{G} ) \overset{}{\underset{ (e_0,S_0)} {\subseteq } } 
\overset{}{\underset{j\overset{}{\underset{ (e_0,S_0)} {\in } }(k)_{K} } {\bigcup } }(j)_{K}	$$

By hypothesis it follows that for all $l\overset{}{\underset{ (e_0,S_0)} {\in } }\{1,2,\dots ,n  \}\setminus (c_1)_{C_1}$ we have 

$$ ( (b_l)_{G} \cap  \overset{}{\underset{i{\underset{(e_0,S_0) } {\in } }(c_1)_{C_1} } {\bigcap } }(b_i)_{G} 
 ) \overset{}{\underset{ (e_0,S_0)} {\subseteq } }  \overset{}{\underset{j\overset{}{\underset{ (e_0,S_0)} {\in } }(k)_{K} } {\bigcup } }(j)_{K}
                                    $$
                                    
so

$$ \overset{}{\underset{ l\overset{}{\underset{ (e_0,S_0)} {\in } }\{1,2,\dots ,n  \}\setminus (c_1)_{C_1} } {\bigcup } } 
 ((b_l)_{G} \cap  \overset{}{\underset{i{\underset{(e_0,S_0) } {\in } }(c_1)_{C_1} } {\bigcap } }(b_i)_{G} )               
\overset{}{\underset{ (e_0,S_0)} {\subseteq } } \overset{}{\underset{j\overset{}{\underset{ (e_0,S_0)} {\in } }(k)_{K} } {\bigcup } }(j)_{K}    $$

Since 
$$ \overset{}{\underset{ l\overset{}{\underset{ (e_0,S_0)} {\in } }\{1,2,\dots ,n  \}\setminus (c_1)_{C_1} } {\bigcup } } 
 ((b_l)_{G} \cap  \overset{}{\underset{i{\underset{(e_0,S_0) } {\in } }(c_1)_{C_1} } {\bigcap } }(b_i)_{G} ) 
\overset{}{\underset{ (e_0,S_0)} {= } }$$ 
$$
( \overset{}{\underset{i{\underset{(e_0,S_0) } {\in } }(c_1)_{C_1} } {\bigcap } }(b_i)_{G} ) \cap
( \overset{}{\underset{l{\underset{(e_0,S_0) } {\in } }\{1,2,\dots ,n  \}\setminus (c_1)_{C_1} } {\bigcup } }(b_l)_{G} ) 		$$

then we get 
$$
( (\overset{}{\underset{i{\underset{(e_0,S_0) } {\in } }(c_1)_{C_1} } {\bigcap } }(b_i)_{G})\setminus \overset{}{\underset{j{\underset{(e_0,S_0) } {\in } }\{1,2,\dots ,n\}\setminus (c_1)_{C_1} } {\bigcup } }(b_j)_{G} )$$  $$ 
\cup ( ( \overset{}{\underset{i{\underset{(e_0,S_0) } {\in } }(c_1)_{C_1} } {\bigcap } }(b_i)_{G} ) \cap
( \overset{}{\underset{i{\underset{(e_0,S_0) } {\in } }\{1,2,\dots ,n  \}\setminus (c_1)_{C_1} } {\bigcup } }(b_i)_{G} ) ) $$
$$	\overset{}{\underset{ (e_0,S_0)} {\subseteq } }
\overset{}{\underset{j\overset{}{\underset{ (e_0,S_0)} {\in } }(k)_{K} } {\bigcup } }(j)_{K}
$$

which implies 

$$
\overset{}{\underset{i{\underset{(e_0,S_0) } {\in } }(c_1)_{C_1} } {\bigcap } }(b_i)_{G}
\overset{}{\underset{ (e_0,S_0)} {\subseteq } }
\overset{}{\underset{j\overset{}{\underset{ (e_0,S_0)} {\in } }(k)_{K} } {\bigcup } }(j)_{K}$$

So by induction we see that $(b_i)_{G}\overset{}{\underset{ (e_0,S_0)} {\subseteq } }\overset{}{\underset{j\overset{}{\underset{ (e_0,S_0)} {\in } }(k)_{K} } {\bigcup } }(j)_{K}$ for any $i\overset{}{\underset{ (e_0,S_0)} {\in } }\{1,2,\dots ,n  \}$ and obtain 
$$\overset{}{\underset{ i\overset{}{\underset{ (e_0,S_0)} {\in } }\{1,2,\dots ,n\} } {\bigcup } }(b_i)_{G}  \overset{}{\underset{ (e_0,S_0)} {\subseteq } }\overset{}{\underset{j\overset{}{\underset{ (e_0,S_0)} {\in } }(k)_{K} } {\bigcup } }(j)_{K} $$

Since
$$\overset{}{\underset{j\overset{}{\underset{ (e_0,S_0)} {\in } }(k)_{K} } {\bigcup } }(j)_{K} 
\overset{}{\underset{ (e_0,S_0)} {\subseteq } }
\overset{}{\underset{ i\overset{}{\underset{ (e_0,S_0)} {\in } }\{1,2,\dots ,n\} } {\bigcup } }(b_i)_{G}$$

we obtain 

$$\overset{}{\underset{ i\overset{}{\underset{ (e_0,S_0)} {\in } }\{1,2,\dots ,n\} } {\bigcup } }(b_i)_{G}  \overset{}{\underset{ (e_0,S_0)} {= } }\overset{}{\underset{j\overset{}{\underset{ (e_0,S_0)} {\in } }(k)_{K} } {\bigcup } }(j)_{K} $$

We conclude that

$$ [ (a)_{A} ] \overset{}{\underset{ (e_0,S_0)} {= } }[ \overset{}{\underset{ i \overset{}{\underset{ (e_0,S_0)} {\in } }  \{1,2,\dots ,n\} } {\bigcup } }(b_i)_{G}] \overset{}{\underset{ (e_0,S_0)} {= } } [ \overset{}{\underset{j\overset{}{\underset{ (e_0,S_0)} {\in } }(k)_{K} } {\bigcup } }(j)_{K} ] $$

Let $(h_1)_{H_1}:(k)_{K}\overset{}{\underset{ (e_0,S_0)} {\rightarrow } }(r_4)_{R_4}$ be a real function
such that

$\forall k_i\overset{}{\underset{ (e_0,S_0)} {\in } }(k)_{K} ( ((h_1)_{H_1}(k_i))_{R_4}\overset{\mathbb{R}}{\underset{ (e_0,S_0)} {= } }sup((f)_{F}( (k_i)_{K} )))$.

and let $(h_2)_{H_2}:(k)_{K}\overset{}{\underset{ (e_0,S_0)} {\rightarrow } }(r_5)_{R_5}$ be a real function
such that

$\forall k_i\overset{}{\underset{ (e_0,S_0)} {\in } }(k)_{K} ( ((h_2)_{H_2}(k_i))_{R_5}\overset{\mathbb{R}}{\underset{ (e_0,S_0)} {= } }inf((f)_{F}( (k_i)_{K} )))$.

Let 
$$(u)_{U}\overset{\mathbb{R}}{\underset{ (e_0,S_0)} {= } } {\underset{k_i\overset{}{\underset{ (e_0,S_0)} {\in } }(k)_{K}}{\sum}}((h_1)_{H_1}(k_i))_{R_4}((m)_{M}(k_i))_{R_3}$$

$$(l)_{L}\overset{\mathbb{R}}{\underset{ (e_0,S_0)} {= } } {\underset{k_i\overset{}{\underset{ (e_0,S_0)} {\in } }(k)_{K}}{\sum}}((h_2)_{H_2}(k_i))_{R_5}((m)_{M}(k_i))_{R_3}$$

We get
$$(u)_{U}-(l)_{L}\overset{\mathbb{R}}{\underset{ (e_0,S_0)} {= } } {\underset{k_i\overset{}{\underset{ (e_0,S_0)} {\in } }(k)_{K}}{\sum}}( ( (h_1)_{H_1}(k_i) )_{R_4}- ( (h_2)_{H_2}(k_i) )_{R_5} )({m}_{M}(k_i))_{R_3} $$

$$\overset{\mathbb{R}}{\underset{ (e_0,S_0)} {\leq } }{\underset{k_i\overset{}{\underset{ (e_0,S_0)} {\in } }(k)_{K}}{\sum}}((\epsilon)_{E}/((m)_{M}(a_1))_{R_3})({m}_{M}(k_i))_{R_3} \overset{\mathbb{R}}{\underset{ (e_0,S_0)} {\leq } }(\epsilon)_{E}$$

so $(f)_{F}:(a)_{A}\overset{}{\underset{ (e_0,S_0)} { \rightarrow} }(r_1)_{R_1}$ is simply integrable over \\  $((m)_{M},(g)_{G},(r_3)_{R_3},(x)_{X},(d)_{D},(r_2)_{R_2})_{(e_0,S_0)}$ 

\end{proof}

Further generalizations will be considered in subsequent papers.

\bibliography{bibli}{}
\bibliographystyle{plain}
%\begin{thebibliography}{1}

%\bibliographystyle{rsl}

%\bibitem{cohen1}Cohen, Paul J. (December 15, 1963). ``The Independence of the Continuum Hypothesis''. Proceedings of the National Academy of Sciences of the United States of America \textbf{50}(6): 1143-1148.

%\bibitem{cohen2}Cohen, Paul J. (January 15, 1964). ``The Independence of the Continuum Hypothesis II''. Proceedings of the National Academy of Sciences of the United Stated of America \textbf{51}(1): 105-110.

%\bibitem{godel}G\"odel, K. (1940). ``The Consistency of the Continuum-Hypothesis''. Princeton University Press. 

%\bibitem{skolem}Skolem, Thoralf (1922). ``Some remarks on axiomatized set theory''. Reprinted in \textit{From Frege to G\"odel: A Source Book in Mathematical Logic}, van Heijnoort, 1967, in English translation by Stefan Bauer-Mengelberg pp.291-301

%\bibitem{zermelo}Zermelo, Ernst (1908). ``Untersuchungen \"uber die Grundlagen der Mengenlehre I''. Mathematische Annalen \textbf{65}: 261-281. English translation in Heijenoort, Jean van (1967). ``Investigations in the foundations of set theory''. \textit{From Frege to G\"odel: A Source Book in Mathematical Logic, 1879-1931}. Source Books in the History of the Sciences. Harvard University Press. pp.199-215.

%\end{thebibliography}

\end{document}